\documentclass[11pt,a4paper,oneside]{article}

\usepackage[utf8]{inputenc}
\usepackage[english]{babel}
\usepackage{url}
\usepackage{booktabs}
\usepackage{amsmath,amsfonts,amsthm,amssymb}
\usepackage{latexsym,enumitem,xargs}
\usepackage{xcolor,bbm}
\usepackage{appendix}
\usepackage{authblk}
\usepackage{mathrsfs}
\usepackage{algorithmic}
\usepackage{subcaption}
\usepackage{tikz}
\usetikzlibrary{matrix}
\usepackage{mathtools} 
\usepackage{chngcntr} 

\usepackage[round,comma]{natbib}
\usepackage{hypernat}
\usepackage[pagebackref=false,colorlinks=true,urlcolor=blue,citecolor=black!70!blue,linkcolor=black!70!green,plainpages=false,pdfpagelabels=true]{hyperref}

\numberwithin{equation}{section}
\theoremstyle{plain}
\newtheorem{thm}{Theorem}[section]
\newtheorem{dfn}{Definition}[section]
\newtheorem{prop}{Proposition}[section]
\newtheorem{cor}{Corrollary}[section]
\newtheorem{lem}{Lemma}[section]
\newtheorem{meth}{Method}%[section]

\newcommand{\E}{\mathbb{E}}
\def\1{\mbox{1\hspace{-.35em}1}}
\def\R{\mathbb{R}}
\def\N{\mathbb{N}}
\def\P{\mathbb{P}}
\newcommand{\mH}{\mathcal{H}}
\newcommand{\mR}{\mathcal{R}}

\newcommand{\Var}{\mbox{Var}}

\newcommand{\Sidak}{\v{S}id\'ak}
\newcommand{\MaxT}{MaxT}
\newcommand{\BootRW}{BootRW}

\renewcommand{\hat}{\widehat}
\renewcommand{\restriction}[2]{\left.#1\right\lvert_{#2}}

\makeatletter
\newcommand{\subalign}[1]{%
  \vcenter{%
    \Let@ \restore@math@cr \default@tag
    \baselineskip\fontdimen10 \scriptfont\tw@
    \advance\baselineskip\fontdimen12 \scriptfont\tw@
    \lineskip\thr@@\fontdimen8 \scriptfont\thr@@
    \lineskiplimit\lineskip
    \ialign{\hfil$\m@th\scriptstyle##$&$\m@th\scriptstyle{}##$\hfil\crcr
      #1\crcr
    }%
  }%
}
\makeatother

\graphicspath{{.}{./figures/}{./code/figures/}}

\title{Asymptotic control of FWER under Gaussian assumption: application to correlation tests}

\author[a]{Sophie Achard}
\author[b]{Pierre Borgnat}
\author[c]{Ir\`ene Gannaz}
\affil[a]{Univ. Grenoble Alpes, CNRS, Inria, Grenoble INP, LJK, 38000 Grenoble, France}
\affil[b]{Univ Lyon, ENS de Lyon, UCB Lyon 1, CNRS, Laboratoire de Physique, F-69342 Lyon, France}
\affil[c]{Univ Lyon, INSA de Lyon, CNRS UMR 5208, Institut Camille Jordan, F-69621 Villeurbanne, France}

\begin{document}

\mathtoolsset{showonlyrefs} % number only cited equations
\counterwithout{equation}{section} % number equation without section

\maketitle

\setlength{\parindent}{0pt}
\setlength{\parskip}{\baselineskip}

\begin{abstract}
In many applications, hypothesis testing is based on an asymptotic distribution of statistics. The aim of this paper is to clarify and extend multiple correction procedures when the statistics are asymptotically Gaussian. We propose a unified framework to prove their asymptotic behavior which is valid in the case of highly correlated tests. We focus on correlation tests where several test statistics are proposed. All these multiple testing procedures on correlations are shown to control FWER. An extensive simulation study on correlation-based graph estimation highlights finite sample behavior, independence on the sparsity of graphs and dependence on the values of correlations. Empirical evaluation of power provides comparisons of the proposed methods. Finally validation of our procedures is proposed on real dataset of rats brain connectivity measured by fMRI.  We confirm our theoretical findings by applying our procedures on a full null hypotheses with data from dead rats. Data on alive rats show the performance of the proposed procedures to correctly identify brain connectivity graphs with controlled errors.  
\end{abstract}


Consider a family of probability distributions $\mathcal P$.
Let $\mathbb X^{(n)}=\left(X_1,\dots,X_n \right)$ be independent realizations from an unknown probability distribution $P$. We assume that $P$ belongs to the family $\mathcal P$. Denote
$
\theta (P) = \big(\theta_1(P),\dots,\theta_m(P)\big)$
the parameter vector of interest, with $m\geq 2$.
Observing $\mathbb X^{(n)}$, we aim at testing the following two-sided null hypotheses, for all $i \in \{1,\hdots,m\}$
\begin{equation}
\label{eqn:tests} \tag{P-1}
H_{0,i}: \:  \theta_i = 0\text{~ against ~ } H_{1,i}: \:  \theta_i \neq 0.
\end{equation}
For all $i \in \{1,\dots,m\}$, we consider test statistics $T_{n,i} \left (\mathbb X^{(n)} \right)$ {chosen according to \eqref{eqn:tests}.} 

The objective of multiple testing procedure is to give a rejection set $$\mathcal R = \{ i,\;  1 \leq i \leq m : (H_{0,i})\text{~rejected}\},$$ such that the error is controlled. We will consider here the type I error called Family Wise Error Rate (FWER), defined as $$\text{FWER}(\mathcal R, P) = \mathbb P\left(\exists i\in\mathcal R  :  \theta_{i}=0 \right).$$ To control the FWER for a given level $\alpha \in [0, 1]$, the objective is to find a procedure yielding a rejection set $\mathcal R$ such that $\text{FWER}(\mathcal R, P)\leq\alpha$.

The classical method to control the FWER is the \cite{bonferroni1935calcolo}'s method  where each individual hypothesis is tested at a significance level of $\alpha /m$ with $\alpha$ the desired overall FWER level and $m$ the total number of hypotheses to test. 
Although this method is very intuitive, it could be conservative if there is a large number of false hypotheses relative to the number of hypotheses being tested. Alternative methods have been proposed to improve the power, {\it e.g.}  \cite{holm1979simple},  \cite{dudoit2003multiple}, ponderated Bonferroni in \cite{finner2009controlling}. \cite{goeman2010sequential} proposed a general framework to describe most of these methods by using the sequential rejection principle.

This manuscript investigates the problem of FWER control when the test statistics $T_{n,i} \left (\mathbb X^{(n)} \right)$ are possibly dependent and have an asymptotic Gaussian distribution. This work is motivated by an application in neuroscience \citep{achard.2006.1}. Let $V$ denote the set of the indexes of brain regions, $V=\{1,\dots,p\}$. {Using brain imagery facilities, it is possible to record non invasively the activity of each brain regions.} The data are {then processed to give } estimations of a correlation matrix between the activity of brain regions $(\rho_{i,j})_{(i,j)\in V\times V}$. The objective is to infer the dependence graph $G=(V, E)$ where the set of edges $E$ is a subset of $V \times V$ defined by $E=\{ (i,j)  : \rho_{i,j}\neq 0\}.$
To estimate these edges, one has to consider for all $(i,j)\in V\times V$ the tests of the form 
\begin{equation}
\label{letest_intro} \tag{P-2}
H_{0,i j} : \rho_{i j}=0  \text{~ against ~}  H_{1,i j} :  \rho_{i j}\neq 0 % \quad (i,j)\in V\times V, 
\end{equation} 
where $\rho_{i j}$ denotes a correlation between a variable $i$ and a variable $j$, corresponding to nodes {$V$}. In this setting, test statistics $T_{n,ij}\left( \mathbb X^{(n)} \right)$ are asymptotically Gaussian and possibly highly correlated. 
We hence propose to apply a multiple testing procedure on correlations. The set of edges is then estimated as $\{(i,j)\in V\times V\; | \; H_{0,i j} \text{~is rejected}\}$. For a given graph $G$, estimation $\hat G$ is obtained applying {$m=p\,(p-1)/2$} tests (since the set $E$ is symmetric by symmetry of the correlations -- the edges are undirected). Such an approach for graph inference has been proposed in \cite{drton2007multiple}. Other methods to estimate dependence graphs exist based on regularisation estimators \citep{Friedman2008, Meinshausen2006}. However, these methods need sparse assumptions on the graphs to be valid. In addition, to our knowledge, none of these latter methods ensures the control of FWER \citep{Rothman, Kramer, Cai}.  As an example, simulation studies state that Graphical Lasso approach select too many edges, see {\it e.g.} \cite{Kramer}. In applications such as network inference in neuroscience, a fine control of false discovered edges is crucial and this motivates the present work.

In this article, four existing asymptotic FWER controlling procedures are presented in an unified framework: \cite{bonferroni1935calcolo},  \cite{dudoit2003multiple}, \cite{romano2005exact},  \cite{drton2007multiple}. For each of them, we confirm that they control asymptotically the FWER for any underlying dependence structure, and when the sample size is sufficiently large. {Recent results are reviewed in a more general setting}. Our main  contribution is to clarify and {supplement} existing results in the literature. Additionally we provide an extensive comparison of methods, with different statistics. An empirical evaluation of the power of the proposed method using different graph structures is described where it is shown that the sample size is the only crucial parameter. These results are then applied on a real dataset {consisting of small animal brain recordings} by fMRI, where recordings on dead rats provide a null model from the experiments \cite{guillaume2020functional}.

The article is organized in five parts. The first part defines the asymptotic tests setting and describe the four multiple testing procedures. For each of these methods, single-step and step-down approaches are described. The second part is dedicated to the application to multiple correlation testing. Simulations are proposed in a third part, where we study among others the behavior of the power with respect to the sparsity of the set of rejected hypothesis. In the fourth part, our approach is applied to a real fMRI dataset on rats. Finally, we comment a possible extension to False Discovery Rate control.

\section{Procedures controlling asymptotically the FWER}
 
\label{sec:procedures}

For all $i \in \{1,\hdots,m\}$, we consider tests (\ref{eqn:tests}), based on test statistics $T_{n,i} \left (\mathbb X^{(n)} \right)$. {Fix} $T_{n,i} \left (\mathbb X^{(n)} \right)= \sqrt n \: \widehat \theta_{n,i}\left (\mathbb X^{(n)} \right).$ %\sa{Is this an assumption?}\ig{no, just another way to write statistics}
Throughout this manuscript, we assume that $\hat \theta_{n,\cdot}\left (\mathbb X^{(n)} \right)$ is a consistent estimator of $\theta (P)$ and that it has an asymptotic Gaussian distribution. Namely, for all $P \in \mathcal P$
\begin{equation} \label{hypGauss}
\sqrt{n}\left( \: \hat \theta_{n,\cdot}(\mathbb X^{(n)}) - \theta(P) \right)  \mathop{\longrightarrow}^{\quad d \quad} \mathcal N_m( 0, \Sigma), \quad \text{when~}{n\to +\infty},
\end{equation}
where $\displaystyle{\mathop{\longrightarrow}^{\quad d \quad}}$ denotes the convergence in distribution. We assume that $\Sigma$ is invertible and $\Sigma_{ii} =1$ when $\theta_i=0$. That is, every statistic is normalized under the null hypothesis of \eqref{eqn:tests}, so that $\Var(T_i)=1$, for all $i\in\{j : H_{0j}\text{~holds} \}$. It is not restrictive since one needs to control the variance of the statistics independently of the observations in order to apply a statistical test.

Let $ \left(p_{n,i} \left (\mathbb X^{(n)} \right) \right )_{1 \le i \le m}$ be a family of $p$-values resulting from each $m$ individual test. The asymptotic Gaussian assumption~\eqref{hypGauss} gives rise to the asymptotic $p$-value process:
\begin{equation} \label{pval} 
\forall i \in \{1,\hdots,m\}, \: p_{n,i} \left ( \mathbb X^{(n)}\right) = 2\left [ 1- \Phi \left(  \bigl| T_{n,i}(\mathbb X^{(n)}) \bigr| \right)  \right],
\end{equation}
where $\Phi$ is the standard Gaussian cumulative distributive function.
Multiple testing procedures will be based on this $p$-value process. It is worth pointing that no assumption on the dependence structure of the $p$-values is needed.

{In this section, we proceed with the study of testing simultaneously $m$ hypotheses, $H_{0,i}$ against $H_{1,i}$ for $i \in \{1,\dots,m\}$. Results are reviewed in a more general setting where the Family Wise Error Rate (FWER) is used as multiple testing criterion.} 

For all $P \in \mathcal P$, let $\mathcal H_0(P) = \{i \in \{1,\hdots,m\}: \theta_i = 0\}$ be the index set of true null hypotheses, that is, the index set of all $i$ such that $H_{0,i}$ is satisfied for $P$. Denote $m_0(P) = |\mathcal H_0(P)|$, its cardinality. The FWER depends on the rejected set $\mathcal R$ and on the (unknown) distribution $P$ of the observations. The FWER corresponds to the probability of rejecting at least one true null hypothesis, namely
\begin{equation}
\forall P \in \mathcal P, \text{ FWER}( \mathcal R,P) = \mathbb P(|\mathcal R \cap \mathcal H_0(P)| \ge 1).
\end{equation}

Since we consider an asymptotic $p$-value process, we can only get asymptotic results in terms of control of the errors.
\begin{dfn} A multiple testing procedure $\mathcal R$ is said to asymptotically control the FWER for a distribution family $\mathcal P$ at level $\alpha$ if for all $P \in \mathcal P$,
 \begin{equation}
 \underset{n \to + \infty}{\limsup} \text{ FWER}(\mathcal R,P) \le \alpha. 
 \end{equation}
 \end{dfn}
{In this section we will} describe multiple testing procedures which asymptotically control FWER for the two-sided testing problem \eqref{eqn:tests}, based on the asymptotic $p$-value process \eqref{pval}.

\subsection{Single-step procedures}

We propose four procedures to determine the rejection set.

\subsection*{Bonferroni}
The Bonferroni procedure, \cite{bonferroni1935calcolo}, is the most classical example of FWER control. 

\begin{meth}[Bonferroni] The  Bonferroni multiple testing procedure is defined by
\begin{equation}
\mR^{bonf}_\alpha = \left \{ i \in \{1,\hdots,m\}: p_{n,i} \le \frac{\alpha}{m} \right \}.
\end{equation}
\end{meth}

\begin{prop} 
\label{prop:fwerbonf} 
For the two-sided testing problem \eqref{eqn:tests} based on the asymptotic $p$-value process \eqref{pval}, the method $\mathcal R^{bonf}_\alpha$  provides an asymptotic control of the FWER at level $\alpha$, that is, for all $P \in \mathcal P$,
\begin{equation*}  
\underset{n\to+\infty}{\lim} \rm{ FWER}\Big(\mathcal R^{bonf}_{\alpha},P\Big) \le \alpha.  
\end{equation*}
\end{prop}

This control does not require any assumption on the dependence structure of the $p$-values, however under strong dependence the Bonferroni correction is known to be conservative, see \cite{bland1995multiple}.

%%%%%%%%%%%%%%%%%%%%%%%%%%%%%%%%%%%%%
\subsection*{\v{S}id\'ak} 
\label{sec:PLOD}
As mentioned by \cite{westfall1993resampling}, an asymptotic FWER controlling procedure can be derived by \Sidak's inequality \citep{vsidak1967rectangular}.

\begin{thm}[\Sidak's inequality, \cite{vsidak1967rectangular}]
Let $X$ be a random vector having an $m$-multivariate normal distribution with zero mean values and invertible covariance matrix. Then $X$ satisfies the following inequality, for every positive constant $b \in \R^m_+$,
\begin{equation} 
\label{Sid} 
\P(|X_1| \le b_1, \hdots, |X_m| \le b_m) \,\ge\, \prod_{i=1}^m\P(|X_i| \le b_i). 
\end{equation}  
\end{thm}

For the specific case of correlation testing, \cite{drton2004model} used this inequality to construct a procedure that asymptotically controls the FWER for the problem \eqref{eqn:tests}.

\begin{meth}[\Sidak] 
Let $c_{\alpha}^s = \Phi^{-1} \Big (\frac{1}{2} (1-\alpha)^{1/m} + \frac{1}{2} \Big) > 0$. The  \v{S}id\'ak's multiple testing procedure is defined by
\begin{equation}
\mR^s_\alpha = \big \{ i \in \{1,\hdots,m\}: |T_{n,i}| > c_{\alpha}^s \big \}.
\end{equation}
\end{meth}

\begin{prop} 
\label{prop:fwersidak} 
For the two-sided testing problem~\eqref{eqn:tests} for which the asymptotic Gaussian assumption~\eqref{hypGauss} holds, the method $\mathcal R^s_\alpha$  provides an asymptotic control of the FWER at level $\alpha$, namely, for all $P$ such that $\Sigma$ is invertible,
\begin{equation*}  
\underset{n\to+\infty}{\lim} \: \rm{ FWER}(\mathcal R^s_{\alpha},P) \le \alpha.  
\end{equation*}
\end{prop}

The procedure is valid for any dependencies, as soon as the inequality~\eqref{Sid} holds. This is in particular true in Gaussian setting \citep{vsidak1967rectangular}.

The \Sidak's procedure is less conservative than the Bonferroni procedure. This comparison is illustrated in table 2.2 of  \cite{westfall1993resampling}{, where} the difference between the two adjustments becomes larger with larger $m$.    

%%%%%%%%%%%%%%%%%%%%%%%%%%%%%%%%%%%%%%%%%%%%%%
\subsection*{Non parametric bootstrap}
\label{{sec:romano-wolf}}
\cite{romano2005exact} propose an asymptotic FWER controlling procedure which only requires a monotonic assumption on the family of thresholds. 
     
\begin{meth}[\BootRW] For all  $\mathcal C \subset \{1,\hdots,m\}$, let $t_{n,\alpha}(\Sigma,\mathcal C)$ be the $ (1-\alpha)$-quantile of the probability distribution $\mathcal L \left( \| \restriction{ \mathcal N_m(0, \Sigma)}{\mathcal C}\|_{\infty}\right)$,
where $\restriction{ \mathcal N_m(0,\Sigma)}{\mathcal C}$ is the restriction of the Gaussian distribution $ \mathcal N_m(0,\Sigma)$ on $\mathcal C$, namely $ \mathcal N_{|\mathcal C|}\Big(0,(\Sigma)_{i,i^{\prime} \in \mathcal C \times \mathcal C}\Big) $. The Romano-Wolf's multiple testing procedure is defined by
\begin{equation}
\mR^{\BootRW}_\alpha = \Big \{ i \in \{1,\hdots,m\}: |T_{n,i}| > \hat t_{n,\alpha}(\Sigma,\mathcal C_m)\Big \},
\end{equation}
where $\mathcal C_m = \{1,\hdots,m\}$ and $ \hat t_{n,\alpha}(\Sigma,\mathcal C_m)$ is computed using bootstrap resamples of $\mathbb X^{(n)}$. A bootstrap resample from $\mathbb X^{(n)}$ is denoted by $\mathbb X^{(n)^*}$ and defined as an $n$ independent and identically distributed ({\it i.i.d.}) sample from the empirical distribution of $\mathbb X^{(n)}$.
\end{meth}

Note that this method closely relies on its ability to approximate the joint distribution of the test statistics. 
\begin{prop}[\cite{romano2005exact}] \label{prop:rwcontrol}  Assume that for any metric $d$ metrizing weak convergence on $\mathbb R^{m_0(P)}$,
\begin{equation} d \left ( \mathcal L \left(    \left(T_{n,i}\left(\mathbb X^{(n)^*}\right)\right)_{i \in \mathcal H_0(P)} \Big | \: \mathbb X^{(n)}   \right), \mathcal L \left( \left(T_{n,i}\left(\mathbb X^{(n)}\right)\right)_{i \in \mathcal H_0(P)} \right)  \right) \overset{\mathbb P}{\longrightarrow} 0, \nonumber \end{equation}  
where $\overset{\mathbb P}{\longrightarrow}$ means a convergence in probability. $ \mathcal L \left(    \left(T_{n,i}\left(\mathbb X^{(n)^*}\right)\right)_{i \in \mathcal H_0(P)} \Big | \: \mathbb X^{(n)}   \right)$ denotes the conditional distribution of $ \left(T_{n,i}\left(\mathbb X^{(n)^*}\right)\right)_{i \in \mathcal H_0(P)}$ given $\mathbb X^{(n)}$.

Then, for the two-sided testing problem~\eqref{eqn:tests} for which the asymptotic Gaussian assumption~\eqref{hypGauss} holds, the method  $\mathcal R^{\BootRW}_\alpha$ provides an asymptotic control of the FWER at level $\alpha$, that is, for all $P$ such that $\Sigma$ is invertible,
\begin{equation*}  
\underset{n\to+\infty}{\limsup} \: \rm{ FWER}(\mathcal R^{\BootRW}_\alpha, P) \le \alpha.  
\end{equation*}
\end{prop}
This result is also derived in \cite{dudoit2007multiple} (where \BootRW~is procedure 4.21). The proof in our setting is rewritten in the Appendix.

%%%%%%%%%%%%%%%%%%%%%%%%%%%%%%%%%%%%%%%%%%%%%%
\subsection*{Parametric bootstrap}
\cite{drton2007multiple} {detailled} a parametric bootstrap method for testing \eqref{eqn:tests} {on} partial correlation coefficients. {This method differs from \cite{romano2005exact}.} Indeed, the quantile is here evaluated on the asymptotic distribution rather than the empirical distribution. {An estimation} of the matrix $\Sigma$ {is needed}. Denote by $\hat \Sigma_n$ such an estimator. 

\begin{meth}[\MaxT] Let $t_{n,\alpha}(\hat \Sigma_n)$ be the $(1-\alpha)$-quantile of the distribution $\mathcal L\Big (\| \mathcal N_m(0,\hat \Sigma_n)\|_{\infty} \Big)$. The \MaxT~multiple testing procedure is defined by
\begin{equation}\label{eq:maxTinf}
\mR^{\MaxT}_\alpha = \Big \{ i \in \{1,\hdots,m\}: |T_{n,i}| > t_{n,\alpha} \big(\hat \Sigma_n \big) \Big \},
\end{equation}
where $t_{n,\alpha} \big(\hat \Sigma_n \big)$ is computed using (simulated) samples of $\mathcal N_m(0,\hat \Sigma_n)$.
\end{meth}

\begin{prop}\label{prop:maxTcontrol}
Assume that $(\hat \Sigma_n)$ is a consistent estimator of $\Sigma$, that is, when $n$ goes to infinity,
\begin{equation} \label{estcv} \hat \Sigma_n \overset{\mathbb P}{\longrightarrow}  \Sigma.
\end{equation} 
Then, for the two-sided testing problem~\eqref{eqn:tests} for which the asymptotic Gaussian assumption~\eqref{hypGauss} holds, the method  $\mathcal R^{\MaxT}_\alpha$ provides an asymptotic control of the FWER at level $\alpha$, namely,  for all $P$ such that $\Sigma$ is invertible,
\begin{equation}  
\underset{n\to+\infty}{\limsup} \: \rm{ FWER}(\mathcal R^{\MaxT}_\alpha, P) \le \alpha.  
\end{equation}
\end{prop}

{Even if} many results are known on the maximum of Gaussian variables (see {\it e.g.} \cite{nadarajah2008exact} and references therein), there is no explicit formula of quantile $t_{n,\alpha}(\hat \Sigma_n)$ of absolute multivariate Gaussian distributions in general case. {Therefore an estimation is required. We propose here to estimate the quantile by parametric bootstrap in Method~4.} {It is also possible to use, for example,} \cite{GenzBretz}'s algorithm (available in function \texttt{qmvnorm} of \texttt{mvtnorm} on \texttt{R}). {However, even if} the estimation has a good quality, the {computational cost} is very high.

Procedure \MaxT~is available with any consistent estimation of $\Sigma$. However, in practice, the quality of estimation may influence the quality of the procedure for a given number $n$ of observations. 

A natural candidate for $\hat \Sigma_n$ is the empirical covariance of observations $\mathbb X^{(n)}$. Yet, \cite{johnstone2001distribution} established that when $n$ increases while $\frac{p}{n}$ converges to a constant, on zero limit, then the empirical estimation provides a non consistent estimate because its eigenvalues do not converge to those of the covariance matrix. Several methods {have been proposed} to reduce the dimension of the estimation setting to overpass this problem, requiring assumptions of sparsity or structured matrices. 
For examples, we can cite banding methods in \cite{bickel2008regularized} and \cite{wu2009banding}, thresholding rules in \cite{bickel2008covariance} for instance, shrinkage estimation in \cite{ledoit2012nonlinear} and convex optimization techniques in \cite{banerjee2006convex}. Those estimators can be plugged into \MaxT~procedure.

\subsection{Step-down versions}

Single-step methods can be conservative. A well-known improvement is step-down method. It is a recursive algorithm which increases the power of the procedures, still preserving FWER control. See {\it e.g.} \cite{romano2005exact,goeman2010sequential}. The principle is to iterate multiple testing on the non-rejected hypothesis, as described below.

\textbf{Step-down Algorithm.}
\begin{algorithmic}
\STATE Let $\mathcal C_0 = \varnothing$ and $\mathcal C_1 = \{1,\hdots,m\};$
\STATE Initialize $j=1$. 
    \WHILE{$ \mathcal C_j \neq \mathcal C_{j-1}$}
        \STATE  Let $R_{j}$ be the {set of rejected indexes} of a given multiple testing procedure applied on tests indexes $\mathcal C_{j}$.
        \STATE Define ${\mathcal C}_{j+1}=\{i\in{\mathcal C}_{j},\;i\notin R_{j} \}$.
\STATE Do $j=j+1$.
\ENDWHILE
\RETURN ${\mathcal C}_\infty = \mathcal C_j.$
\STATE ${\mathcal C}_\infty$ is the final set of non-rejected null-hypothesis indexes. \\
\end{algorithmic}

If a single-step method provides an asymptotic control of the FWER, {then} the following proposition provides sufficient conditions under which its step-down version preserves this control. 
\begin{prop} \label{lapropo} Let $\big(R_{\mathcal C} \big)_{\mathcal C \subset \mathcal H}$ be a family of rejection sets given by multiple testing procedure. Suppose that 
\begin{itemize}
\item {For all~}  $P \in \mathcal P,$
 \begin{equation} \label{controlSSas}
  \underset{n\to+\infty}{\limsup}  \text{ FWER} \Big(R_{\mathcal H_0(P)},P \Big) \le \alpha, 
\end{equation}
\item For all $\mathcal C \subset \{1,\hdots,m\},\: \mathcal C \mapsto \mathcal R_{\mathcal C}$ is decreasing in $\mathcal C$, that is,
\begin{equation}  \label{incas} 
\forall  \mathcal C,\mathcal C^{\prime} \subseteq \{1,\hdots,m\}, \mathcal C \subseteq \mathcal C^{\prime} \Rightarrow   \mathcal R_{\mathcal C}\supseteq \mathcal R_{\mathcal C^{\prime}}. 
\end{equation}
\end{itemize}
Then, for all $P \in \mathcal P$, 
\begin{equation}\underset{n\to+\infty}{\limsup} \text{ FWER }\Big(R_{{\mathcal C}_\infty},P\Big) \le \alpha. \nonumber \end{equation} 
\end{prop}

%%%%%%%%%%%%%%%%%%%%%
We deduce that the four methods displayed above control asymptotically the FWER for tests~\eqref{eqn:tests}.

\begin{cor}\label{thm:sdcontrol}
 For the two-sided testing problem \eqref{eqn:tests} for which the asymptotic Gaussian assumption \eqref{hypGauss} holds, 
step-down algorithm applied to Methods 1 to 4  are asymptotic FWER controlling procedures at level $\alpha$. 
\end{cor}
\begin{proof}
It is sufficient to verify that the four methods satisfy \eqref{controlSSas} and \eqref{incas}. 
\end{proof}

\section{Application to correlation tests}

\label{sec:correlation}

Let $\left\{Y_{1},\hdots,Y_{n} \right\}$ be independent realizations from a  random vector $Y$ with values in $\R^p$. Denote $Y=(Y^{(i)})_{i=1,\dots,p}$. Suppose $Y$ has a finite expectancy and a semi-definite positive covariance matrix. Assume also that $Y$ has finite fourth moments. Define $\Gamma=(\rho_{ij})_{i,j=1,\dots,p}$ the correlation matrix of $Y$. That is, $\rho_{ij}=\text{Cor}(Y^{(i)},Y^{(j)})$. 
Consider the two-sided testing problem:
\begin{equation} \label{letest}\tag{P-2}
 H_{0,ij} : \rho_{ij}=0  \text{~ against ~}  H_{1,ij} :  \rho_{ij}\neq 0, \quad (i,j)\in\mathcal H, 
\end{equation}
where $\mathcal H$ is the set of tested indexes, $\mathcal H=\{1\leq i <j\leq p\}$. {The $p$-value processes associated to tests~\eqref{letest} present a dependence structure which is possible to detail for different tests statistics.}  

Usual statistics proposed in literature for tests~\eqref{letest} are based on the empirical correlation. Define $\hat{\rho_{i j}}=\frac{1}{n}\sum_{\ell=1}^n (Y^{(i)}_\ell-\overline{Y^{(i)}})(Y^{(j)}_\ell-\overline{Y^{(j)}})$, the empirical correlation between $\left\{Y_{1}^{(i)},\hdots,Y_{n}^{(i)} \right\}$ and $\left\{Y_{1}^{(j)},\hdots,Y_{n}^{(j)} \right\}$, for $i,j=1,\dots,p$. The overline denotes the empirical mean, that is, $\overline{Y^{(i)}}=\frac{1}{n}\sum_{\ell=1}^n Y_\ell^{(i)}$. Denote $\hat\Gamma_n=(\hat \rho_{ij})_{1\leq i,j\leq p}$. We will focus our analysis on four test statistics:
\begin{description}
\item[Empirical statistic.]~
 \[
 T^{(1)}(\hat\rho_{ij})=\sqrt{n}\,\hat\rho_{ij}
 \]
\item[Student statistic.]~
 \[
 T^{(2)}(\hat\rho_{ij})=\sqrt{n-2}\,\frac{\hat\rho_{ij}}{\sqrt{1-\hat\rho_{ij}^2}}
 \]See {\it e.g.} Section 4.2.1 of \cite{anderson.2003.1}.

\item[Fisher statistic.]~\\
 \[
 T^{(3)}(\hat\rho_{ij})=\frac{\sqrt{n-3}}{2}\log\left(\frac{1+\hat\rho_{ij}}{1-\hat\rho_{ij}}\right)\,.
 \]
{Fisher transform is commonly used to improve} the convergence to the Gaussian distribution for univariate statistics ({\it i.e.} for fixed $(i,j)$) (see Section 4.2.3 of \cite{anderson.2003.1}).
 
\item[Second-order statistic.]~\\
For given $(i,j)\in\mathcal H$ such that $\Var(Y^{(i)})=\Var(Y^{(j)})=1$, $\hat \rho_{ij}$ is the empirical mean of $(Z_{\ell}^{(ij)})_{\ell=1,\dots ,n}$ with $Z_{\ell}^{(ij)}=(Y_{\ell}^{(i)}-\overline{Y^{(i)}})(Y^{(j)}_{\ell}-\overline{Y^{(j)}})$. Hence a test on $\hat \rho_{ij}$ can be driven using the usual test statistics on an expectation under asymptotic Gaussian assumption: 
 \begin{align*}
 T^{(4)}_{ij}=\sqrt{n}\,\frac{\overline{Z^{(ij)}}}{\sqrt{\hat\theta_{ij}}} \text{~where~} \hat\theta_{ij} &= \frac{1}{n}\sum_{\ell=1}^n \left(Z^{(ij)}_{\ell }-\frac{1}{n}\sum_{\ell'=1}^n Z^{(ij)}_{\ell'}\right)^2\,.
 \end{align*}
  The quantity $\hat\theta_{ij}$ is an estimation of \[\theta_{ij}=\Var\left[(Y^{(i)}-\E(Y^{(i)}))(Y^{(j)}-\E(Y^{(j)}))\right].\]   
When $(i,j)\in\mathcal H_0$, $\E(T^{(4)}_{ij})=0$.
{\cite{cai2016large} studied these statistics for multiple testing of correlations}.

\end{description}

Asymptotics of the empirical correlations are derived in \cite{aitkin1969some} where the author established the asymptotic normality for Gaussian distributed $Y$, and in \cite{corr_nonnormal} for non Gaussian distribution. 

\begin{prop}[\cite{aitkin1969some, corr_nonnormal}] 
\label{prop:corrgauss} 
Let $\left\{Y_{1},\hdots,Y_{n} \right\}$ be independent realizations from a  random vector $Y=(Y^{(i)})_{i=1,\dots,p}$. Suppose $Y$ has finite fourth moments.

The vector of empirical correlations $\hat\rho_{n,\cdot}=(\hat\rho_{n,h})_{h\in\mathcal H}$ is asymptotically Gaussian,
\begin{equation} 
\sqrt{n}(\hat \rho_{n,\cdot} - \rho_{\cdot})  \xrightarrow[n\to +\infty]{\quad d \quad} \mathcal N_m( 0, \Omega(\Gamma)), 
\nonumber  
\end{equation}
with $\Omega(\Gamma) = (\omega_{ij,kl})_{(ij,kl) \in\mH^2}$ given by
\begin{multline*}
\omega_{ij,kl} = \rho_{ijkl}+\frac{1}{4}\rho_{ij}\rho_{kl} \Big( \rho_{iikk} + \rho_{iill}+ \rho_{jjkk}+ \rho_{jjll}\Big)\\  -\frac{1}{2}\rho_{ij}  \Big( \rho_{iikl} + \rho_{jjkl}\Big) -\frac{1}{2}\rho_{kl}  \Big( \rho_{ijkk} + \rho_{ijll}\Big),
\end{multline*}
where for all $i,j,k,l$, \[
\rho_{ijkl}=\frac{\E\left[(Y^{(i)}-\E(Y^{(i)}))(Y^{(j)}-\E(Y^{(j)}))(Y^{(k)}-\E(Y^{(k)}))(Y^{(l)}-\E(Y^{(l)}))\right]}{\sqrt{\mathrm{Var}(Y^{(i)})\mathrm{Var}(Y^{(j)})\mathrm{Var}(Y^{(k)})\mathrm{Var}(Y^{(l)})}}.
\]

In particular, when $Y$ is Gaussian, $\Omega(\Gamma) = (\omega_{ij,kl})_{(ij,kl) \in\mH^2}$ satisfies
\begin{align*}
\label{eqn:omega}
\omega_{ij,ij} &= \Big ( 1 - \rho_{ij}^2 \Big )^2,\\
\omega_{ij,il} &= -\frac{1}{2} \rho_{ij} \rho_{il}\Big(1-\rho_{ij}^2-\rho_{il}^2-\rho_{jl}^2  \Big) + \rho_{jl} \Big ( 1-\rho_{ij}^2-\rho_{il}^2\Big ), \quad \text{for~} j\neq l\\
\omega_{ij,kl} &= \frac{1}{2}\rho_{ij}\rho_{kl} \Big( \rho_{ik}^2 + \rho_{il}^2+ \rho_{jk}^2+ \rho_{jl}^2\Big) +  \rho_{ik} \rho_{jl}+ \rho_{il} \rho_{jk}\\
&\qquad  - \rho_{ik} \rho_{jk} \rho_{kl}- \rho_{ij} \rho_{ik} \rho_{il}- \rho_{ij} \rho_{jk} \rho_{jl}- \rho_{il} \rho_{jl} \rho_{kl}, \quad \text{for~} i\neq k \text{~and~} j\neq l.
\end{align*}

\end{prop}  

Asymptotic distributions of test statistics $\{ (T_{ij}^{(k)})_{1\leq i <j\leq n}, k=1,\dots,4\}$ follow.
\begin{cor} \label{cor:stat}
 For all $k=1,\dots,3$, vector statistics $(T^{(k)}(\hat\rho_{ij}))_{1\leq i <j\leq p}$ converge in distribution to a Gaussian random variable with covariance matrix $\Omega^{(k)}(\Gamma)$,
\[
(T^{(k)}(\hat\rho_{ij})-T^{(k)}(\rho_{ij}))_{1\leq i <j\leq p} \mathop{\longrightarrow}^{\quad d \quad} \mathcal N_m( 0, \Omega^{(k)}(\Gamma)), \quad \text{when~}{n\to +\infty},
 \]
 where \begin{itemize}[topsep=-5pt]
\item[] $\Omega^{(1)}(\Gamma)=\Omega(\Gamma)$ defined in Proposition~\ref{prop:corrgauss}, 
\item[] $\Omega^{(2)}(\Gamma)=(\omega_{ij,kl}^{(2)})$ with $\omega_{ij,kl}^{(2)}= \omega_{ij,kl}/\left((1-\rho_{ij}^2)(1-\rho_{kl}^2)\right)^{3/2}$,
\item[] $\Omega^{(3)}(\Gamma)=(\omega_{ij,kl}^{(3)})$ with $\omega_{ij,kl}^{(3)}= \omega_{ij,kl}/\left({(1-\rho_{ij}^2)(1-\rho_{kl}^2)}\right)$.
\end{itemize}
\end{cor}
\begin{proof}
Expressions of $\Omega^{(2)}(\Gamma)$ and $\Omega^{(3)}(\Gamma)$ result from the Delta method \citep{wasserman2013all}.
\end{proof}

{We can now state an equivalent result for Second-order statistics.}
\begin{prop}
\label{prop:gaussian}
Let $\left\{Y_{1},\hdots,Y_{n} \right\}$ be independent realizations from a  random vector $Y=(Y^{(i)})_{i=1,\dots,p}$. Suppose $Y$ has finite fourth moments.

Vector statistics $(T^{(4)}_{ij})_{1\leq i <j\leq p}$ converge in distribution to a Gaussian random variable with covariance matrix $\Omega^{(4)}(\Gamma)$,
\[
(T^{(4)}_{ij}-\E(T^{(4)}_{ij}))_{1\leq i <j\leq p} \mathop{\longrightarrow}^{\quad d \quad} \mathcal N_m( 0, \Omega^{(4)}(\Gamma)), \quad \text{when~}{n\to +\infty},
\]
  where $\Omega^{(4)}(\Gamma)=(\omega_{ij,kl}^{(4)})$ with $\omega_{ij,kl}^{(4)}= \rho_{ijkl}/\left({(\rho_{ijij}-\rho_{ij}^2)(\rho_{kl kl}-\rho_{kl}^2)}\right)^{1/2}$.
  
When $Y$ is Gaussian, $\Omega^{(4)}(\Gamma)$ satisfies \[\omega_{ij,kl}^{(4)}= \frac{\rho_{ij}\rho_{kl}+\rho_{ik}\rho_{jl}+\rho_{il}\rho_{kj}}{\left({(1+\rho_{ij}^2)(1+\rho_{kl}^2)}\right)^{1/2}}.\]
In particular, for all $(i,j)\in\mathcal H$, $\omega^{(4)}_{ij,ij}=1$.
\end{prop}
\begin{proof}
First, the law of large number ensures that for all $(i,j)\in \mathcal H$, $\hat \theta_{ij}$ converges almost surely to $\theta_{ij}$. {Applying Slutsky's Theorem (\cite{gut2012probability}, Theorem 11.4) gives the asymptotic distribution of}$\sqrt{n}\frac{\overline{Z^{(i,j)}}}{\hat \theta_{ij}}$. Applying the continuity theorem on the characteristic functions, we deduce that the asymptotic distribution is Gaussian. The expression of $\Omega^{(4)}$ then results from the fact that $\theta_{ij}=(\rho_{ijij}-\rho_{ij}^2)\sqrt{\Var(Y_i)\Var(Y_j)}$ and that $\E(Z^{(ij)}Z^{(kl)})=\rho_{ijkl}\sqrt{\Var(Y_i)\Var(Y_j)\Var(Y_k)\Var(Y_l)}$.
\end{proof}

To conclude, {by} Corollary~\ref{cor:stat}{,} procedures of Section~\ref{sec:procedures} {can be applied}. Single-step and step-down Methods~1--4 thus control asymptotically FWER.  Methods~1--3  are non parametric, while Method~4 {requires} an explicit formula of the asymptotic covariance of tests statistics. In particular, the matrix $\Sigma$ in this method corresponds to $\Omega^{(k)}(\Gamma)$ given in Corollary~\ref{cor:stat} and Proposition~\ref{prop:gaussian}. Method~4, \MaxT, can then be applied {by estimating quantiles using plug-in estimate of $\Omega^{(k)}(\Gamma)$}. {This consists simply in plugging in $\hat \Gamma_n$ for $\Gamma$, where $\hat \Gamma_n$ is an estimator of $\Gamma$.} In the general case, such estimators need finite eight moments. %Expressions in the Propositions above enable to built easily empirical versions of $\Omega^{(k)}(\Gamma)$.

When the correlations are evaluated on Gaussian-distributed samples, for all $k=1,\dots,4$, matrices $\Omega^{(k)}(\Gamma)$ are explicitly given with respect to the correlation matrix $\Gamma$. Matrices $\Omega^{(k)}(\Gamma)$ can thus be estimated by $\Omega^{(k)}(\hat\Gamma_n)$, where $\hat \Gamma_n$ is an estimator of {$\Gamma$}. In the following, we will consider that $\hat \Gamma_n$ is the empirical correlation matrix. As previously, it is worth noticing that other estimators may have better performances{. Obviously, it is possible to }plug-in a matrix obtained {\it e.g.} by shrinkage estimation, banding methods, thresholding rules\dots

{We can now state an adaptation of Method~4.}

\renewcommand\themeth{4'}
\begin{meth}[\MaxT~for correlation tests~\eqref{letest}]  \label{meth4bis} Assume $Y$ is Gaussian. Consider $k=1,\dots,4$. Let $t_{n,\alpha}(\hat \Sigma^{(k)}_n)$ be the $(1-\alpha)$-quantile of the distribution $\mathcal L\Big (\| \mathcal N_m(0,\hat \Sigma^{(k)}_n)\|_{\infty} \Big)$, with $\hat\Sigma^{(k)}_n=\Omega^{(k)}(\hat\Gamma_n)$. The \MaxT~ multiple testing procedure is defined by
\begin{equation}
\mR^{\MaxT}_\alpha = \Big \{ (i,j) \in \mathcal H: |T^{(k)}(\hat\rho_{ij})| > t_{n,\alpha}^{(k)} \big(\hat \Sigma^{(k)} \big) \Big \},
\end{equation}
where $t_{n,\alpha}^{(k)} \big(\hat \Sigma^{(k)}_n \big)$ is computed using (simulated) samples of $\mathcal N_m(0,\hat \Sigma^{(k)}_n)$.
\end{meth}

Lemma~\ref{lem:corDcor} below then ensures that Method~\ref{meth4bis} controls asymptotically FWER.

\begin{lem} \label{lem:corDcor}
Assume $Y$ is Gaussian. For $k=1,\dots,4$, $\Omega^{(k)}(\hat\Gamma_n)$ tends in probability to $\Omega^{(k)}(\Gamma)$ when $n$ goes to infinity.
\end{lem}
 \begin{proof}
Proposition~\ref{prop:corrgauss} ensures that $\sqrt{n}(\hat \Gamma_n - \Gamma)  \xrightarrow[n\to +\infty]{\quad d \quad} \mathcal N_m(0,\Omega(\Gamma))  $. Since $\frac{1}{\sqrt{n}} \xrightarrow[n\to +\infty]{\quad\P \quad} 0$, Slutsky's theorem (\cite{gut2012probability}, Theorem 11.4) implies that $\hat \Gamma_n- \Gamma \overset{\mathbb P}{\longrightarrow}  0$, that is, $\hat \Gamma_n \overset{\mathbb P}{\longrightarrow}  \Gamma$ when $n$ goes to infinity. We conclude using the continuous mapping Theorem (\cite{gut2012probability}, Theorem 10.3).
 \end{proof}

{Remark} that all our results can be extended immediately to multiple testing of partial correlation (see \cite{roux}, Section 5.2.4). Indeed, similar convergence as in Proposition~\ref{prop:corrgauss} exists for partial correlation  \citep{aitkin1969some}. Statistics in Methods 1-4 can be rewritten using partial correlation {and} asymptotic normality also holds.

\section{Simulations}

\label{sec:simu}

Correlation matrices can be used to define adjacency matrices of dependence graphs, where nodes correspond to variables and edges correspond to non zero of adjacency matrices (see {\it e.g.} \cite{drton2007multiple}). 
In practice, empirical correlations are often full matrices, while dependence graphs are expected to have a smaller set of edges, see e.g. \cite{kolaczyk2009statistical,newman2010networks}.
Thus multiple testing on correlations can be used for a better graph estimation, where only significant correlations $\rho_{ij}$ give an edge between nodes $i$ and $j$. A prior in many graph estimations is the sparsity of the estimated graph, e.g. in Graphical Lasso, \cite{Friedman2008}. However, this crucial point may not be satisfied in real data applications, for instance in neuroscience \citep{markov2013cortical}.

Simulations are conducted to evaluate the behavior of the proposed approaches described in previous parts \ref{sec:procedures} and \ref{sec:correlation}. Indeed, we showed that the FWER is controlled asymptotically and that the choice of the statistics may influence the asymptotic control. We conduct a precise and careful evaluation of the proposed methods by computing the number of falsely detected edges (FWER) and correctly identified edges (power -- also named sensitivity). Three main parameters of the methods are varied: sparsity of the true graphs, size of the sample, and signal to noise ratio. After presenting the graph model chosen in our simulations, the results are displayed for FWER and latter for power.       

Simulations were done on \texttt{R} with \texttt{igraph} and \texttt{TestCor} packages, \cite{igraph, TestCor}.

\subsection{Choice of models}

\label{sec:simu-model}

Our simulation procedure requires the simulation of Gaussian data with a given covariance matrix defined in order to match a given graph structure. In addition, the covariance matrix has to be positive definite. It is rather simple to get an adjacency matrix following a network model such as small-world or scale-free \citep{newman2010networks}. However, it is difficult to guarantee its positive definiteness theoretically, see \cite{guillot2011posdef}. Starting from a network model, we propose to generate a corresponding adjacency matrix, denoted $\mathbf A$. One way to {guarantee} the positive definiteness of the matrix is to control the eigenvalues. This is the purpose of the following lemma \ref{lem.posdef}.

From the adjacency matrix $\mathbf A$ we build a correlation matrix $\mathbf M$ with a constant non zero value $\rho$ by \[\mathbf M= \mathbf I+\rho \mathbf A,\] where $\mathbf I$ is the identity matrix, and $\rho\in(-1,1)$. 
\begin{lem}
  \label{lem.posdef}
{Let $\mathbf M$ satisfy $\mathbf M=\mathbf I+\rho \mathbf A$ with $\rho$ a real.} 
Then, $\mathbf M$ is positive definite if and only if $|\rho|<\frac{1}{|\lambda_{min}|}$, where $\lambda_{min}$ is the smallest eigenvalue of $\mathbf A$.  
\end{lem}
For regular graphs, it is shown that $\lambda_{min}< -1/r$ independently on the size of the graph, where $r$ is the valence of the graph \citep{cvetkovic2004spectral}.
In particular, for a chain graph, $|\rho|<0.5$ always gives a positive definite matrix $\mathbf M$. {It is getting} more complicated
for small-world or scale-free models \citep{cvetkovic2004spectral}. A classical approach to ensure positive definiteness is diagonal dominance (see {\it e.g.} \cite{SchaferStrimmer} with precision matrices). Our model is quite similar, but considers constant non-zero entries in the correlation matrix.

In the sequel, we focus on Stochastic Block Models \citep{SBM}. We consider a model where the number of nodes $p$ is equal to 26, with two community of size $p/2=13$. The adjacency matrix $\mathbf A$ is defined as 
\[\mathbf{A}=
  \begin{tikzpicture}[baseline=(current bounding box.center),
                      large/.style={font=\large}]
    \matrix (M)[matrix of math nodes, nodes in empty cells,
                left delimiter={(}, right delimiter={)},
                column sep={3em,between origins},
                row sep={2em,between origins}
    ]{   &   &         &&&&     \\
          &     &   &     &  & & \\
          &    &        &&&&     \\
           &    &   &   &&&        \\
            &    &      &&&&       \\
       };
    \node[large] at (M-2-2){$\bigl(a_{i,j}\bigr)_{\subalign{i&=1,\dots,p/2\\ j&=1,\dots,p/2}}$};
    \node[large] at (M-2-6){$\bigl(b_{i,j}\bigr)_{\subalign{i&=p/2+1,\dots,p\\ j&=1,\dots,p/2}}$};
    \node[large] at (M-4-6){$\bigl(a_{i,j}\bigr)_{\subalign{i&=p/2+1,\dots,p\\ j&=p/2+1,\dots,p}}$};
        \node[large] at (M-4-2){$\bigl(b_{i,j}\bigr)_{\subalign{i&=p/2+1,\dots,p\\ j&=1,\dots,p/2}}$};
    \draw([shift={(-3mm,0mm)}]M-3-1.south west)--([xshift=2mm]M-3-7.south east);
    \draw(M-5-4.south east)--([shift={(0mm,2mm)}]M-1-4.north east);
  \end{tikzpicture}
\]
where $(a_{i,j})_{i,j}$ are independent observations from a Bernoulli distribution with parameter $p_{intra}$ and $(b_{i,j})_{i,j}$ are independent observations from a Ber\-nou\-lli distribution with parameter $p_{inter}$. In simulations, the probability to have an edge between nodes inside each communities is set to $p_{intra}=0.6$. Four values of the probability {of} connection between the two communities are considered, $p_{inter}\in\{0.01, 0.05, 0.15, 0.4\}$. The corresponding adjacency matrices are displayed in Figure~\ref{fig:adj}. Parameter $p_{inter}$ determines the degree of sparsity of the correlation matrix, where sparsity corresponds to the number of edges present in the graph with respect to the corresponding complete graph. High values of $p_{inter}$ means a low sparsity.

As explained above, we consider a constant value of non-zero correlations, $\rho$. The value must be chosen such that the resulting matrix $\mathbf M$ is positive definite. In our application, it implies that $\rho \leq 0.2$, independently of $p_{inter}$ and $p_{intra}$.

{The choice of maximal value of $\rho$ is linked to the density of the graph. Indeed, the denser the graph is, the smaller $\rho$ has to be chosen.} We choose the same values of $\rho$ here whatever the adjacency matrix is, to study more specifically the influence of the sparsity. As noticed {\it e.g.} by \cite{Kramer}, the recovery of the edges is in general more difficult when the sparsity of the graph decreases. The authors also remarked that all methods perform well with cluster networks. This observation motivates the analysis of the behavior of multiple testing approaches with respect to the sparsity.

\begin{figure}[!ht]
\centering
\caption{Adjacency matrices. The parameter $sparse$ is the proportion of edges present in the graph with respect to a complete graph. {$p_{intra}=0.6$.}}
\label{fig:adj}
   \begin{subfigure}[t]{0.4\textwidth}
        \caption{$p_\text{inter}=0.01$, $sparse$=25.4\%.}
        \includegraphics[height=5cm]{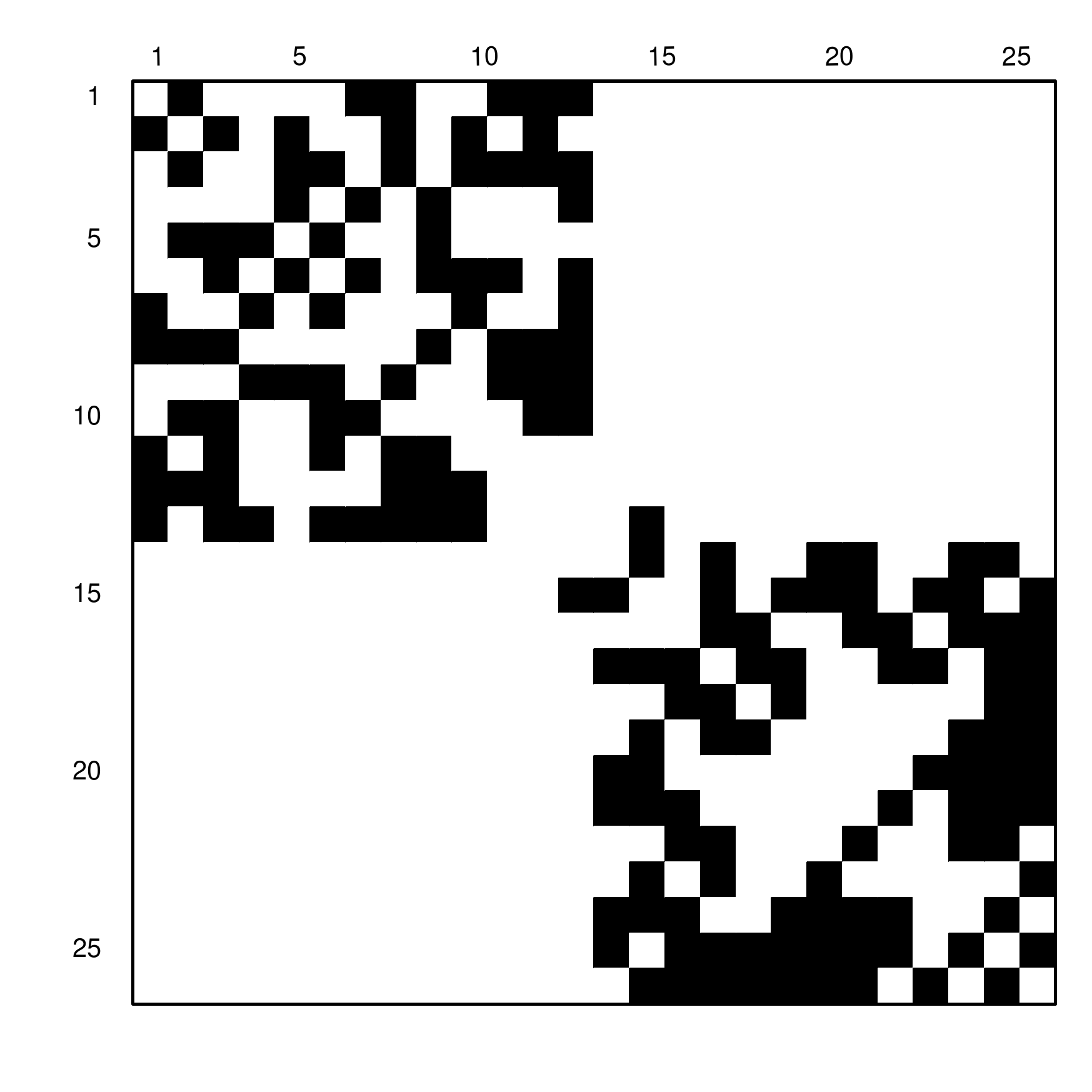}
        \end{subfigure}  
        \begin{subfigure}[t]{0.4\textwidth}
        \caption{$p_\text{inter}=0.05$, $sparse$=30\%.}
        \includegraphics[height=5cm]{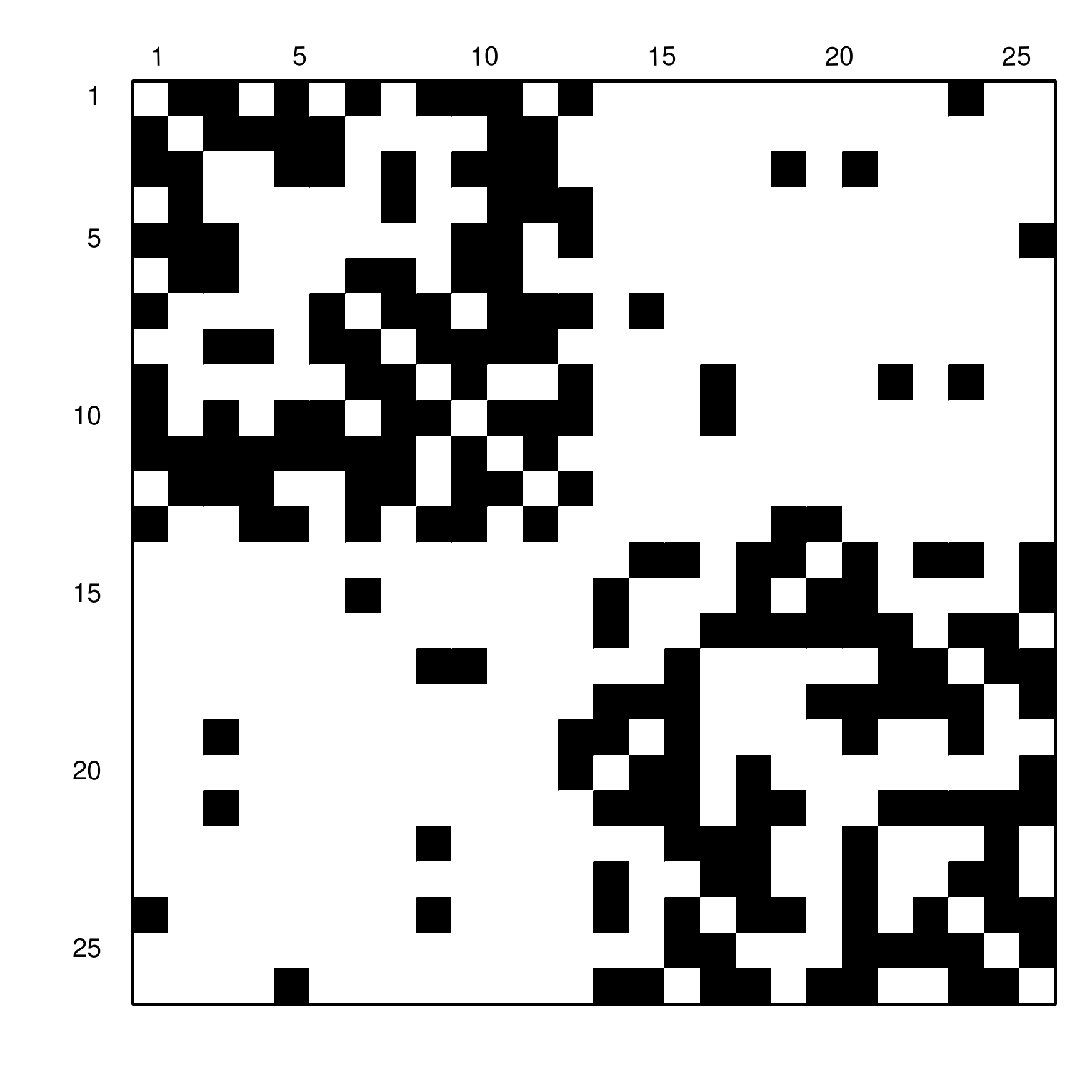}
        \end{subfigure}  \\
        \begin{subfigure}[t]{0.4\textwidth}
        \caption{$p_\text{inter}=0.15$, $sparse$=37\%.}
        \includegraphics[height=5cm]{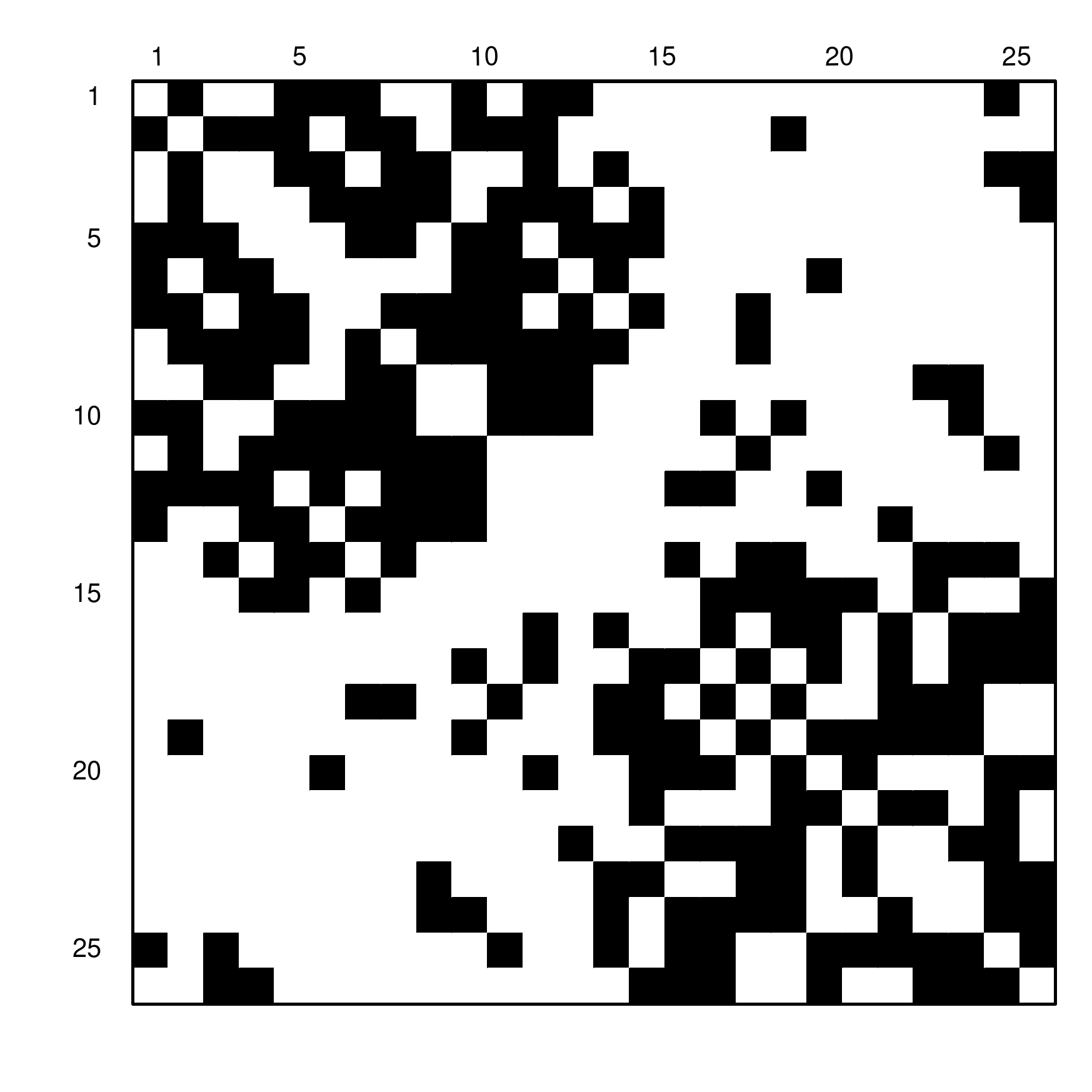}
        \end{subfigure}
                \begin{subfigure}[t]{0.4\textwidth}
        \caption{$p_\text{inter}=0.4$, $sparse$=47.6\%.}
        \includegraphics[height=5cm]{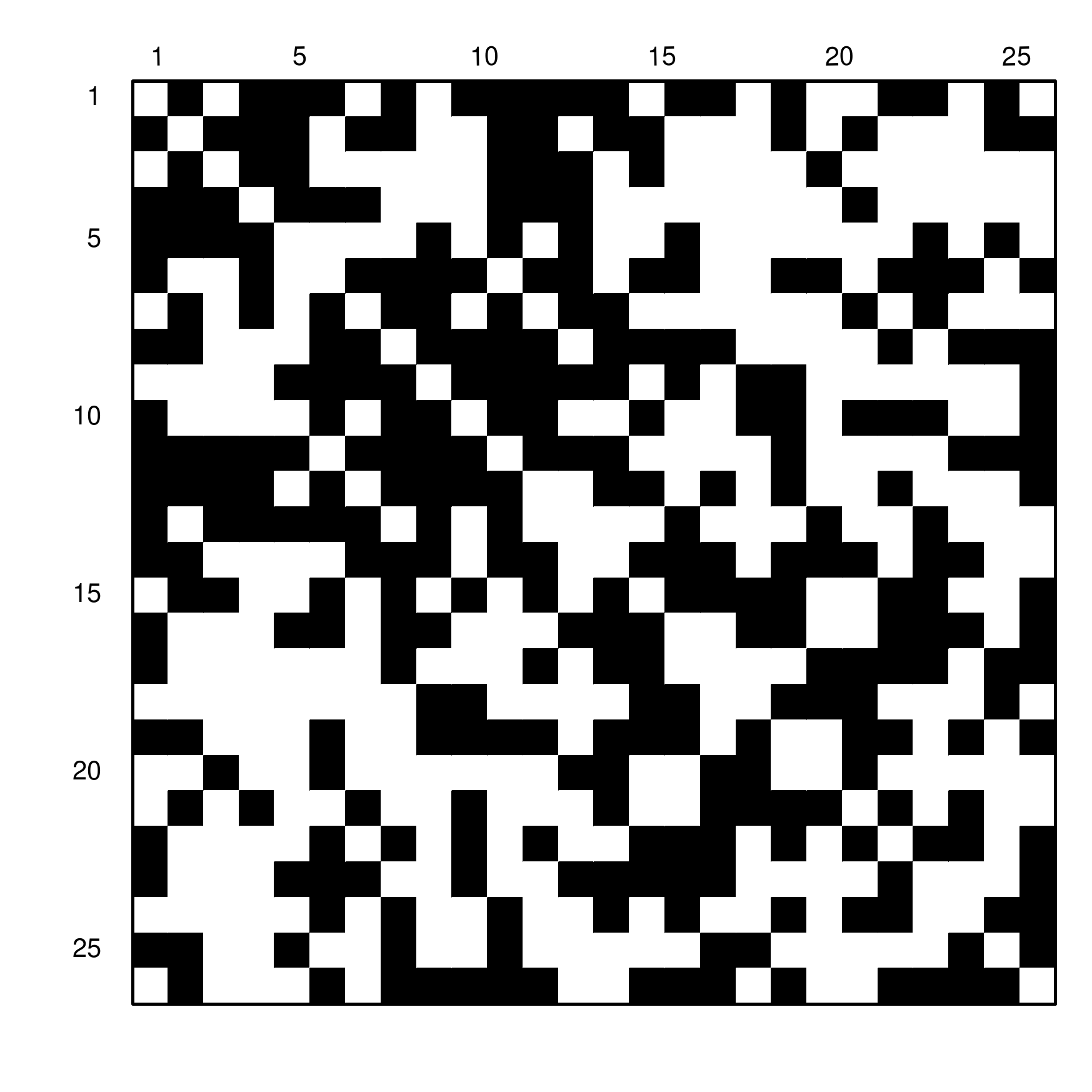}
        \end{subfigure}

\end{figure}

We simulate $Y_1, \dots Y_n$ realizations of a Gaussian random variable $\mathcal N(0,\mathbf M)$, for various values of sample size $n$. Edges are estimated applying correlation tests~\eqref{letest}. 325 tests are needed in this framework. The four statistics described above as the four methods are applied, that is, Methods 1--3 and Method 4'. The level for the FWER control is fixed to $\alpha=5\%$. The non parametric bootstrap and the parametric bootstrap requires to fix the number of samples used in the estimation of quantiles. For non parametric bootstrap, the number of samples is denoted as \BootRW and equal to 100. For parametric bootstrap, denoted \MaxT, the number of samples is fixed to 1000. Resulting computational cost is high but less samples leads to an unsatisfying quality.

Our simulation procedure depends on the three main parameters: $\rho$, $p_{inter}$, and $n$. $\rho$ corresponds to the value of entries in the correlation matrices. This is related to the signal to noise ratio for the identification of edges in the graph. Indeed, when $\rho$ is low, it is harder to detect the correct edges. In the sequel, the simulations are displayed for two values of $\rho$: 0.1 and 0.2. Recall that because of the constraint on our model, it is not possible to choose larger values of $\rho$. $p_{inter}$ is defined as the number of edges between the two components of the stochastic block model. This is controlling the sparsity of the graph and equivalently the number of null hypotheses. When $p_{inter}$ increases, the number of null hypotheses decreases. For the Bonferroni procedure, this has a direct impact on the FWER. Finally, $n$ is the number of independent realizations available to estimate the correlations. $n$ and $\rho$ are linked by our main assumption on the convergence of the estimator \eqref{hypGauss}.  In the two following parts, results of simulations applied on all four methods and using different parameters in the model are given for both FWER and power.  

\subsection{Results on FWER control}

The difference between $\rho=0.2$ and $\rho=0.1$ is illustrated on Figure~\ref{fig:fwer2_empirical} and Figure~\ref{fig:fwer1_empirical}. It is clear that when $\rho=0.1$, all the proposed approaches provide empirical FWER much lower than the theoretical value fixed at 0.05 for any sample sizes and any sparsity.    \\

The influence of the sample size $n$ relative to the use of a specific statistics is displayed on Figure~\ref{fig:fwer2_empirical} for $(T^{(1)}_{ij})_{1\leq i<j\leq p}$, Figure~\ref{fig:fwer2_student} for $(T^{(2)}_{ij})_{1\leq i<j\leq p}$, Figure~\ref{fig:fwer2_fisher} for $(T^{(3)}_{ij})_{1\leq i<j\leq p}$ and Figure~\ref{fig:fwer2_gaussian} for $(T^{(4)}_{ij})_{1\leq i<j\leq p}$. Depending on the chosen statistics, for small sample sizes the FWER may not be controlled: this is due to the initial test procedure rather than the multiple correction. Indeed,  tests \eqref{eqn:tests} are based on the normality of the statistics $(T^{(k)}_{ij})_{1\leq i<j\leq p}$, $k=1,\dots,4$. As the Gaussian distribution is only satisfied asymptotically, the FWER control is valid only for sufficiently large values of $n$. The convergence of the Gaussian statistics $(T^{(4)}_{ij})_{1\leq i<j\leq p}$ is the  slowest, and there is no control of FWER for $n\leq 200$. However, for the Student statistics $(T^{(2)}_{ij})_{1\leq i<j\leq p}$ the control of FWER is valid for $n\geq 200$. With the Fisher statistics $(T^{(2)}_{ij})_{1\leq i<j\leq p}$ only the case $n=50$ is not controlled, while the empirical statistics $(T^{(1)}_{ij})_{1\leq i<j\leq p}$ behaves well even for small sample sizes. 

{The advantage of using \BootRW~procedure lies in the fact that it does require a Gaussian distribution.} As the quantile is evaluated on bootstrap resamples of the statistics, it corresponds to the effective distribution and not to the theoretical one. As a consequence, the FWER is controlled for all tests statistics, whatever the sample size but this comes with an increase of the computation time.

The comparison of step-down procedure with single step is also illustrated on the Figures~\ref{fig:fwer1_empirical} to \ref{fig:fwer2_gaussian}. In particular, Figures~\ref{fig:fwer2_student},~\ref{fig:fwer2_fisher} and~\ref{fig:fwer2_gaussian} show that the control for step-down \MaxT~is not always obtained. This results from numerical instability of the quantile estimation. Indeed, as discussed in Section~\ref{sec:simu-model}, the number of bootstrap samples is critical. The numerical approximation is more critical for step-down procedures. It is shown in particular for sparse models and large sample sizes.

The empirical FWER slightly increases with the sparsity of the matrix. This is due to the link between the sparsity and the number of null hypotheses $m_0(P)$. The FWER level of multiple testing procedures usually increases with $m_0(P)$. For example, for Bonferroni procedure, the FWER is in fact controlled at level $\alpha\frac{m_0(P)}{m}$.

\begin{figure}[!ht]

\centering

\caption{Empirical FWER on 10000 simulations, with respect to the sample size $n$, for empirical statistics, $T^{(1)}$. The four sparsity frameworks are considered. $\rho=0.2$.}
  \label{fig:fwer2_empirical}
    \begin{subfigure}[t]{0.3\textwidth}
        \caption{Bonferroni}
        \includegraphics[height=4cm]{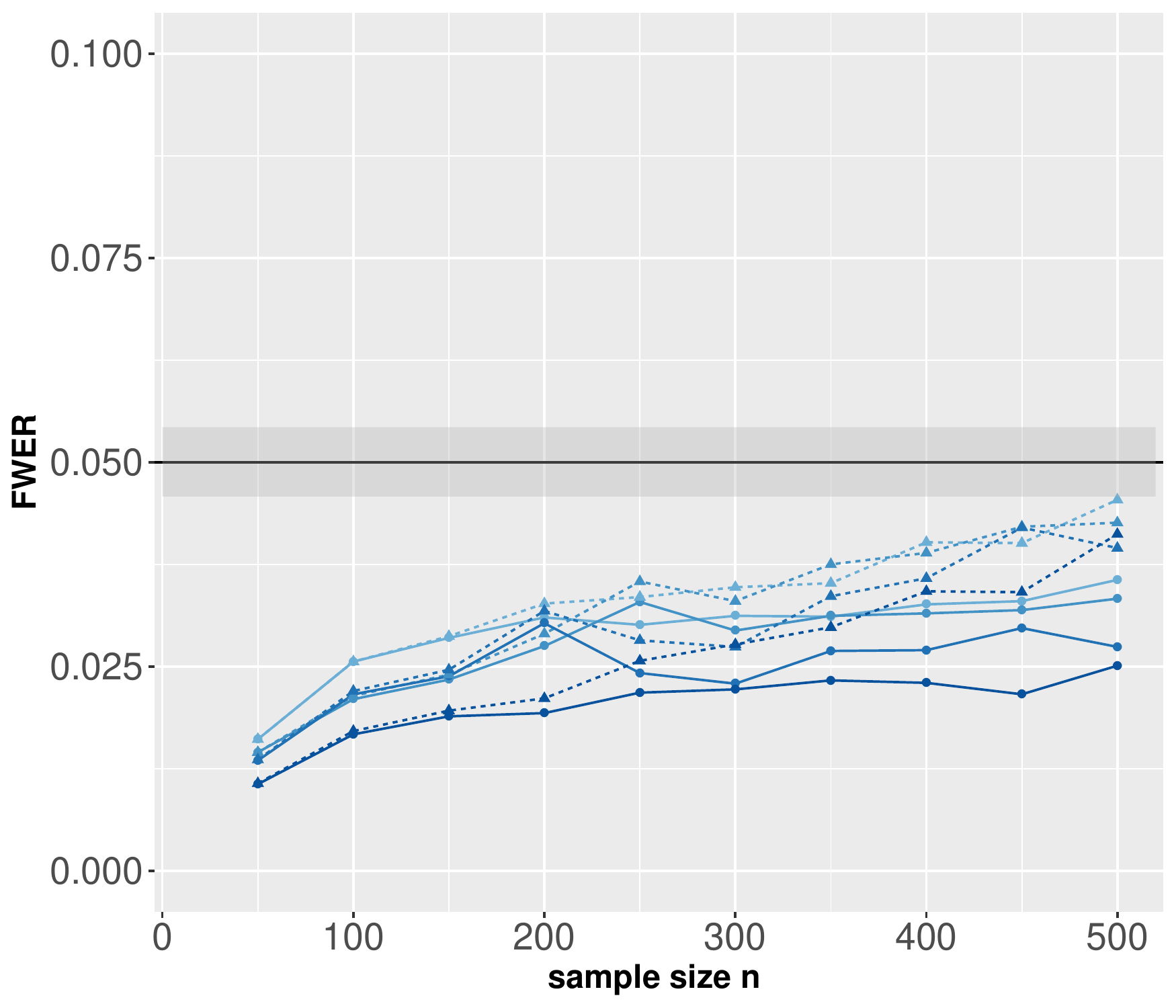}
        \end{subfigure}  
        \begin{subfigure}[t]{0.3\textwidth}
        \caption{\Sidak}
        \includegraphics[height=4cm]{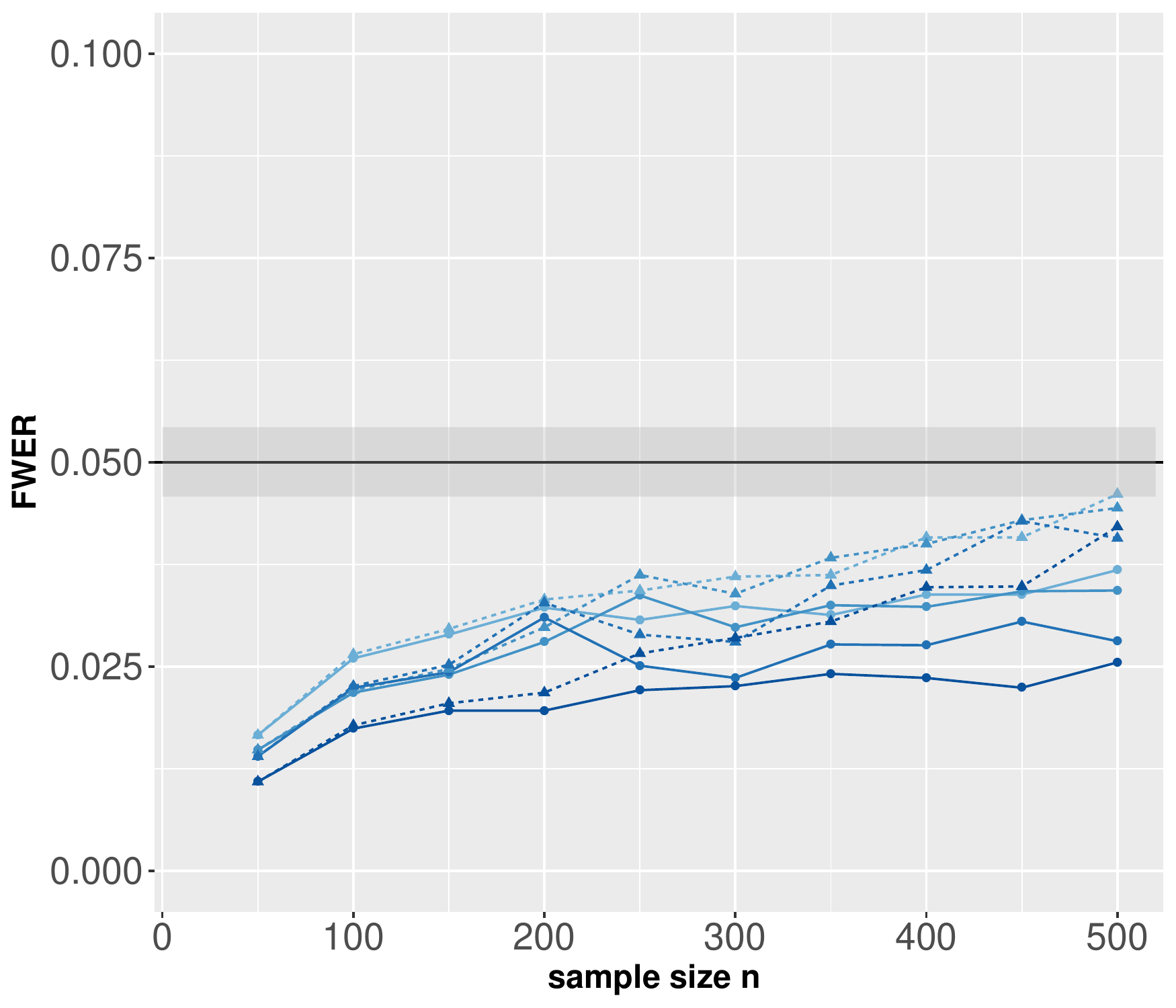}
        \end{subfigure} 
    \begin{subfigure}[t]{0.33\textwidth}
        \caption{\BootRW}
        \includegraphics[height=4cm]{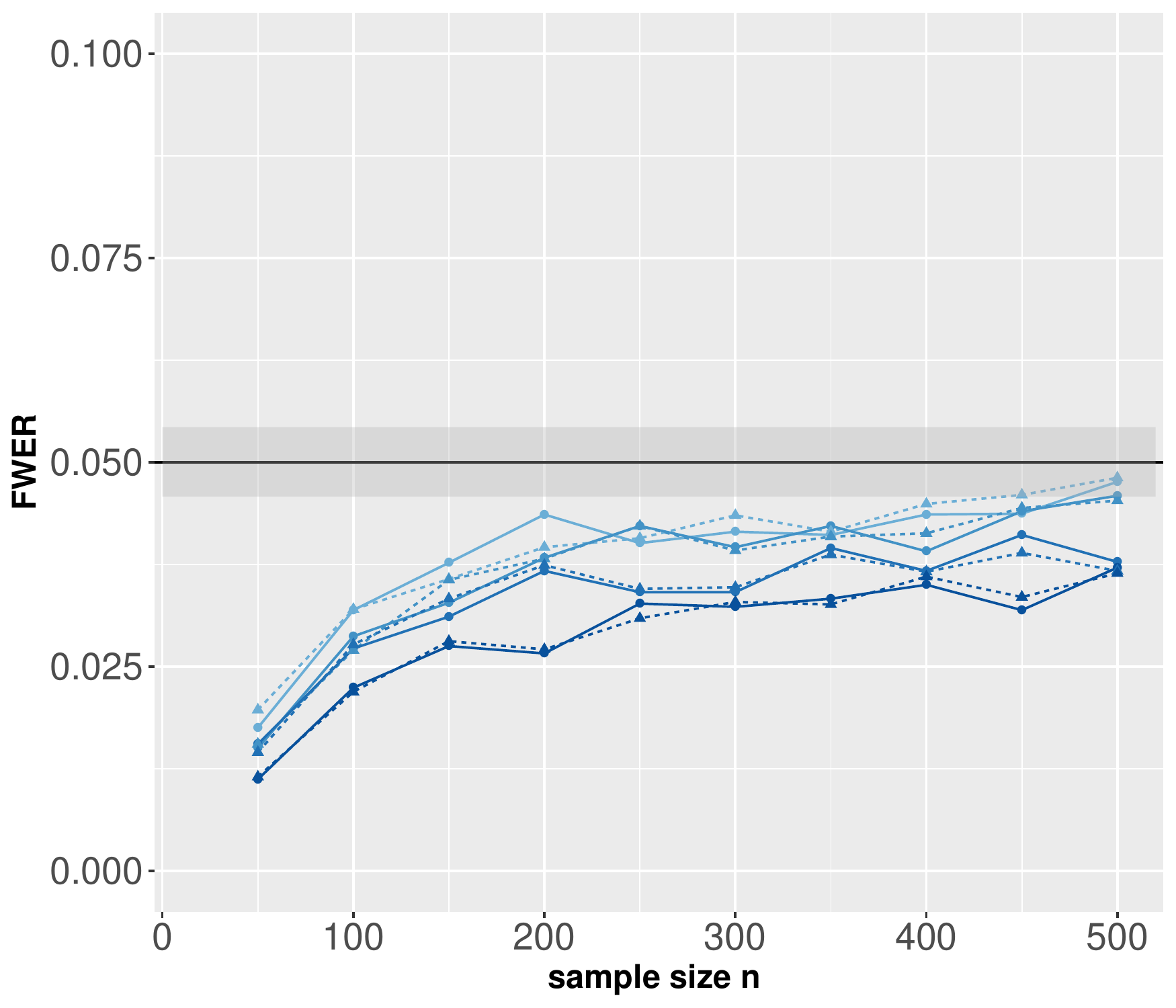}
        \end{subfigure}  \\
        \begin{subfigure}[t]{0.3\textwidth}
        \caption{\MaxT}
        \includegraphics[height=4cm]{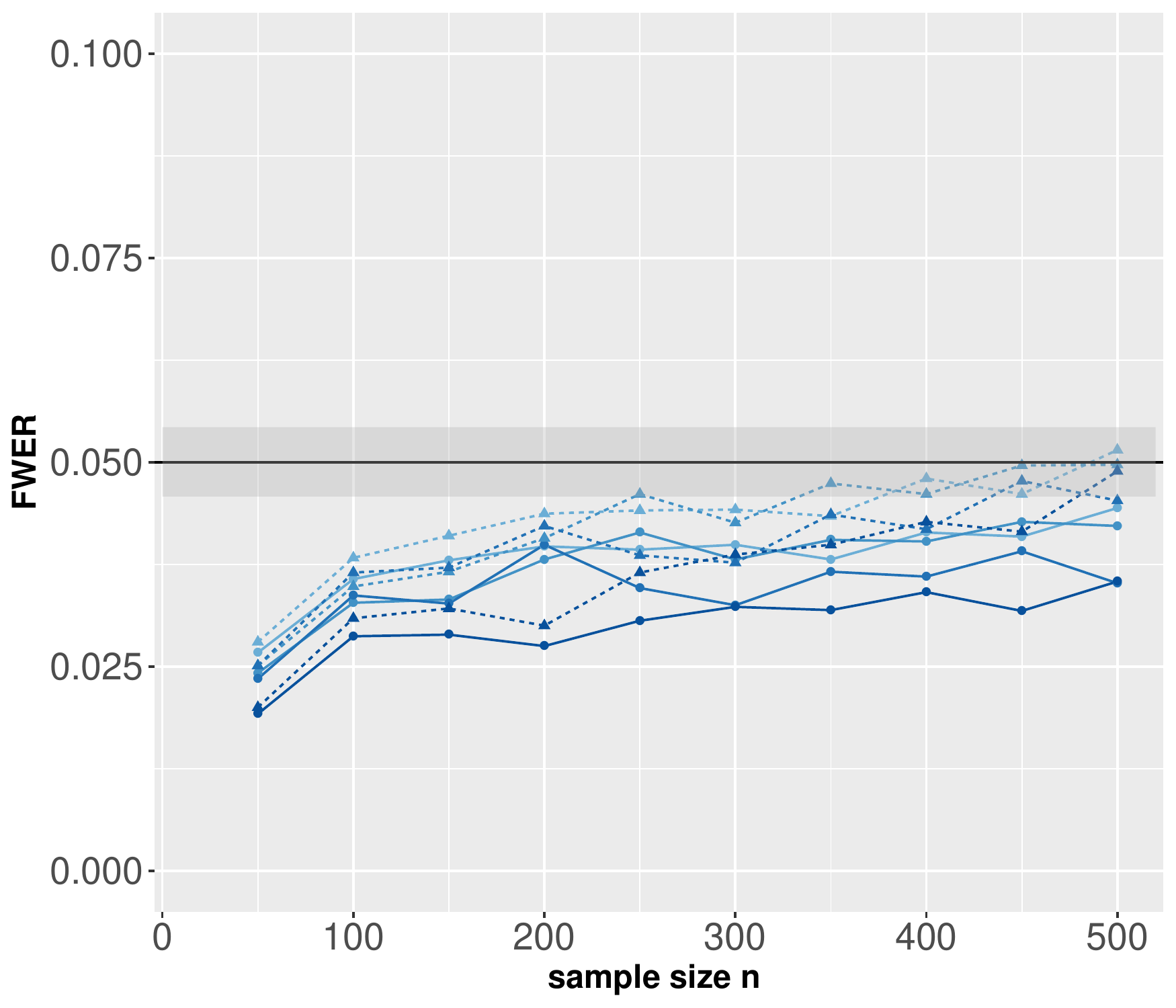}
        \end{subfigure}
                \begin{subfigure}[t]{0.33\textwidth}
        \caption{Oracle \MaxT}
        \includegraphics[height=4cm]{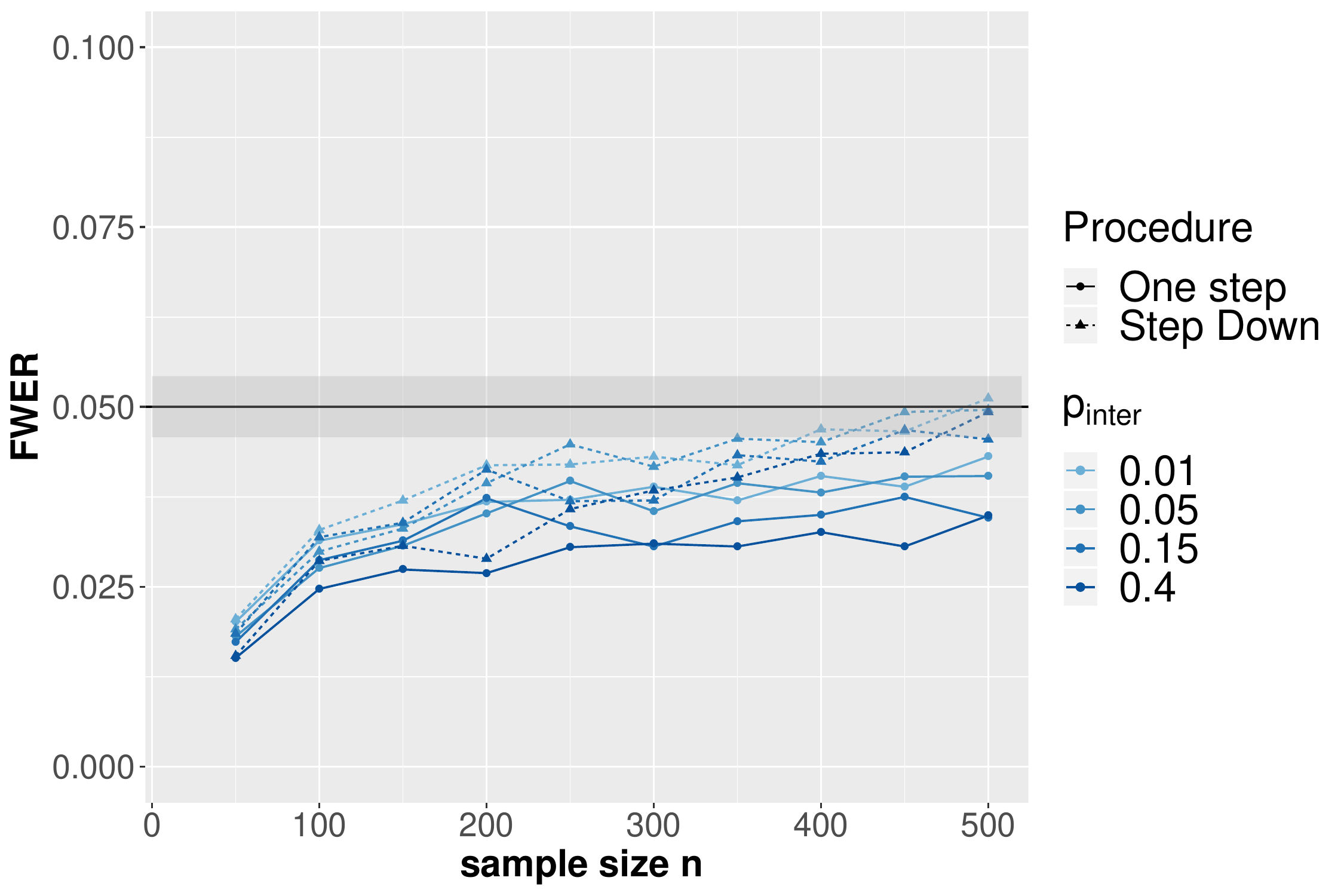}
        \end{subfigure}

\end{figure}

\begin{figure}

\centering

\caption{Empirical FWER on 10000 simulations, with respect to the sample size $n$, for empirical statistics, $T^{(1)}$. The four sparsity frameworks are considered. $\rho=0.1$.}
  \label{fig:fwer1_empirical}
    \begin{subfigure}[t]{0.3\textwidth}
        \caption{Bonferroni}
        \includegraphics[height=4cm]{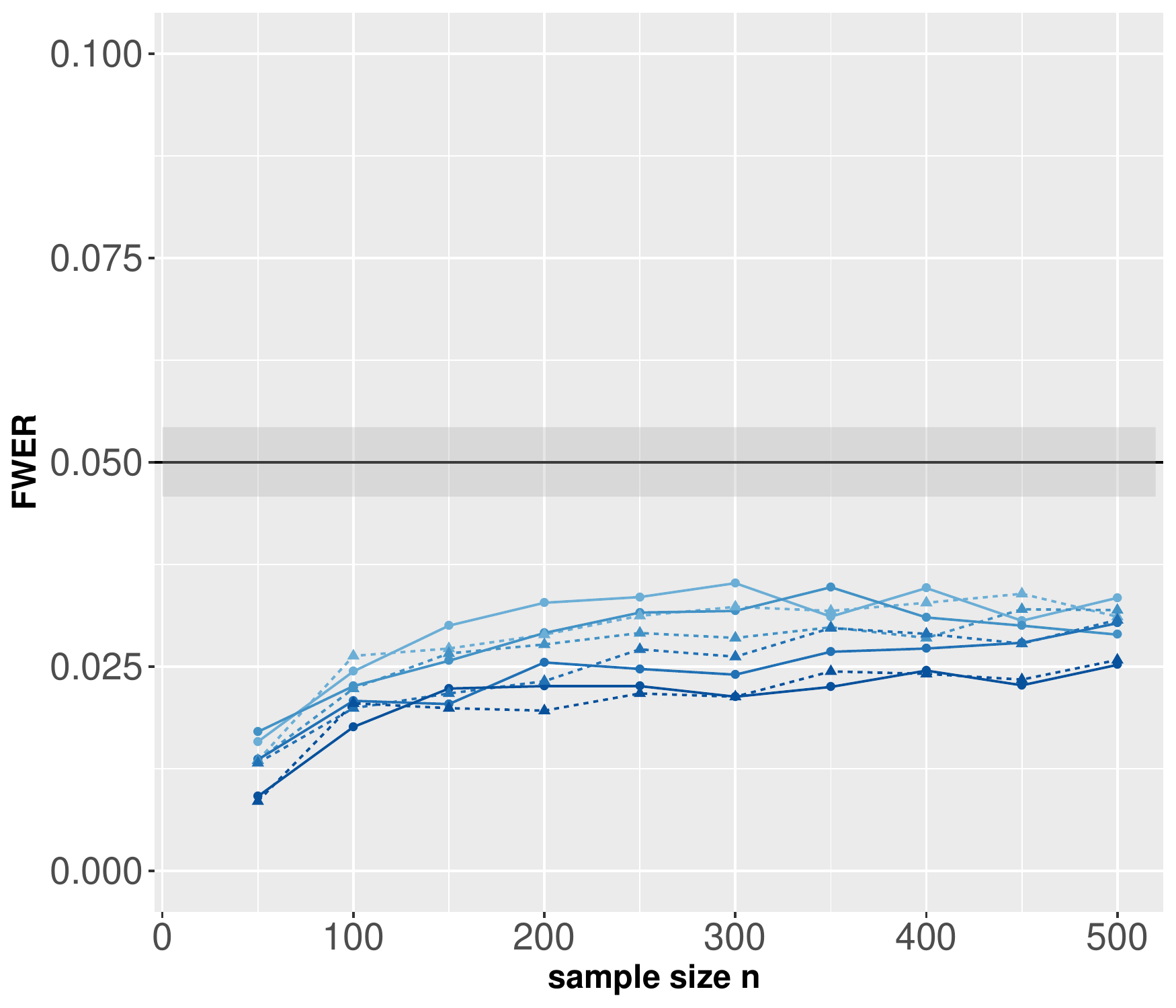}
        \end{subfigure}  
        \begin{subfigure}[t]{0.3\textwidth}
        \caption{\Sidak}
        \includegraphics[height=4cm]{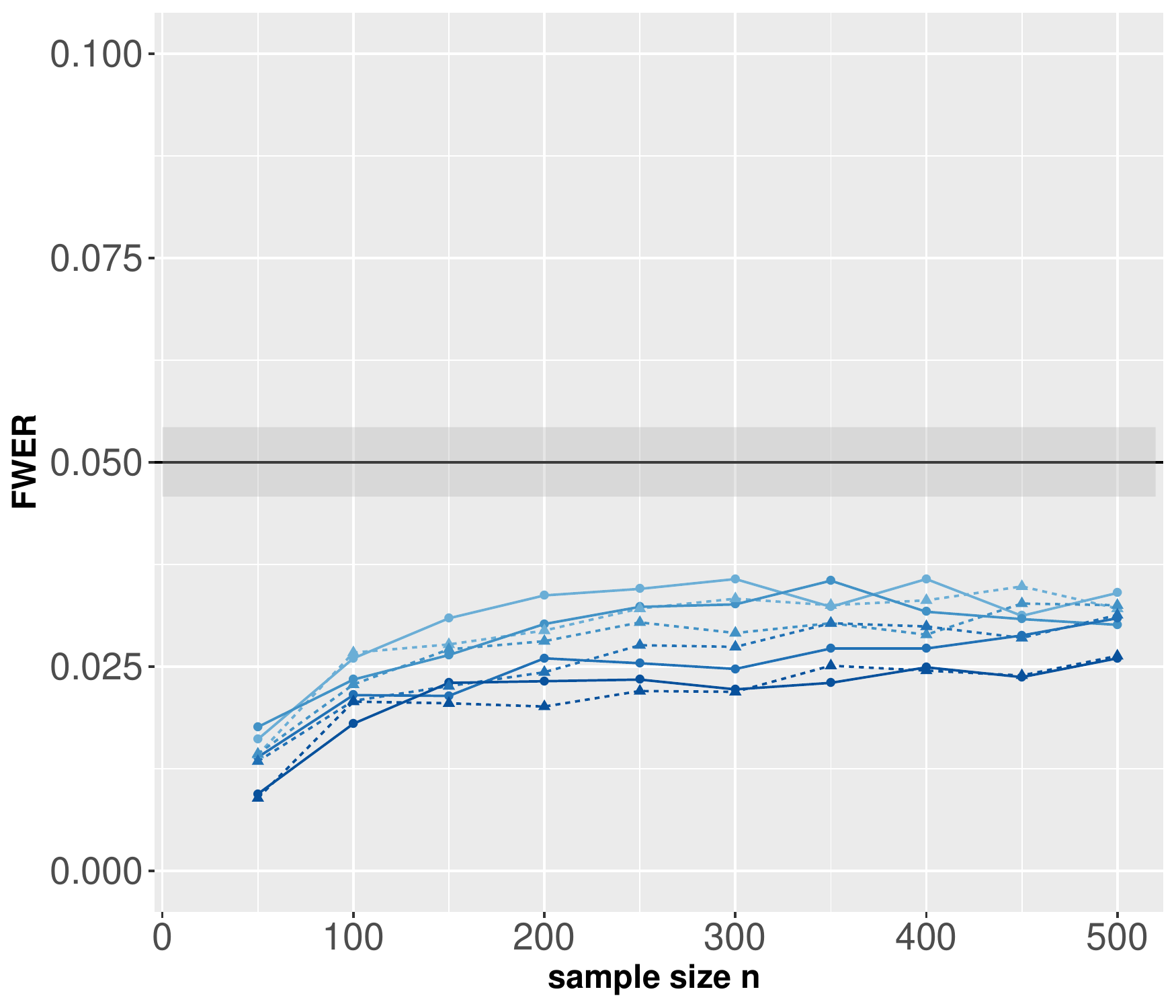}
        \end{subfigure} 
    \begin{subfigure}[t]{0.32\textwidth}
        \caption{\BootRW}
        \includegraphics[height=4cm]{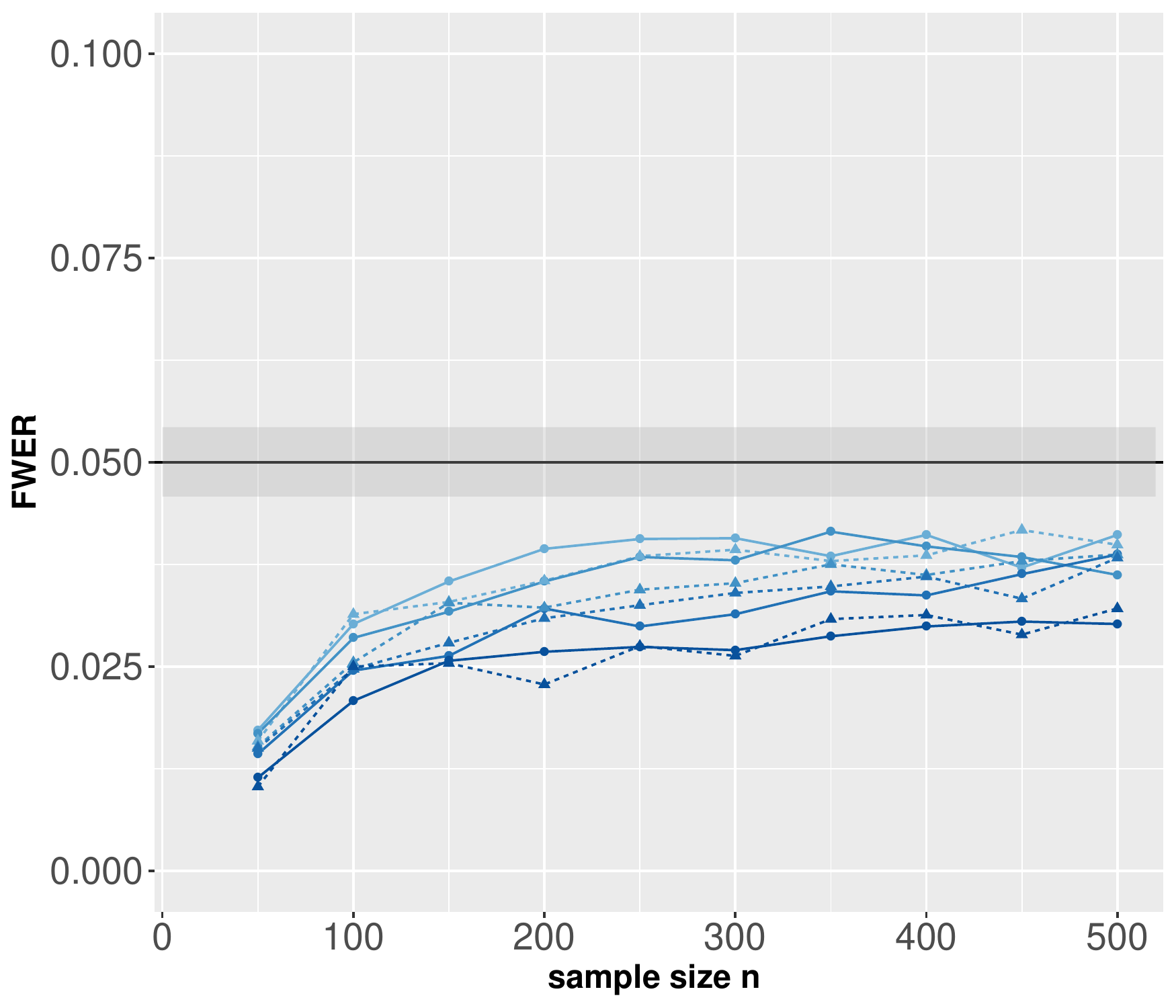}
        \end{subfigure}  \\
        \begin{subfigure}[t]{0.3\textwidth}
        \caption{\MaxT}
        \includegraphics[height=4cm]{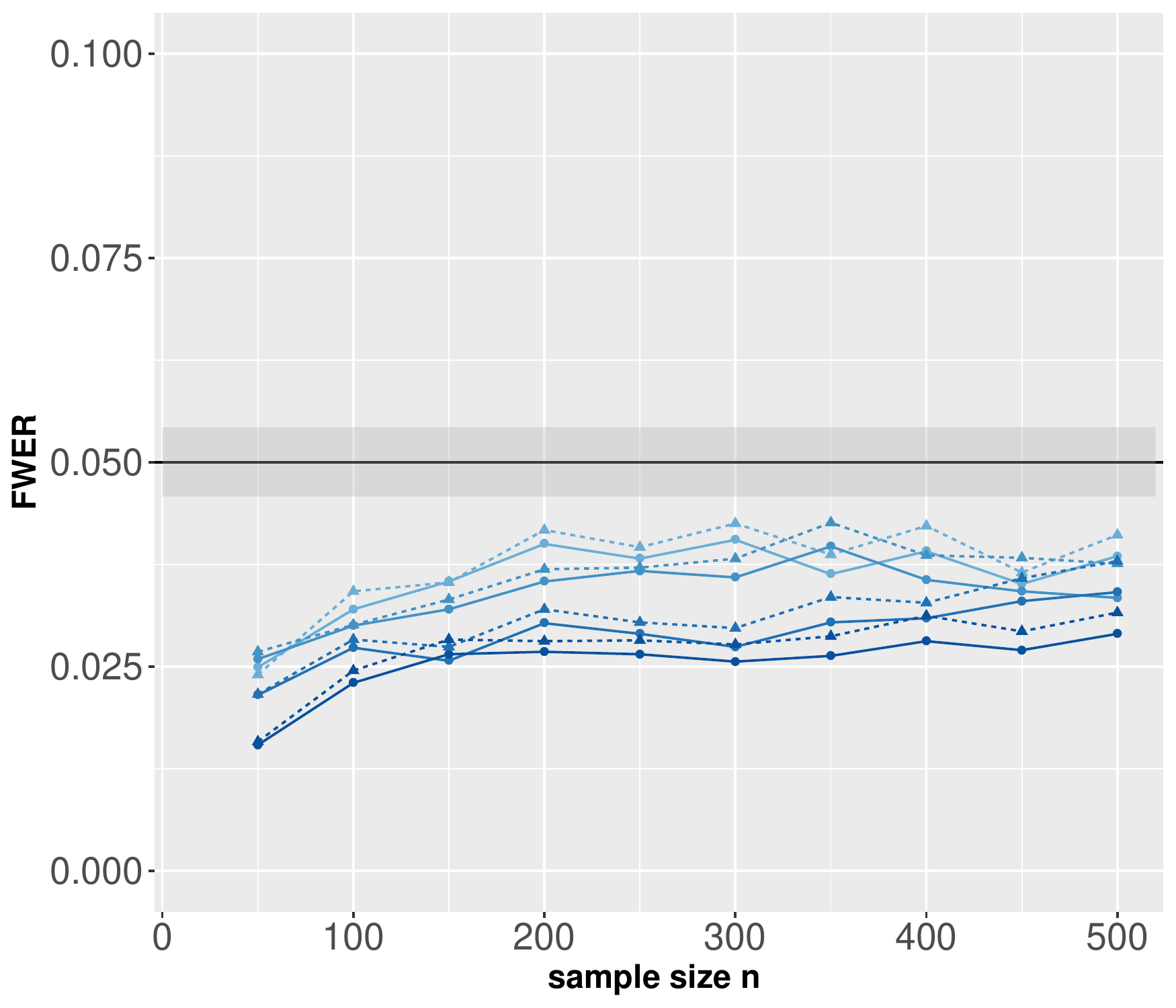}
        \end{subfigure}
                \begin{subfigure}[t]{0.35\textwidth}
        \caption{Oracle \MaxT}
        \includegraphics[height=4cm]{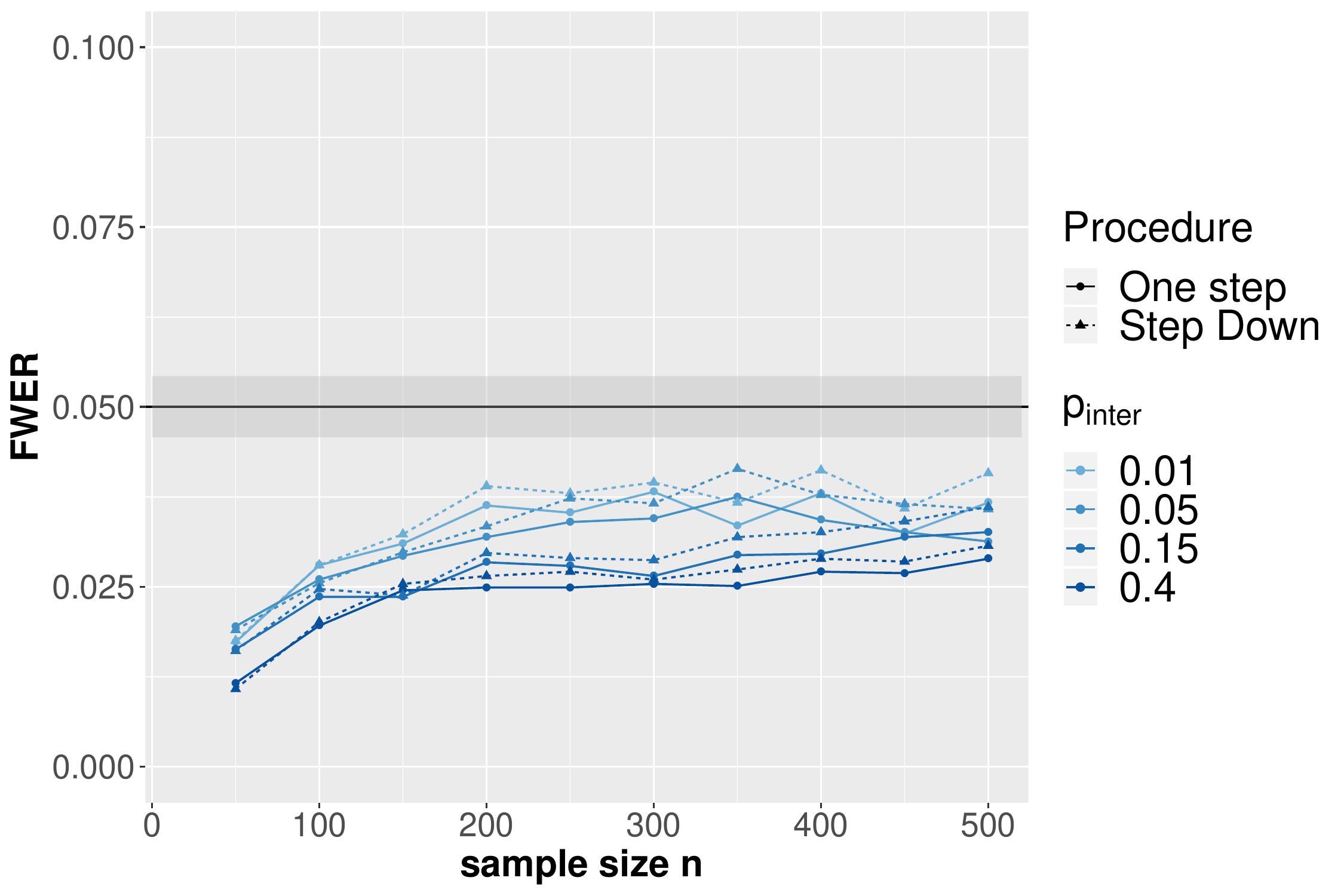}
        \end{subfigure}

\caption{Empirical FWER on 10000 simulations, with respect to the sample size $n$, for Student statistics, $T^{(2)}$. The four sparsity frameworks are considered. $\rho=0.2$.}
  \label{fig:fwer2_student}
    \begin{subfigure}[t]{0.3\textwidth}
        \caption{Bonferroni}
        \includegraphics[height=4cm]{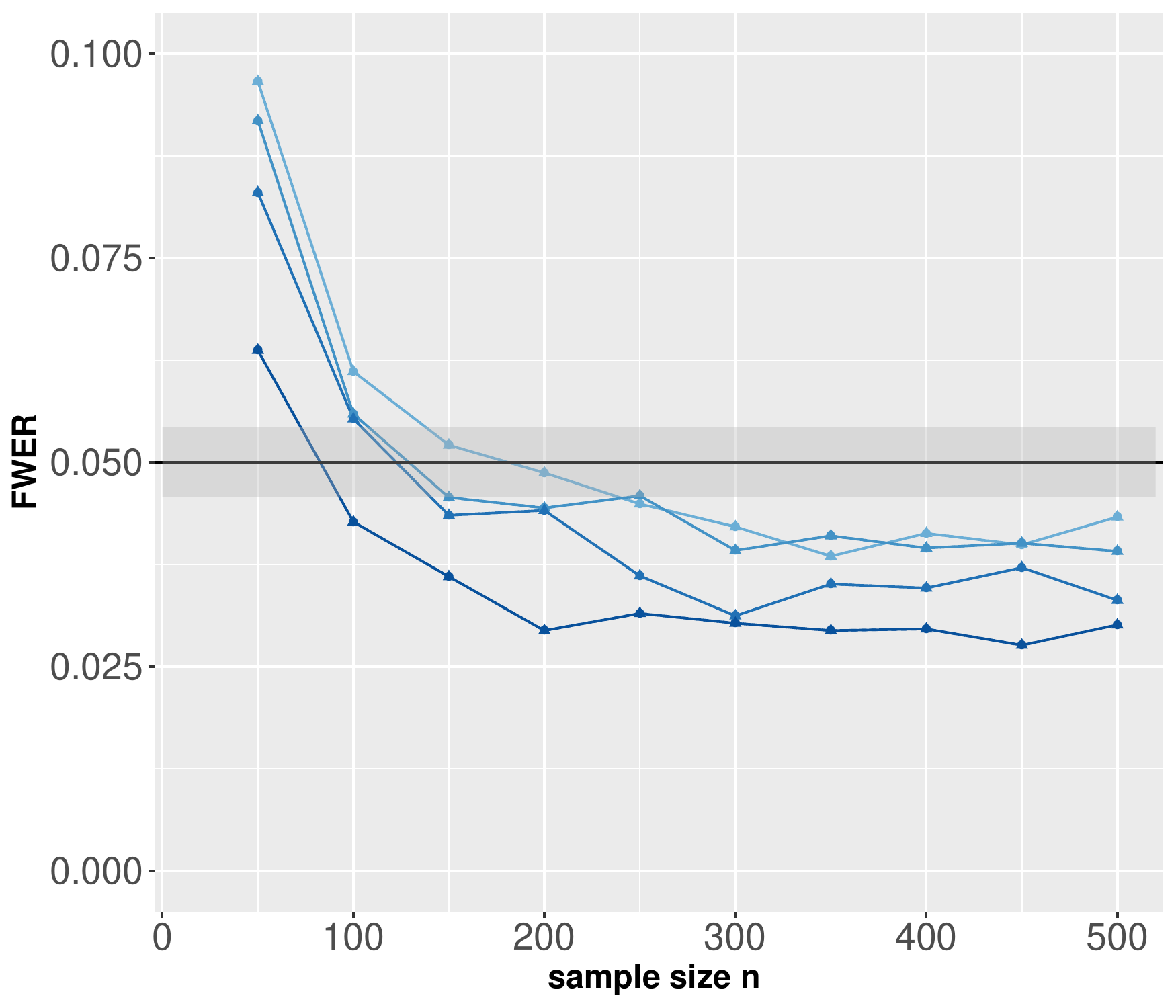}
        \end{subfigure}  
        \begin{subfigure}[t]{0.3\textwidth}
        \caption{\Sidak}
        \includegraphics[height=4cm]{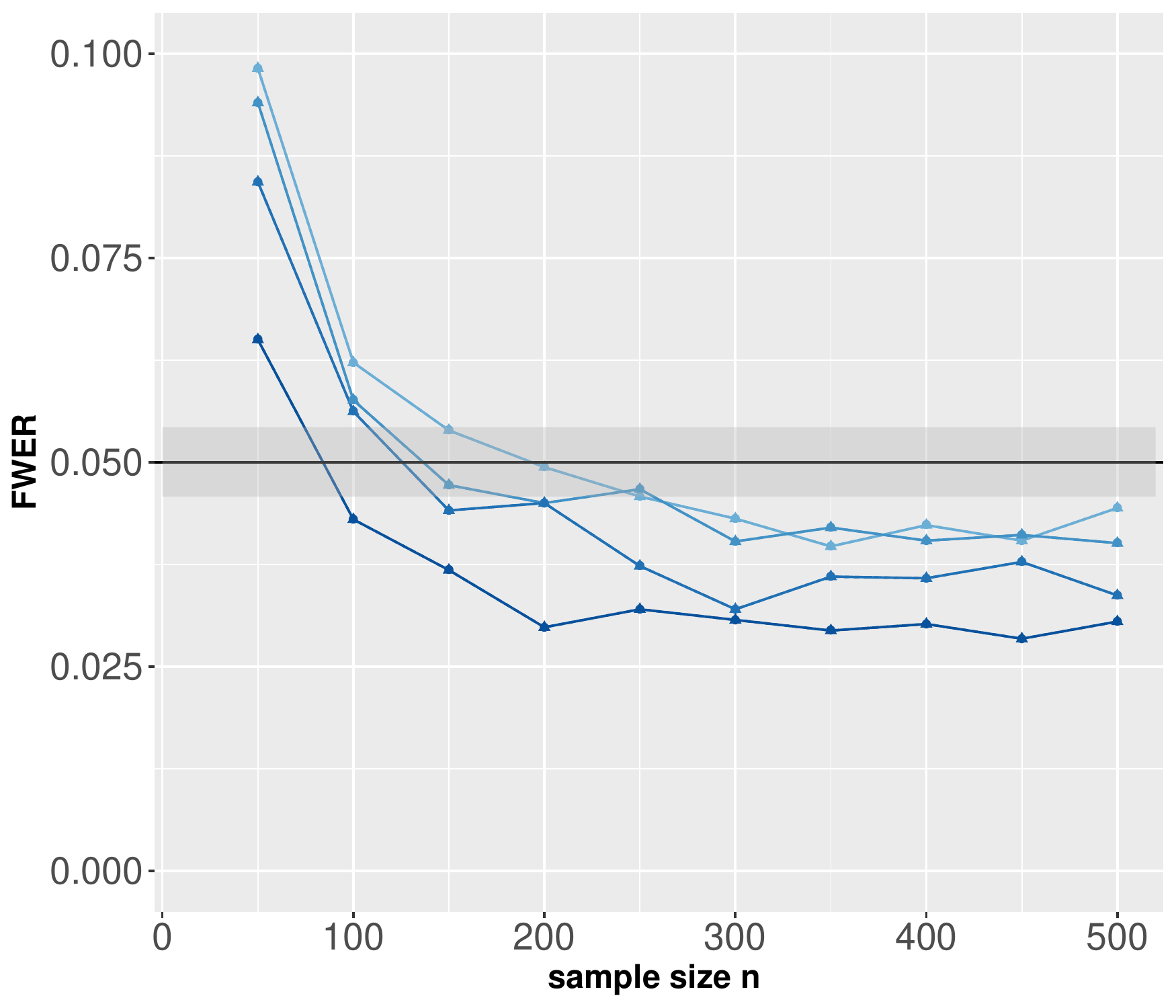}
        \end{subfigure} 
    \begin{subfigure}[t]{0.32\textwidth}
        \caption{\BootRW}
        \includegraphics[height=4cm]{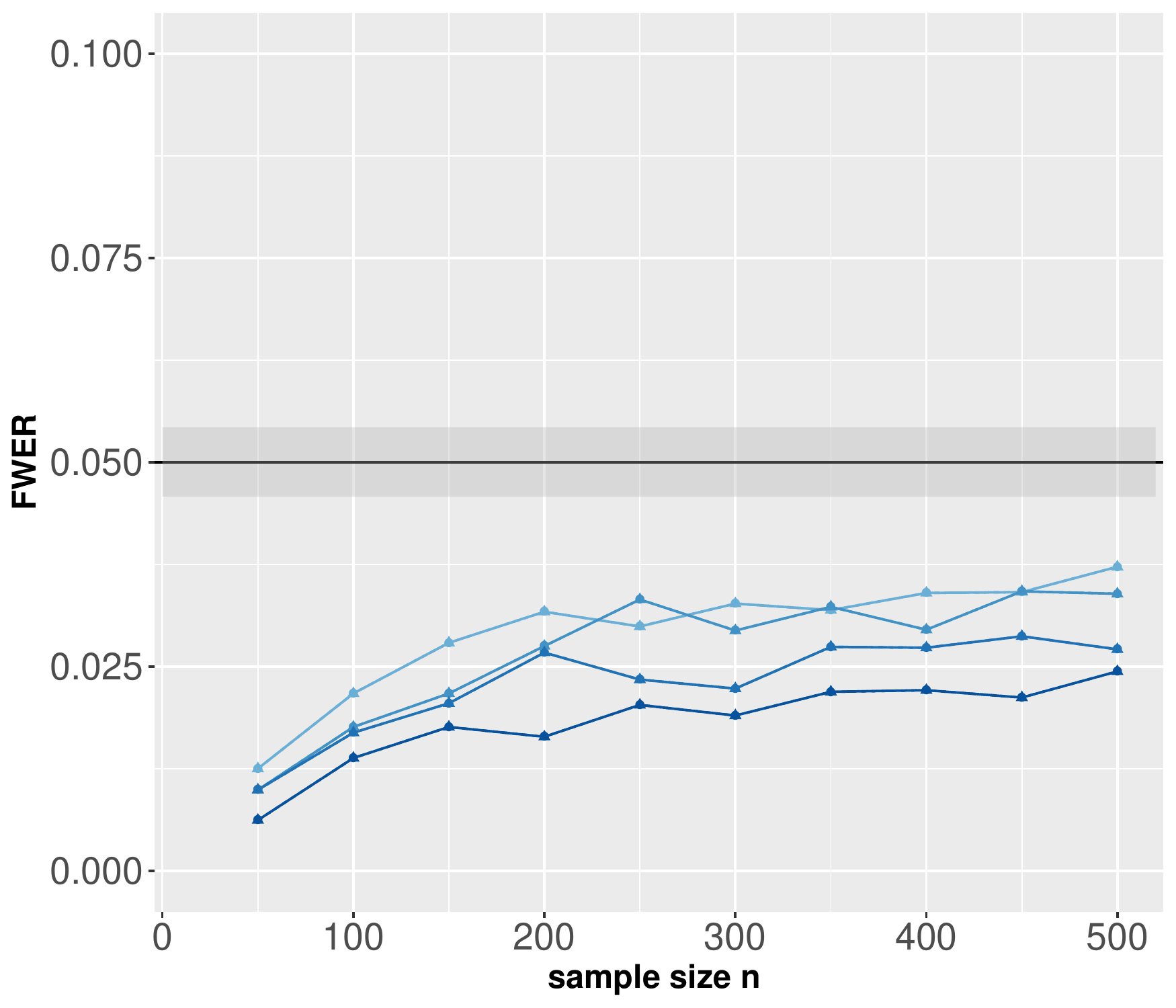}
        \end{subfigure}  \\
        \begin{subfigure}[t]{0.3\textwidth}
        \caption{\MaxT}
        \includegraphics[height=4cm]{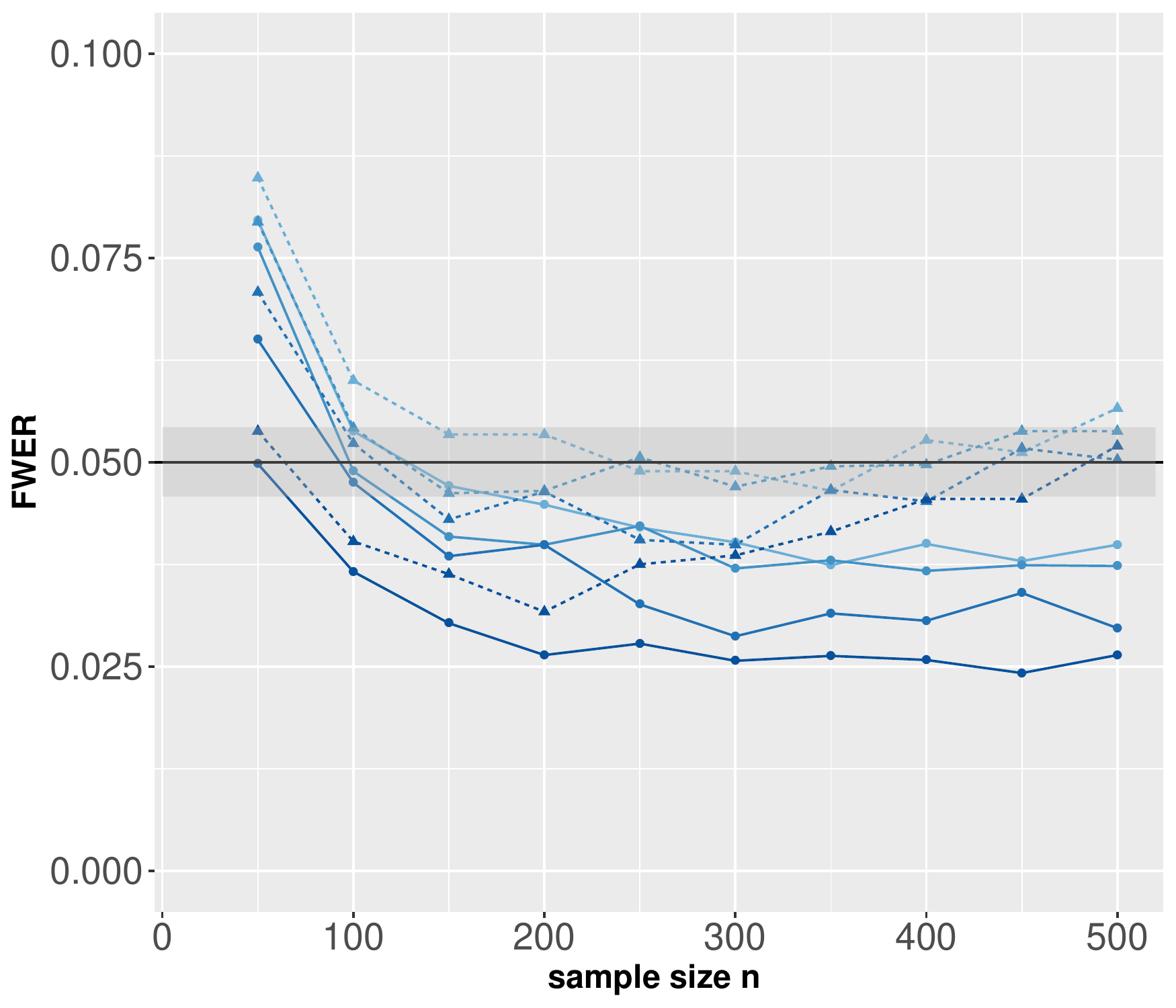}
        \end{subfigure}
                \begin{subfigure}[t]{0.32\textwidth}
        \caption{Oracle \MaxT}
        \includegraphics[height=4cm]{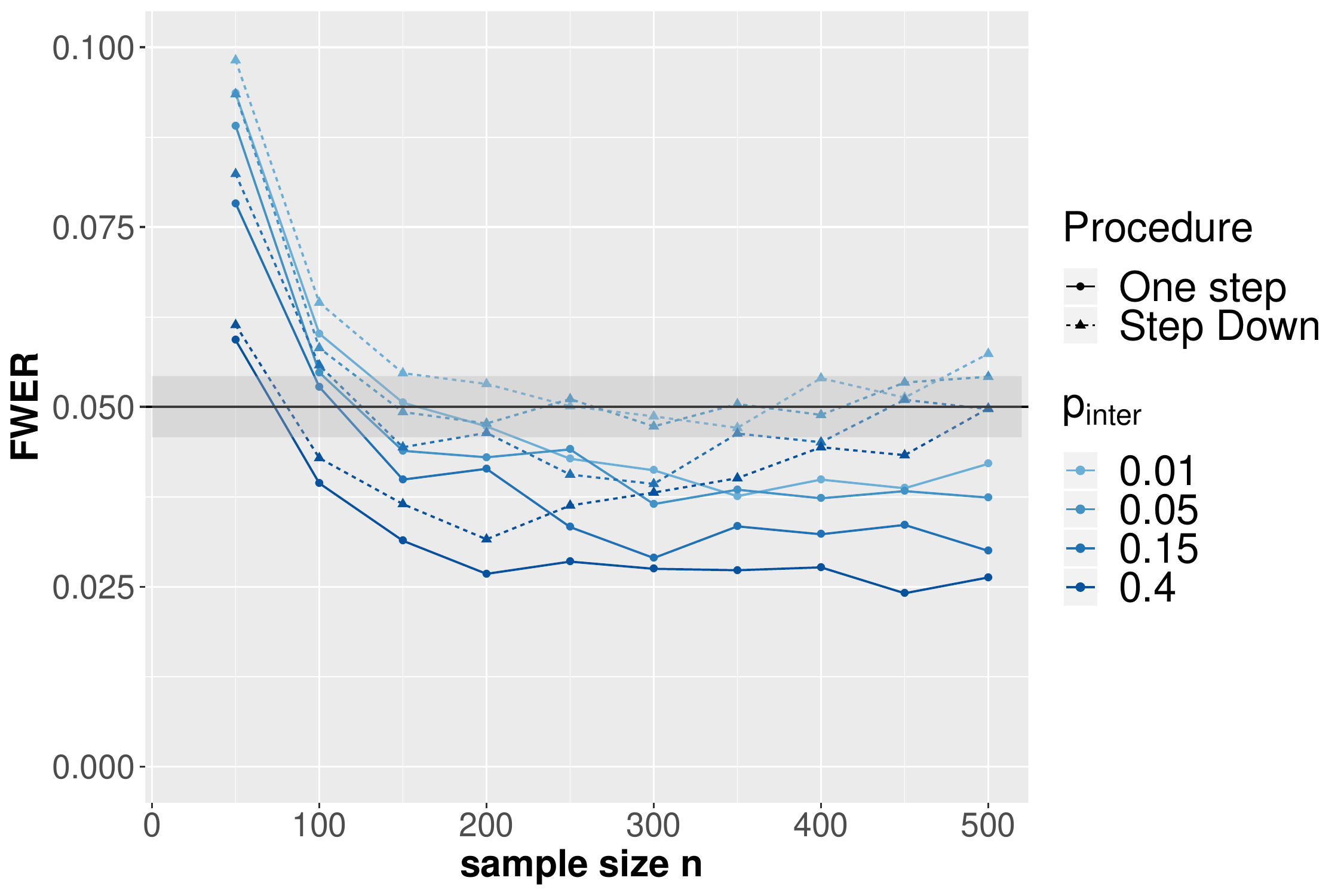}
        \end{subfigure}

\end{figure}

\begin{figure}
\centering

\caption{Empirical FWER on 10000 simulations, with respect to the sample size $n$, for Fisher statistics, $T^{(3)}$. The four sparsity frameworks are considered. $\rho=0.2$.}
  \label{fig:fwer2_fisher}
    \begin{subfigure}[t]{0.3\textwidth}
        \caption{Bonferroni}
        \includegraphics[height=4cm]{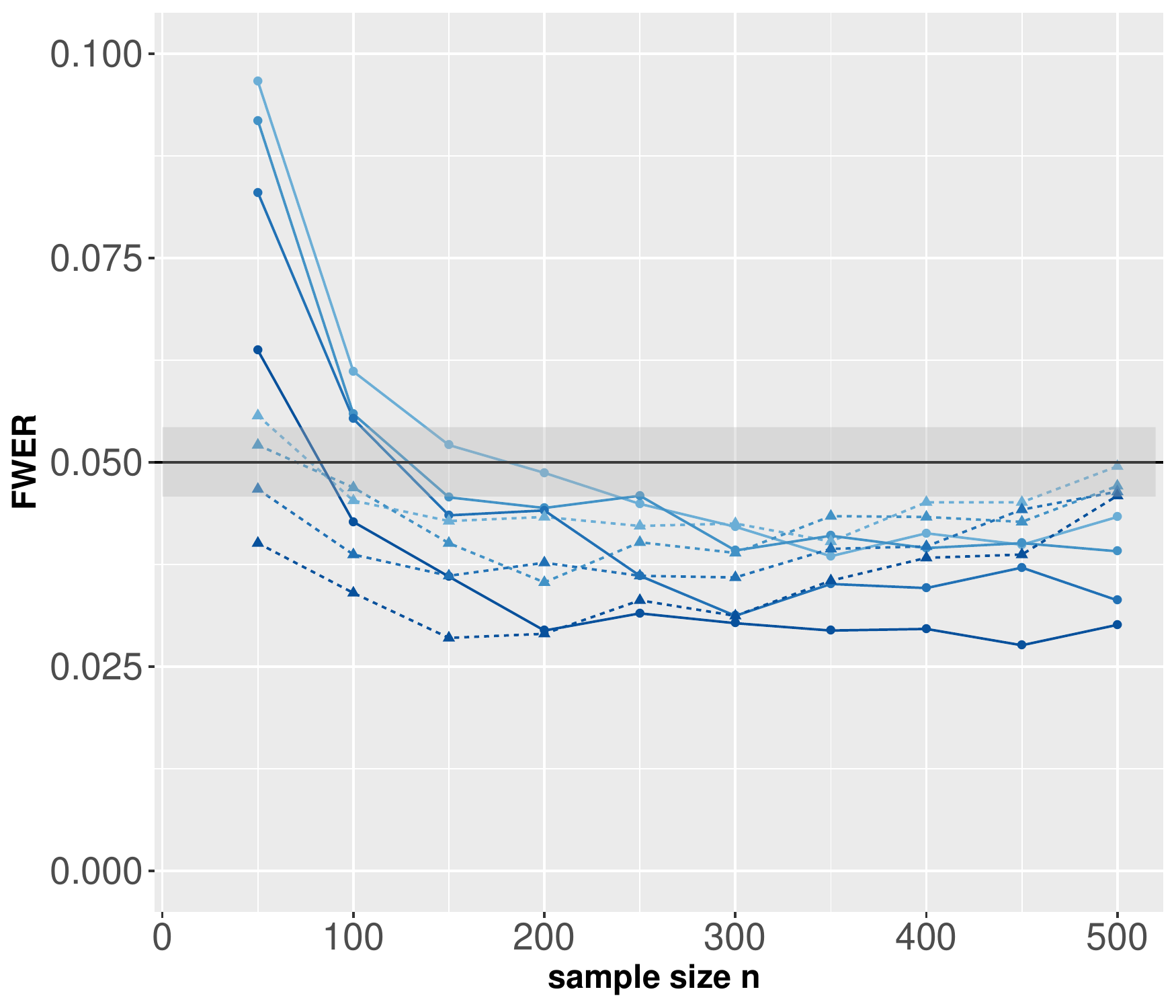}
        \end{subfigure}  
        \begin{subfigure}[t]{0.3\textwidth}
        \caption{\Sidak}
        \includegraphics[height=4cm]{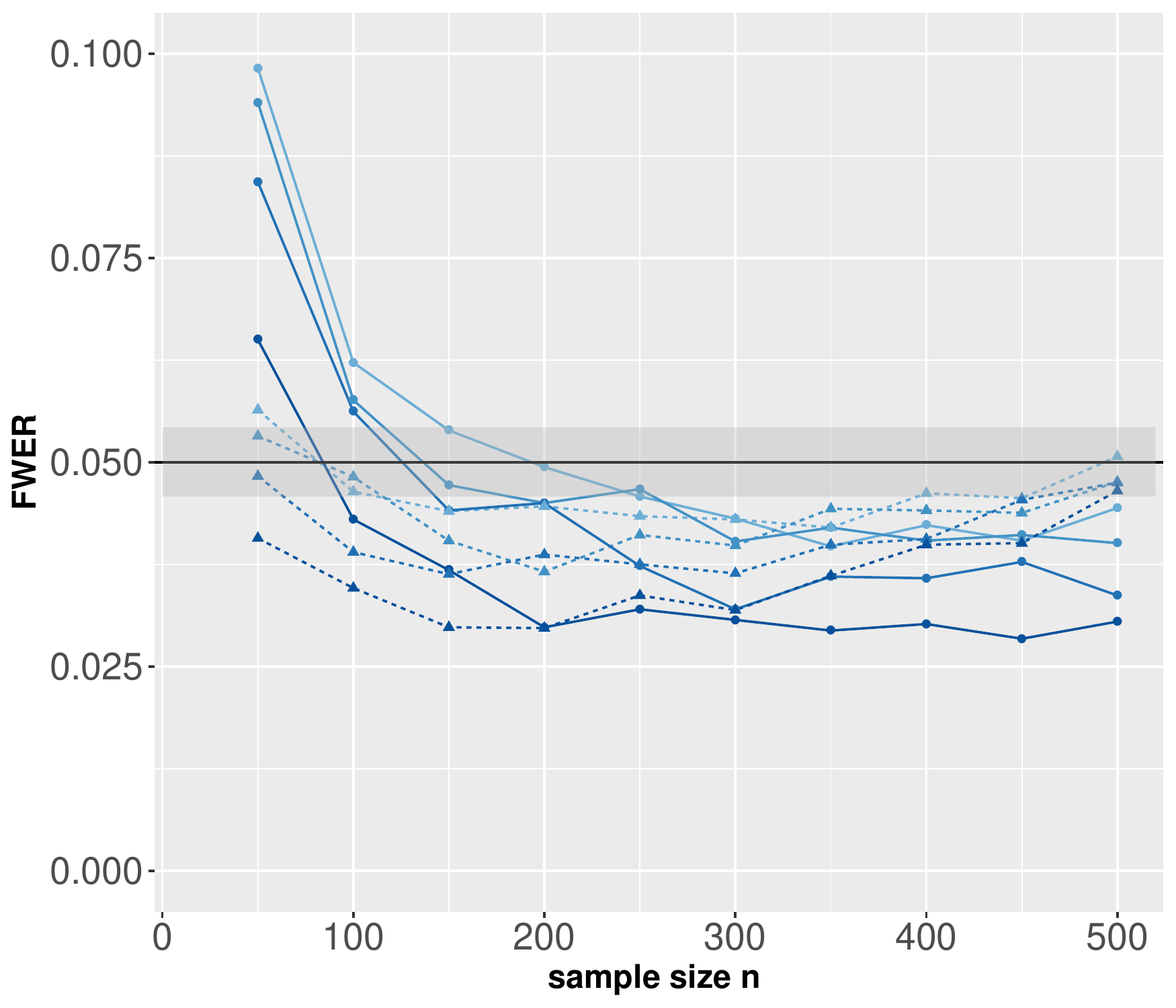}
        \end{subfigure} 
    \begin{subfigure}[t]{0.32\textwidth}
        \caption{\BootRW}
        \includegraphics[height=4cm]{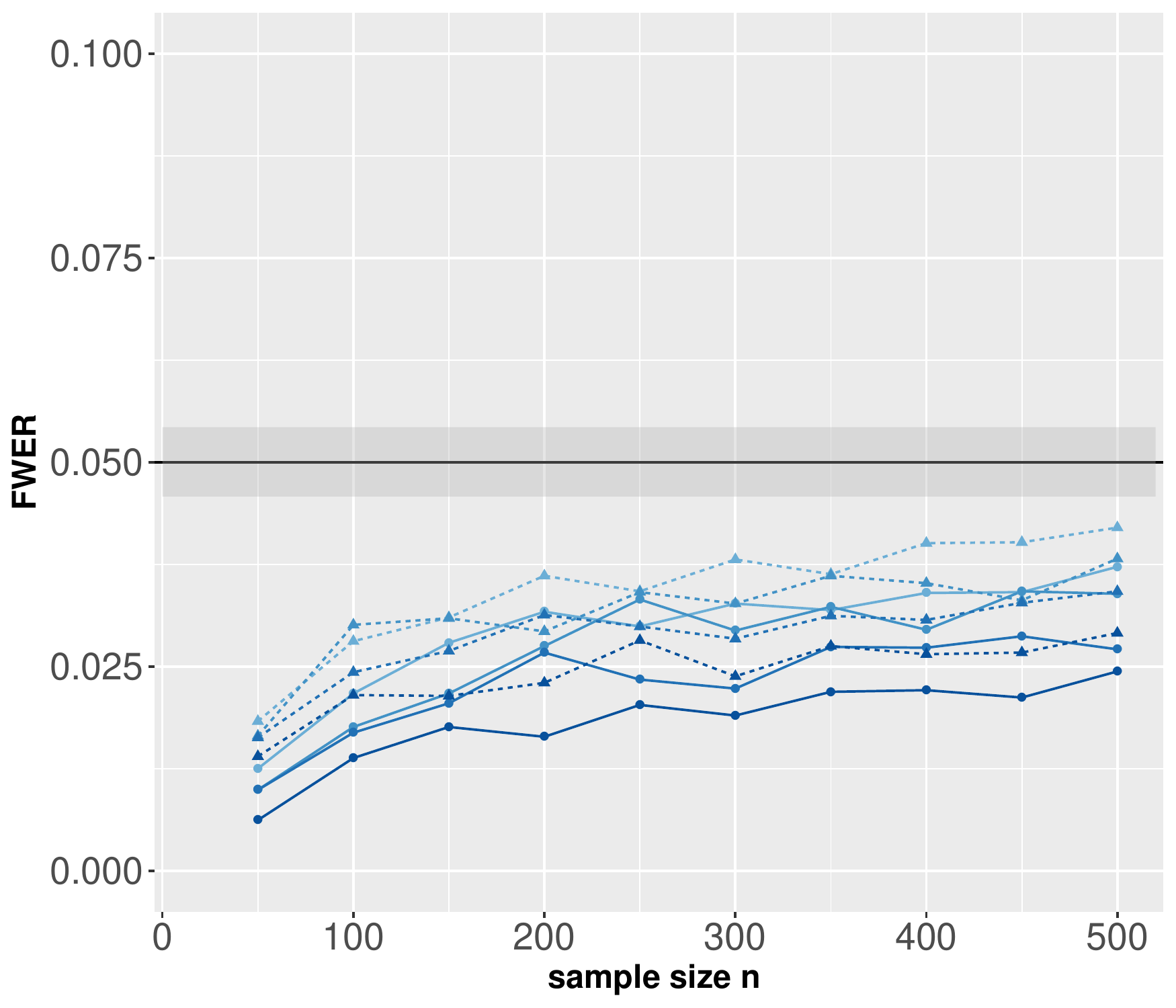}
        \end{subfigure}  \\
        \begin{subfigure}[t]{0.3\textwidth}
        \caption{\MaxT}
        \includegraphics[height=4cm]{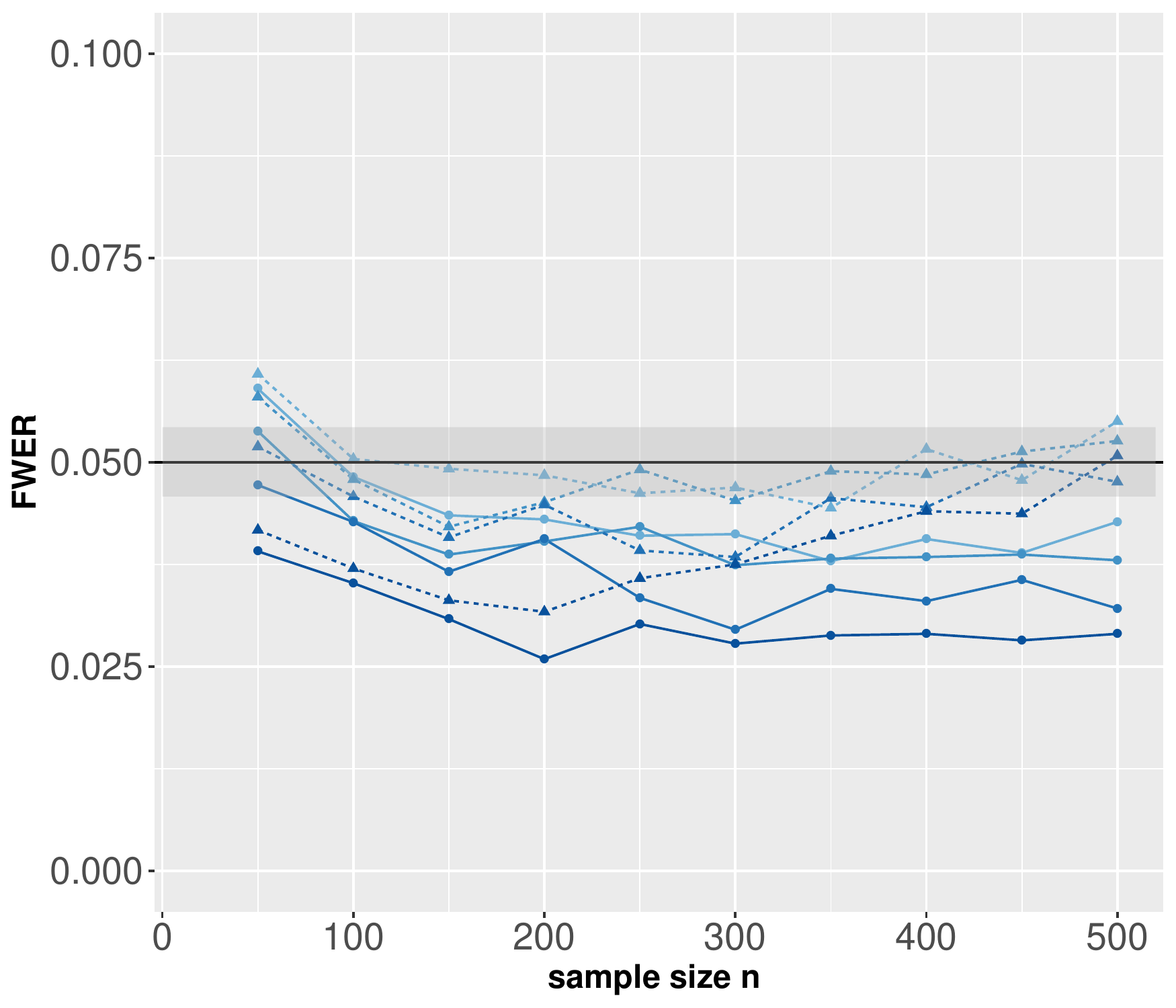}
        \end{subfigure}
                \begin{subfigure}[t]{0.32\textwidth}
        \caption{Oracle \MaxT}
        \includegraphics[height=4cm]{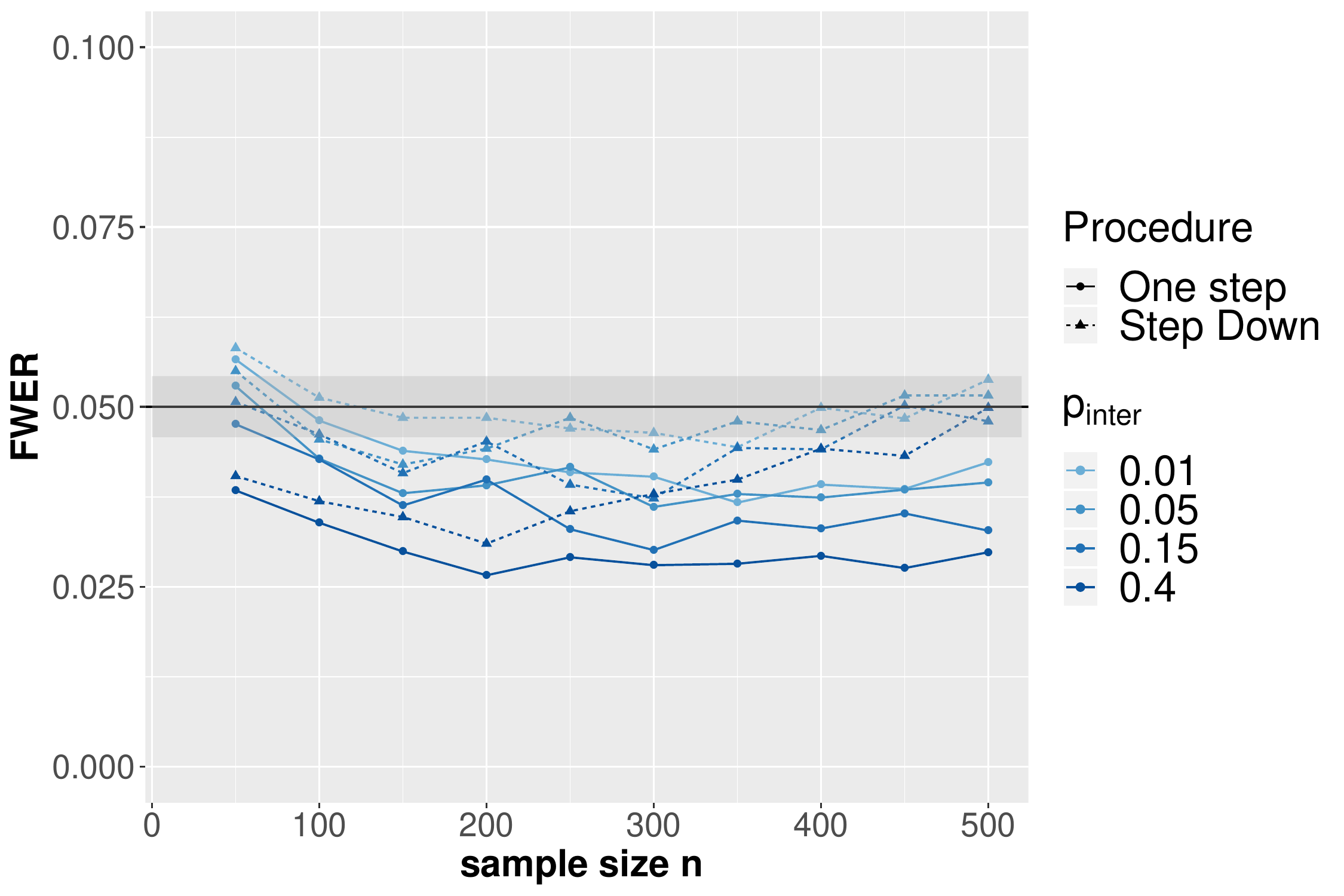}
        \end{subfigure}

\centering
\caption{Empirical FWER on 10000 simulations, with respect to the sample size $n$, for Gaussian statistics, $T^{(4)}$. The four sparsity frameworks are considered. $\rho=0.2$.}
  \label{fig:fwer2_gaussian}
    \begin{subfigure}[t]{0.3\textwidth}
        \caption{Bonferroni}
        \includegraphics[height=4cm]{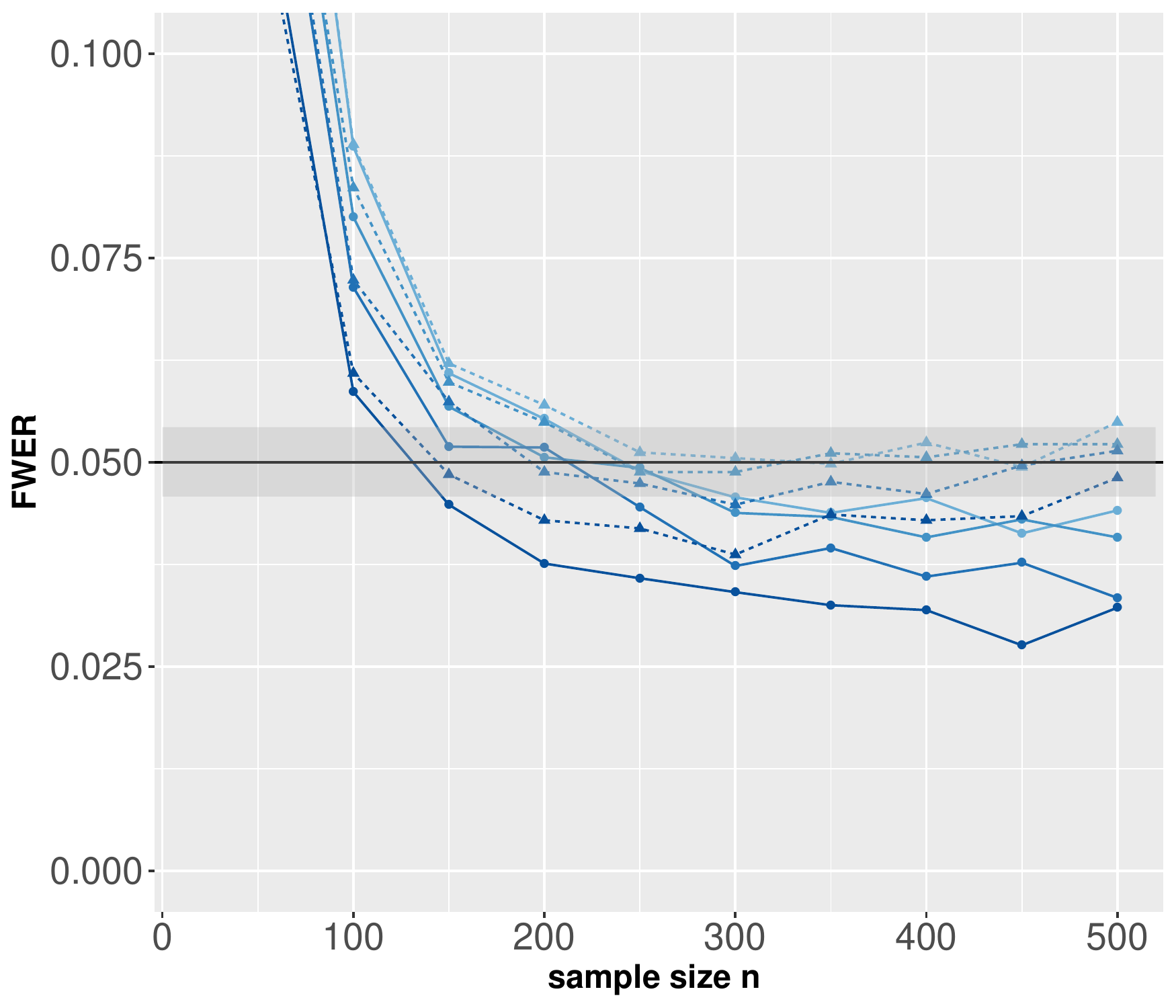}
        \end{subfigure}  
        \begin{subfigure}[t]{0.3\textwidth}
        \caption{\Sidak}
        \includegraphics[height=4cm]{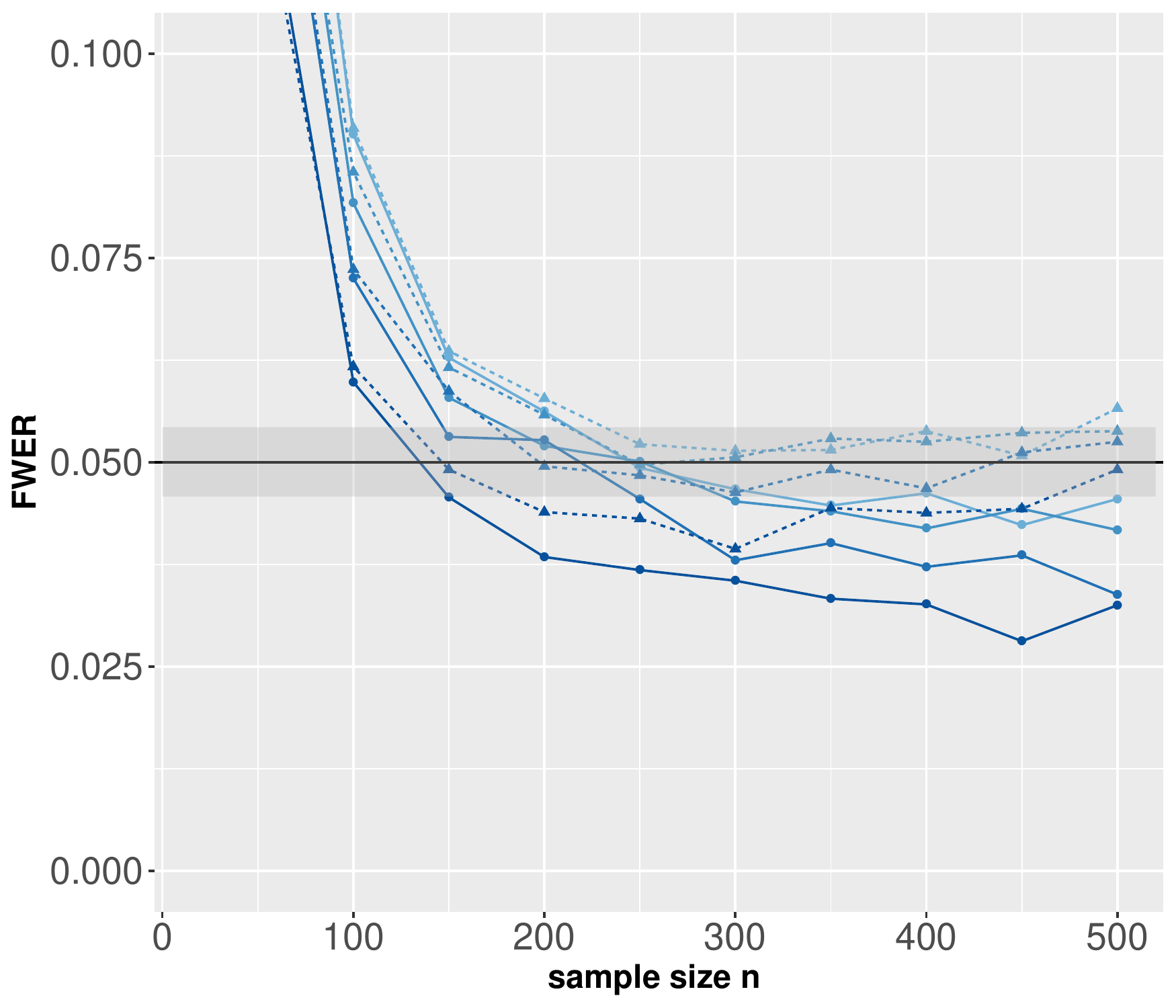}
        \end{subfigure} 
    \begin{subfigure}[t]{0.32\textwidth}
        \caption{\BootRW}
        \includegraphics[height=4cm]{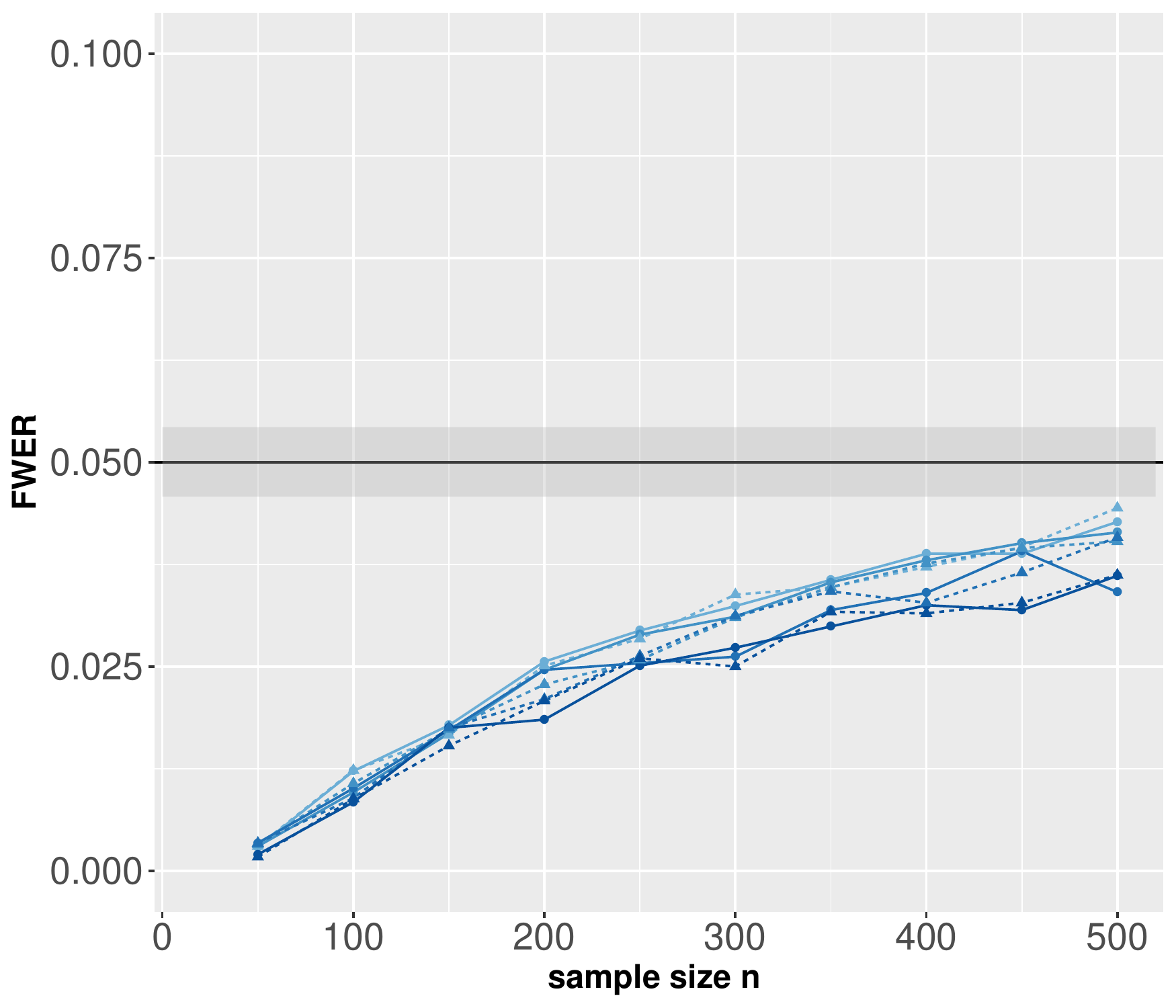}
        \end{subfigure} \\
        \begin{subfigure}[t]{0.3\textwidth}
        \caption{\MaxT}
        \includegraphics[height=4cm]{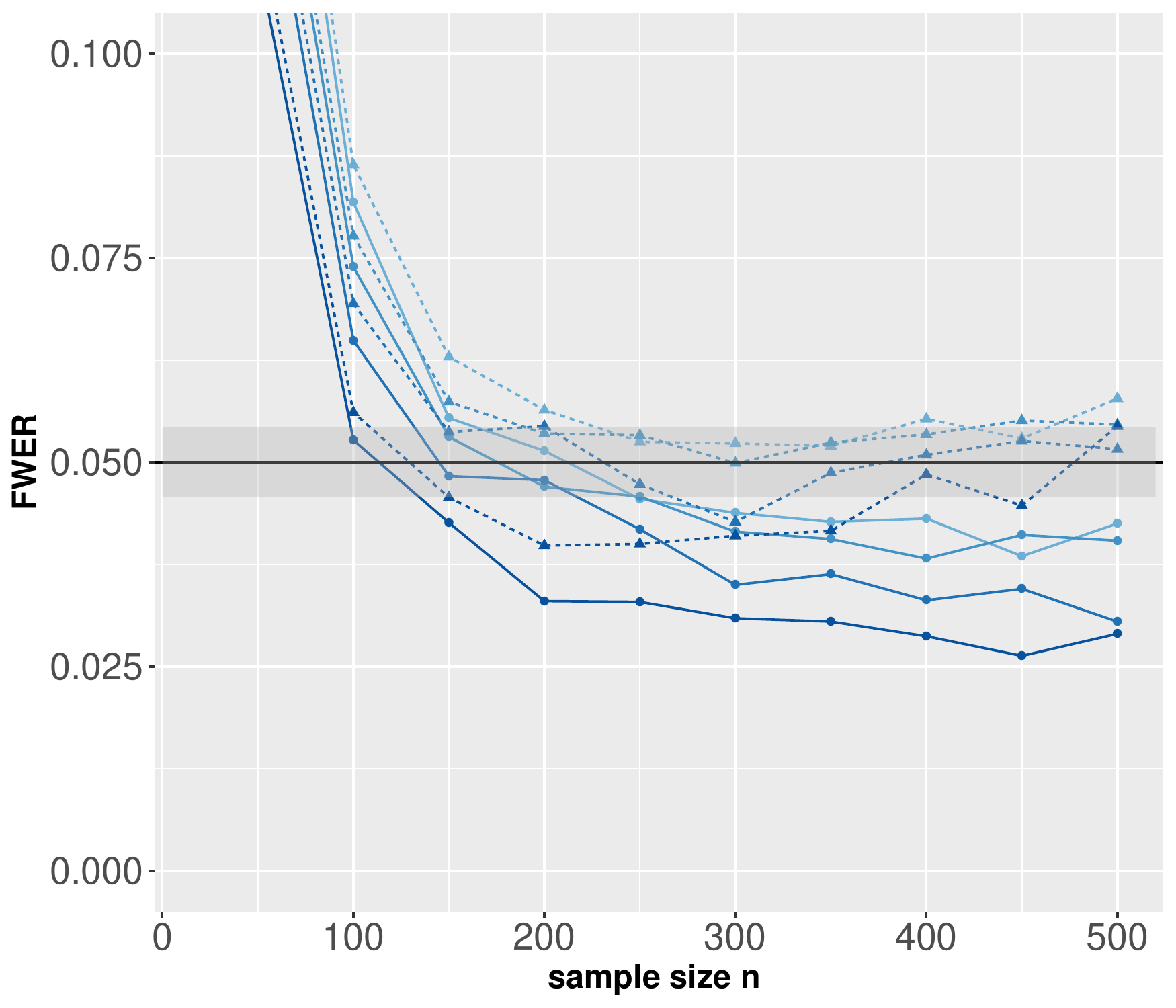}
        \end{subfigure}
                \begin{subfigure}[t]{0.32\textwidth}
        \caption{Oracle \MaxT}
        \includegraphics[height=4cm]{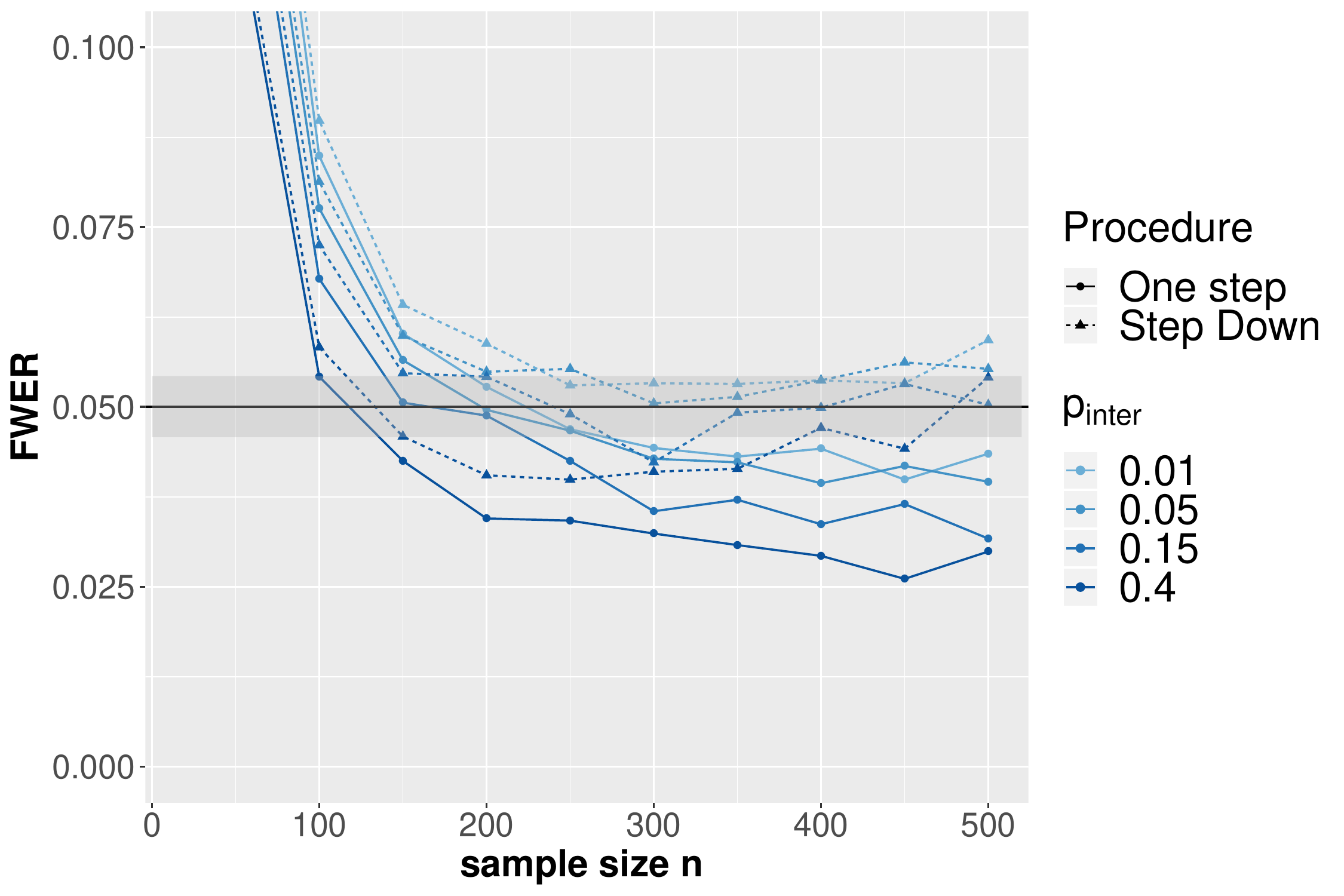}
        \end{subfigure}
\end{figure}

\subsection{Results on power}

This Section {provides a detailed evaluation of} the power of the multiple testing procedures on simulations. The quality is {computed} by the proportion of significant correlations detected, {also called }True Discovery Proportion, defined as \[\frac{\lvert\mathcal R\cap\{\mathcal H\setminus \mathcal H_0\}\rvert}{\lvert\mathcal H\setminus \mathcal H_0 \rvert},\] where for every sets $A$ and $B$, $A\setminus B=\{x\in A : x\notin B\}$.

First, a crucial point in many edges detection procedures is the graph sparsity (for instance \cite{bickel2008covariance}, \cite{ledoit2012nonlinear}, \cite{ravikumar2011high}, \cite{fan2016overview}). The sparsity is controlled here by the parameter $p_{inter}$: recall that the lower $p_{inter}$ is, the sparser the graph is. Figure~\ref{fig:power2_empirical} illustrates that with a multiple testing approach, the sparsity does not influence the power of the procedure, since the proportion of edges rightly detected is nearly constant with respect to the sparsity. 

Of course clearly, Figure~\ref{fig:power2_empirical} shows that the power increases with the sample size $n$.
 Indeed, with small $n$, the procedures are not able to detect significant correlations because the statistical tests cannot discriminate the null and the alternative hypothesis.  When the sample size is small, the variance of the empirical correlations on which the test statistics are based is high, see Figure~\ref{fig:hist1}.
 
Similarly, Figure~\ref{fig:power1_empirical} illustrates the failure of the procedures to detect significant correlations when the value of $\rho$ is small, here $\rho=0.1$. Again, this comes from the power of statistical tests rather than the multiple correction. Similar discussion can be found in  \cite{hero} (see {\it e.g.} Figure 2). Figures~\ref{fig:hist2} and \ref{fig:hist1} give an insight of maximal power reachable in our setting, respectively with $\rho=0.2$ and $\rho=0.1$. It also illustrates that indeed the sparsity does not influence the support of the empirical correlations distributions, which explains the stability of the power with respect to sparsity.

Comparison between the power of the procedure with respect to the test statistics is displayed in Table~\ref{tab:power_stat}. Only the case $\rho=0.2$, $p_{inter}=0.4$ and three values of sample size $n$ are considered for the sake of simplicity. Similar results are obtained in other settings. For all procedures except \BootRW, the Student statistics $T^{(3)}$ is the most powerful statistic. Concerning \BootRW, the empirical test statistics $T^{(1)}$ give the best results. Second-order statistic $T^{(4)}$ shows the lowest power for all multiple corrections.

{\cite{romano2005exact}'s non parametric bootstrap \BootRW~has the advantage of controlling FWER independently on the statistic and the sample size}. However it is less powerful than other step-down procedures. Results for step-down \MaxT~are high but this procedure may not control the FWER. The power of single-step parametric bootstrap \BootRW~is lower than Bonferroni and \Sidak~step-down procedures. 

We observe that the empirical test statistics $T^{(1)}$ always control the FWER. Thus this choice seems {particularly adequate}. {In addition,} the most powerful multiple correction is step-down \Sidak. If the sample size $n$ is sufficiently {large} so that the FWER control is acquired ($n\geq 200$ for the sparse model, see Figure~\ref{fig:fwer2_student}), step-down \Sidak~correction with Student statistics based tests gives the highest power. {An advantage of using Bonferroni and \Sidak~procedures lies in the fact that the computational time is very small.}

\begin{figure}
\centering
\caption{Empirical power on 10000 simulations, with respect to the sample size $n$, for empirical statistics, $T^{(1)}$. The four sparsity frameworks are considered. $\rho=0.2$. }
  \label{fig:power2_empirical}
    \begin{subfigure}[t]{0.3\textwidth}
        \caption{Bonferroni}
        \includegraphics[height=4cm]{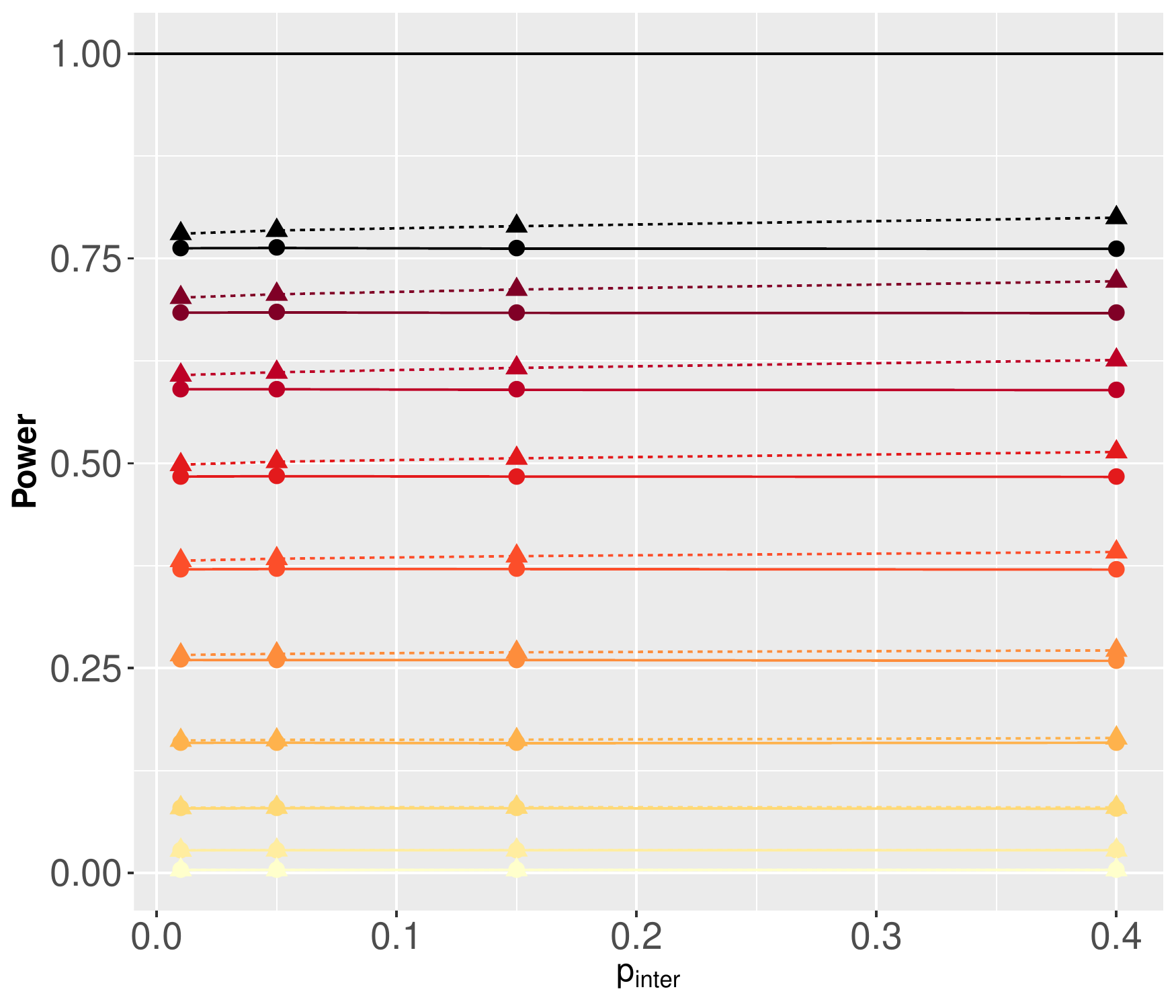}
        \end{subfigure}  
        \begin{subfigure}[t]{0.3\textwidth}
        \caption{\Sidak}
        \includegraphics[height=4cm]{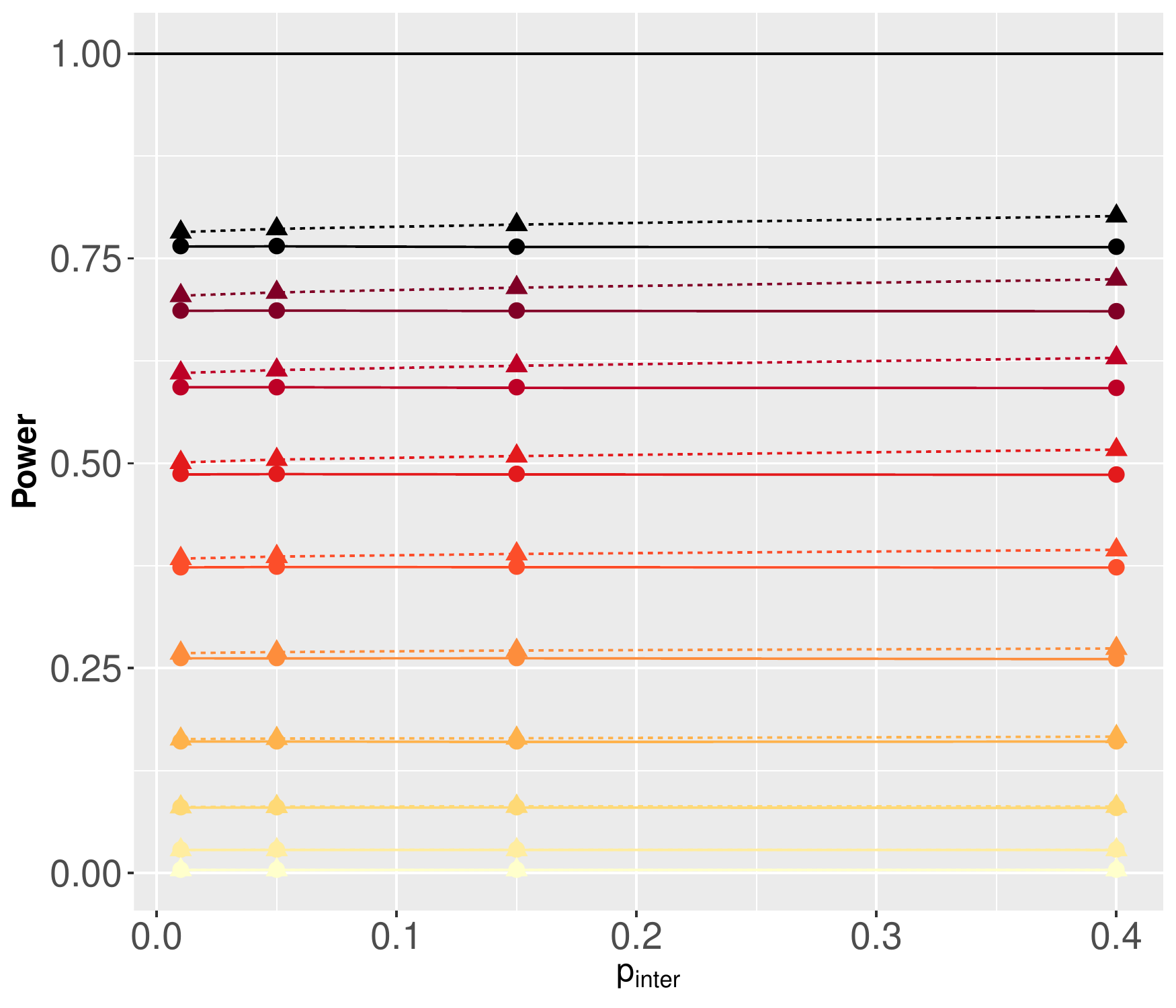}
        \end{subfigure} 
    \begin{subfigure}[t]{0.32\textwidth}
        \caption{\BootRW}
        \includegraphics[height=4cm]{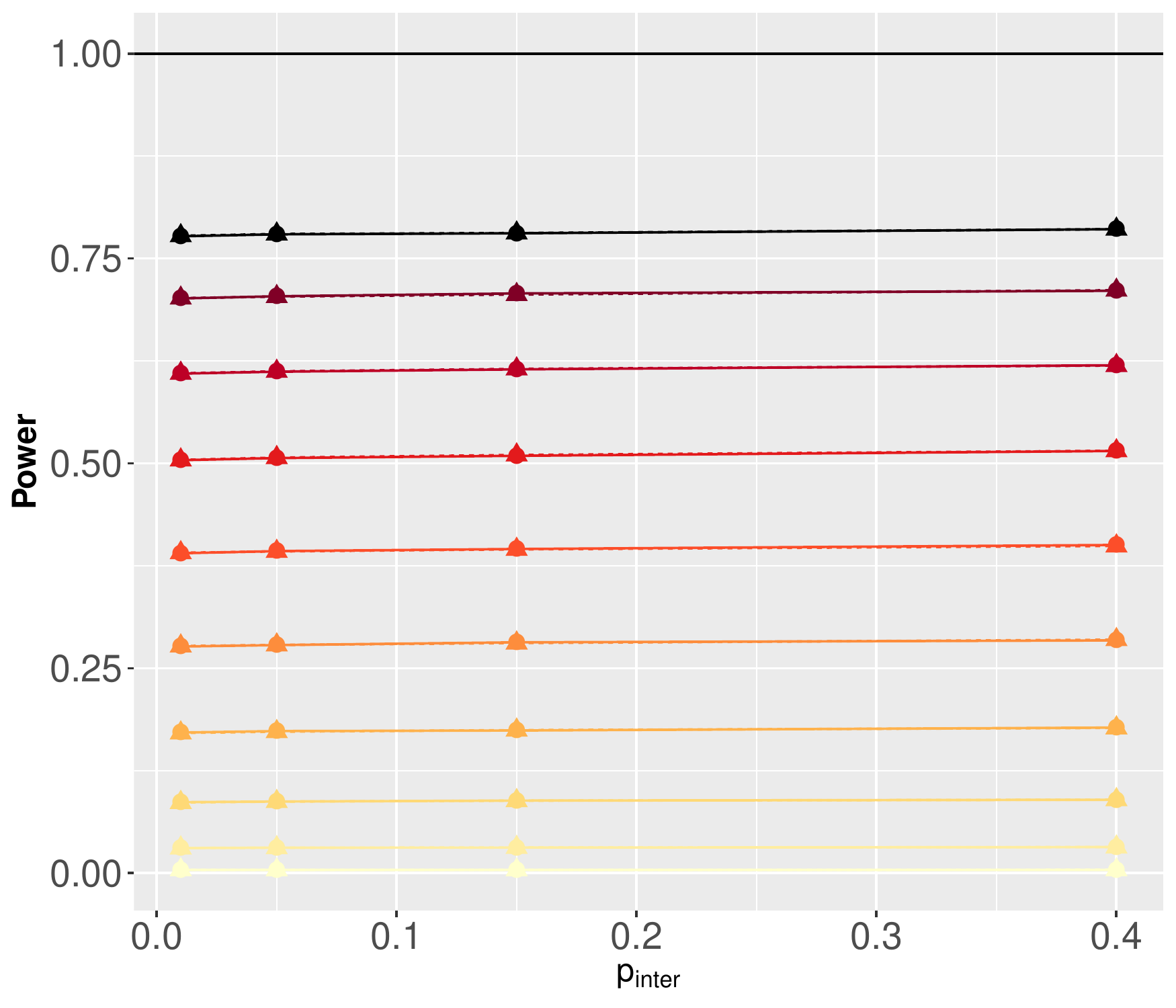}
        \end{subfigure}  \\
        \begin{subfigure}[t]{0.3\textwidth}
        \caption{\MaxT}
        \includegraphics[height=4cm]{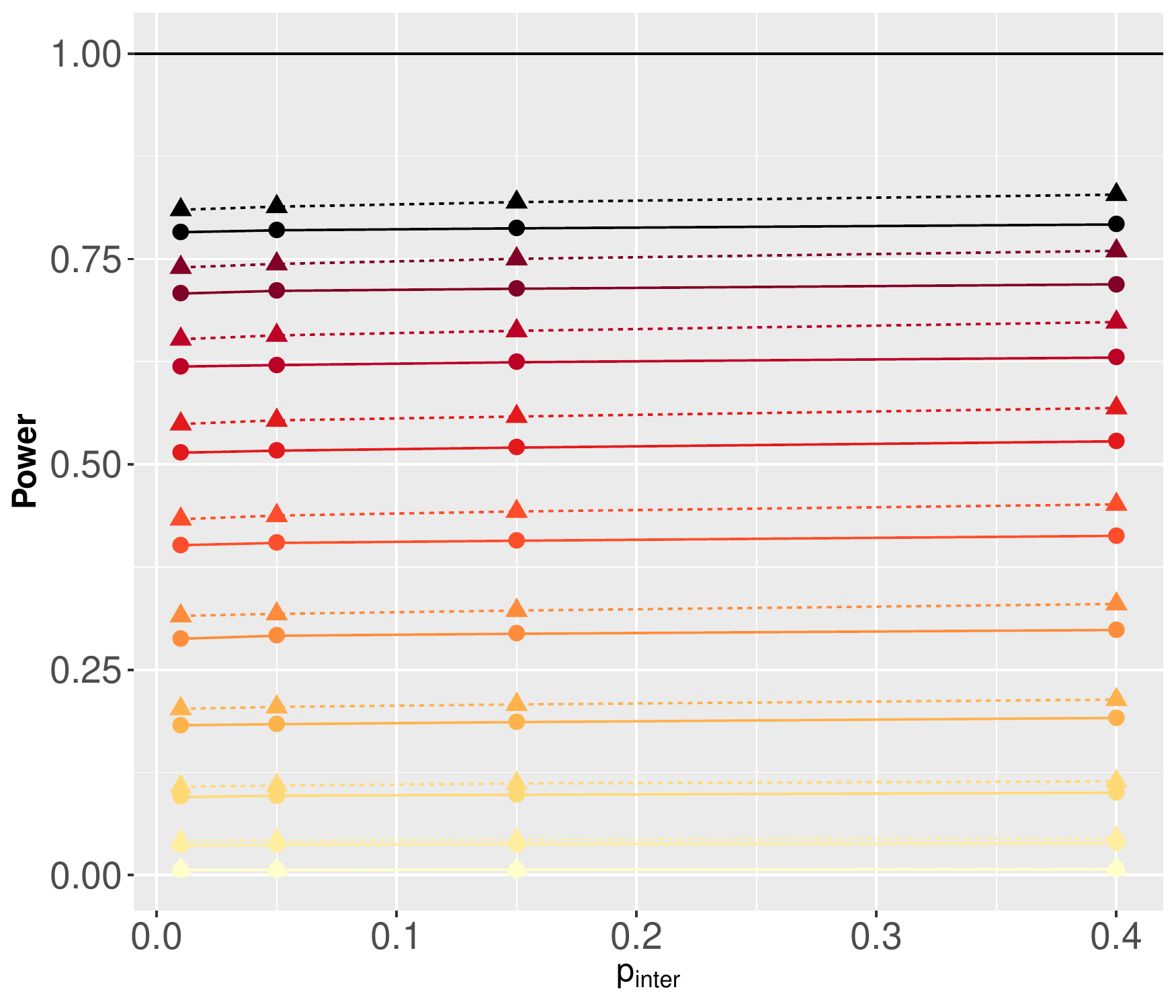}
        \end{subfigure}
                \begin{subfigure}[t]{0.32\textwidth}
        \caption{Oracle \MaxT}
        \includegraphics[height=4cm]{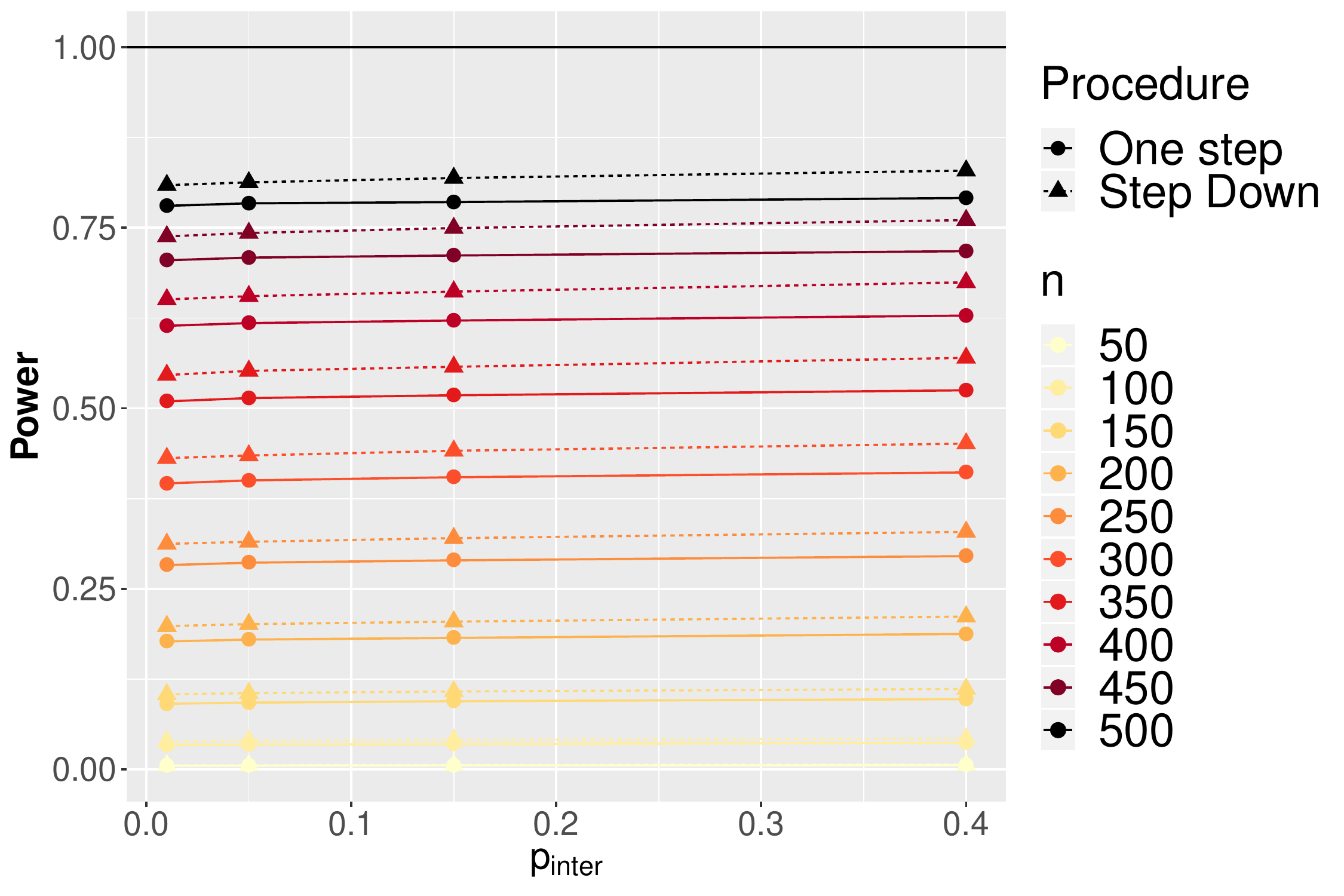}
        \end{subfigure}

\centering
\caption{Empirical power on 10000 simulations, with respect to the sample size $n$, for empirical statistics, $T^{(1)}$. The four sparsity frameworks are considered. $\rho=0.1$.}
  \label{fig:power1_empirical}
    \begin{subfigure}[t]{0.3\textwidth}
        \caption{Bonferroni}
        \includegraphics[height=4cm]{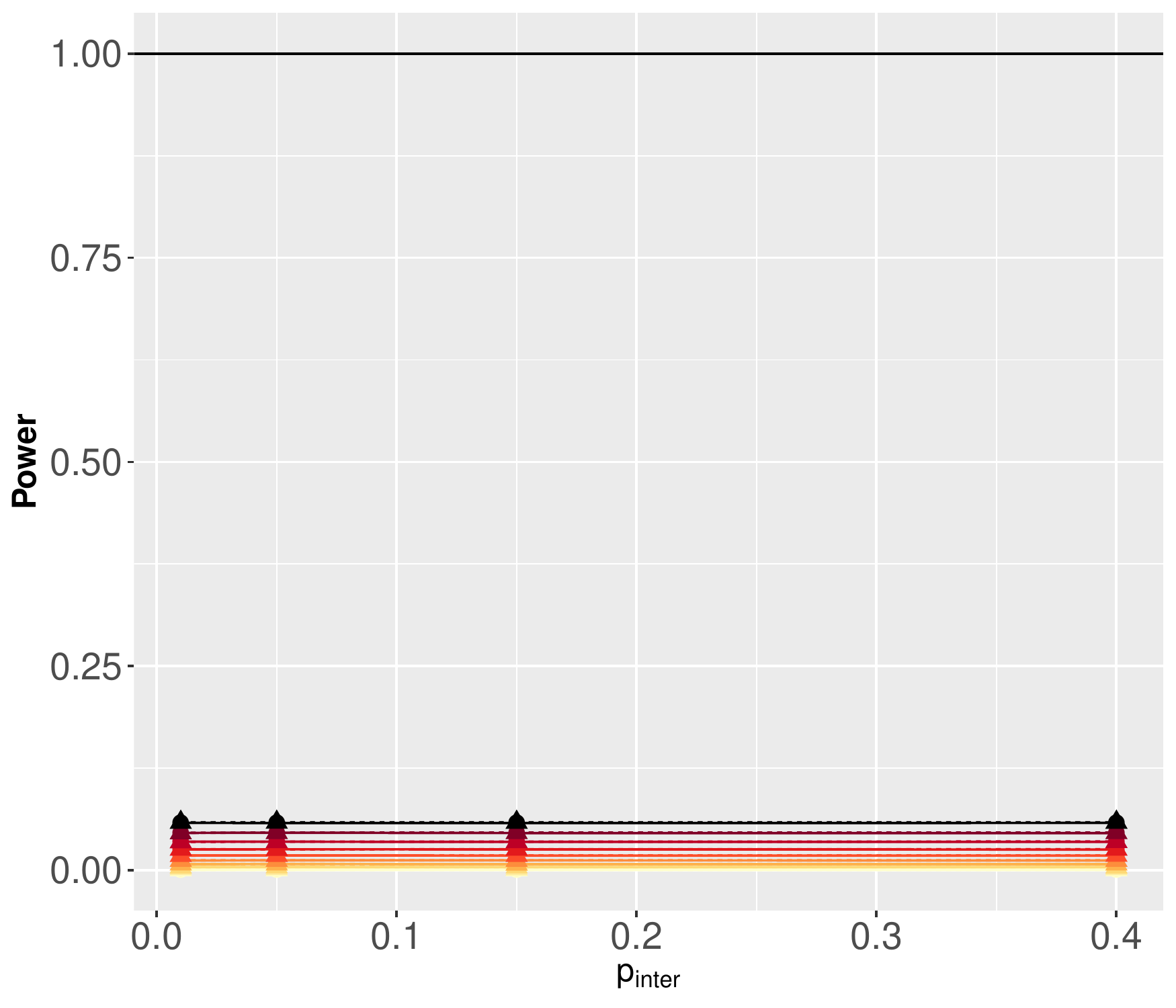}
        \end{subfigure}  
        \begin{subfigure}[t]{0.3\textwidth}
        \caption{\Sidak}
        \includegraphics[height=4cm]{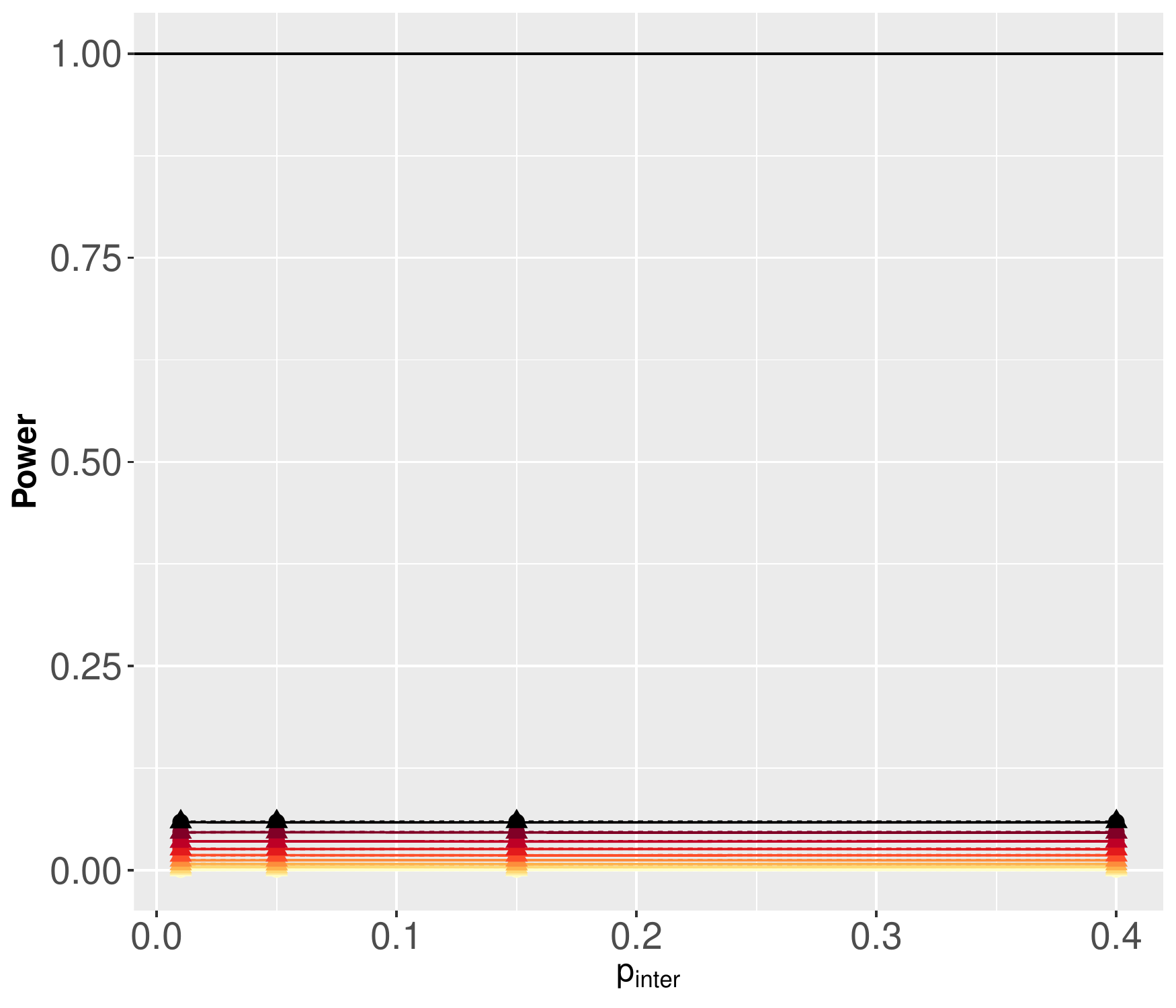}
        \end{subfigure} 
    \begin{subfigure}[t]{0.32\textwidth}
        \caption{\BootRW}
        \includegraphics[height=4cm]{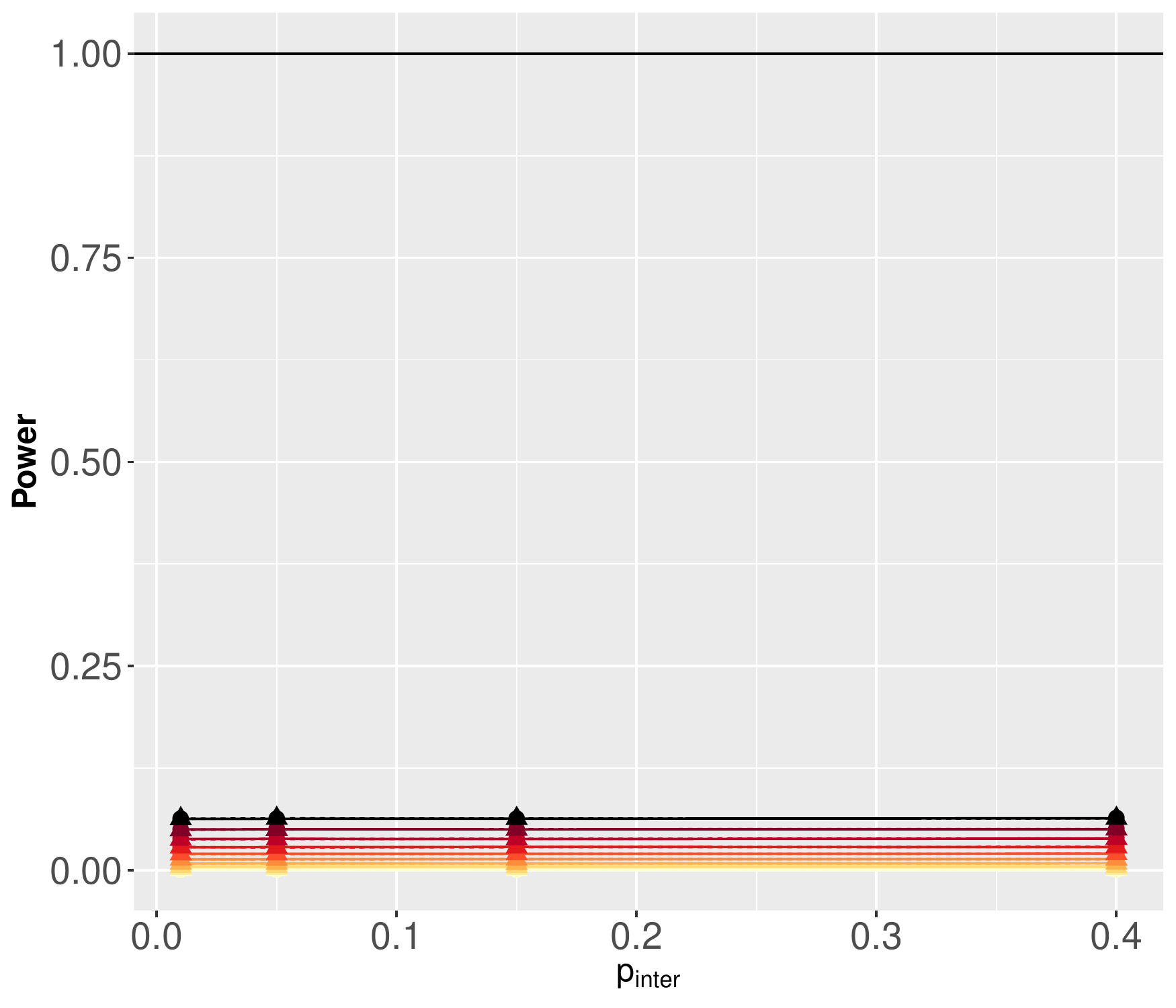}
        \end{subfigure}  \\
        \begin{subfigure}[t]{0.3\textwidth}
        \caption{\MaxT}
        \includegraphics[height=4cm]{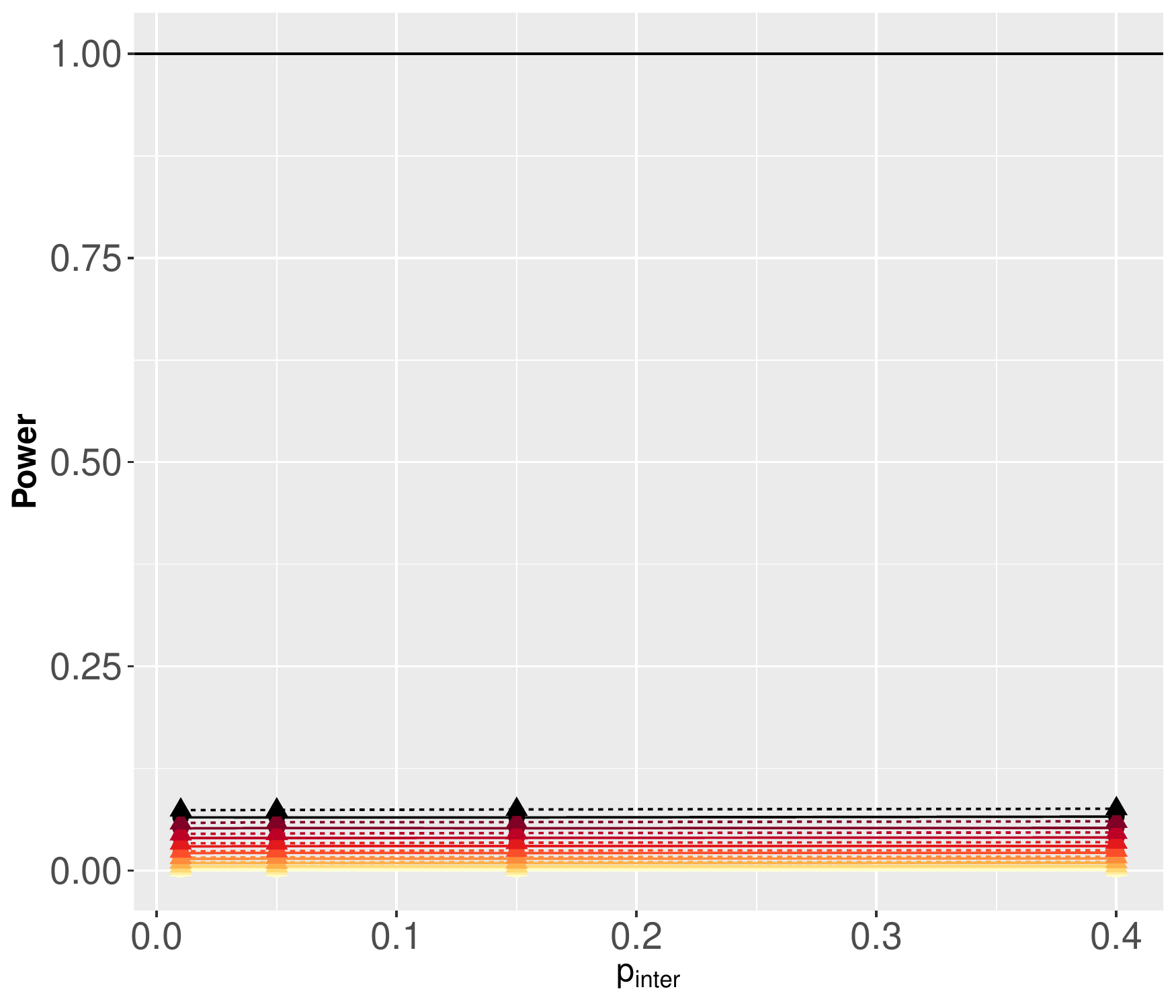}
        \end{subfigure}
                \begin{subfigure}[t]{0.32\textwidth}
        \caption{Oracle \MaxT}
        \includegraphics[height=4cm]{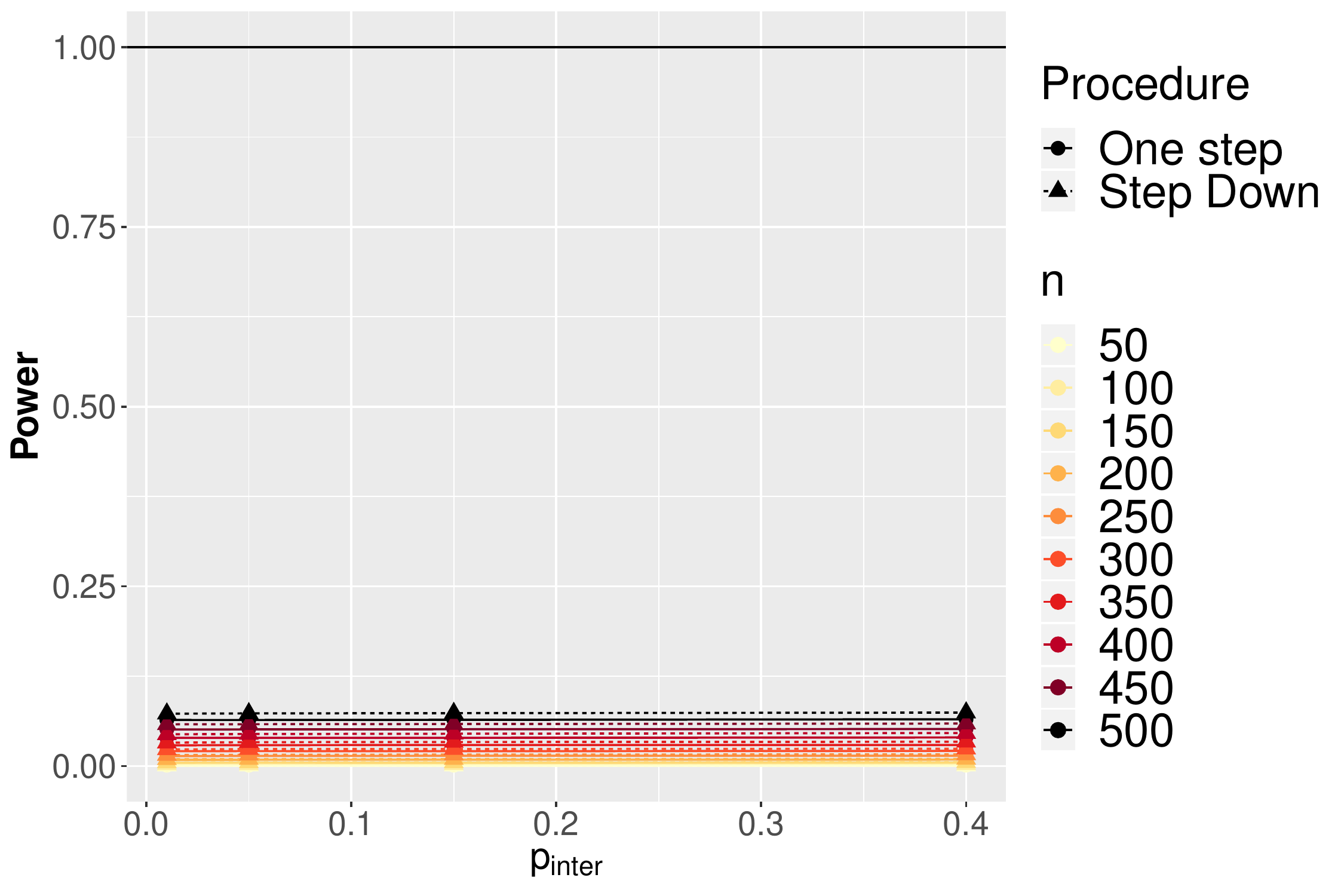}
        \end{subfigure}
\end{figure}

\begin{table}[!ht]
\centering

\caption{Empirical power on 10000 simulations, with respect to the multiple testing correction for sample size $n$ equal to 100, 300, 500. The adjacency matrix is given by $p_{inter}=0.4$, {\it i.e.} with the lowest sparsity. $\rho=0.2$. The left hand side value is the empirical power obtained with single-step methods and the right hand side value is obtained with step-down method.}
\label{tab:power_stat}
\begin{tabular}{lccccc}
\toprule
 \multicolumn{5}{c}{$n=100$}\\
\midrule
& Bonferroni & \Sidak & \BootRW & \MaxT & oracle \MaxT\\
Empirical $T^{(1)}$ & 0.028 $|$ 0.028 & 0.028 $|$ 0.028 & 0.032 $|$ 0.032 &  0.036 $|$ 0.038 &  0.034    $|$ 0.036 \\
Student~~ $T^{(2)}$ & {\bf 0.047} $|$ {\bf 0.047} & {\bf  0.047} $|$ {\bf 0.047}  & 0.024 $|$ 0.024 & 0.042 $|$ 0.045     &   {0.044} $|$ {\bf 0.047}\\
Fisher ~~~~$T^{(3)}$ & {\bf 0.047} $|$ {0.040} & {\bf 0.047} $|$ {0.040} &  0.024 $|$ 0.029 & 0.040  $|$ 0.043 &   0.040 $|$ { 0.042} \\
Gaussian ~$T^{(4)}$ & 0.026 $|$ 0.026 & 0.026 $|$ 0.026 & 0.006 $|$ 0.006 &  0.023 $|$  0.025   &  0.024  $|$   0.026    \\
\midrule
 \multicolumn{5}{c}{$n=300$}\\
\midrule
& Bonferroni & \Sidak & \BootRW & \MaxT & oracle \MaxT\\
Empirical $T^{(1)}$ & 0.370 $|$  0.392 & 0.373 $|$  0.394  & 0.400 $|$  0.399 & 0.405 $|$  0.422  & 0.401  $|$   {0.423}  \\
Student~~ $T^{(2)}$ & 0.400 $|$ 0.400 & 0.403 $|$ 0.403 &   0.357 $|$ 0.357 &  0.387 $|$ {\bf 0.424} &  0.389 $|$ 0.421\\
Fisher ~~~~$T^{(3)}$ & 0.400 $|$ 0.412 &  0.403 $|$ 0.414 &   0.357 $|$ 0.378 & 0.394 $|$ {0.422} &  0.393 $|$ {0.422} \\
Gaussian ~$T^{(4)}$ & 0.339 $|$ 0.359 & 0.342 $|$ 0.362 & 0.311 $|$ 0.311 &   0.328 $|$   0.360  &   0.330  $|$    0.358  \\
\midrule
 \multicolumn{5}{c}{$n=500$}\\
\midrule
& Bonferroni & \Sidak & \BootRW & \MaxT & oracle \MaxT\\
Empirical $T^{(1)}$ & 0.762 $|$ 0.800 & 0.764 $|$  0.802  & 0.786 $|$  0.786 & 0.787 $|$ 0.813   & 0.786 $|$  0.814\\
Student~~ $T^{(2)}$ & 0.776 $|$ 0.776 &  0.778 $|$ 0.778 &   0.752 $|$ 0.752 & 0.766 $|$ {\bf 0.818} &  0.767 $|$ 0.816\\
Fisher ~~~~$T^{(3)}$ &  0.776 $|$ 0.807 & 0.778 $|$ 0.809 &  0.752 $|$ 0.767 & 0.775 $|$ 0.816 &  0.774 $|$ {0.816} \\
Gaussian ~$T^{(4)}$ & 0.746 $|$ 0.787 & 0.748 $|$ 0.789 & 0.751 $|$ 0.751 &  0.737 $|$   0.793  &    0.738 $|$ 0.792  \\
\bottomrule
\end{tabular}

\end{table}

\begin{figure}
\centering
\caption{Histograms of the empirical correlations. Blue histograms correspond to true null hypothesis and red histograms correspond to observations under the alternative hypothesis $\rho=0.2$. On the first row the sparsity parameter satisfies $p_{inter}=0.01$ and on the second row $p_{inter}=0.4$. }
\label{fig:hist2}
\includegraphics[width=14cm]{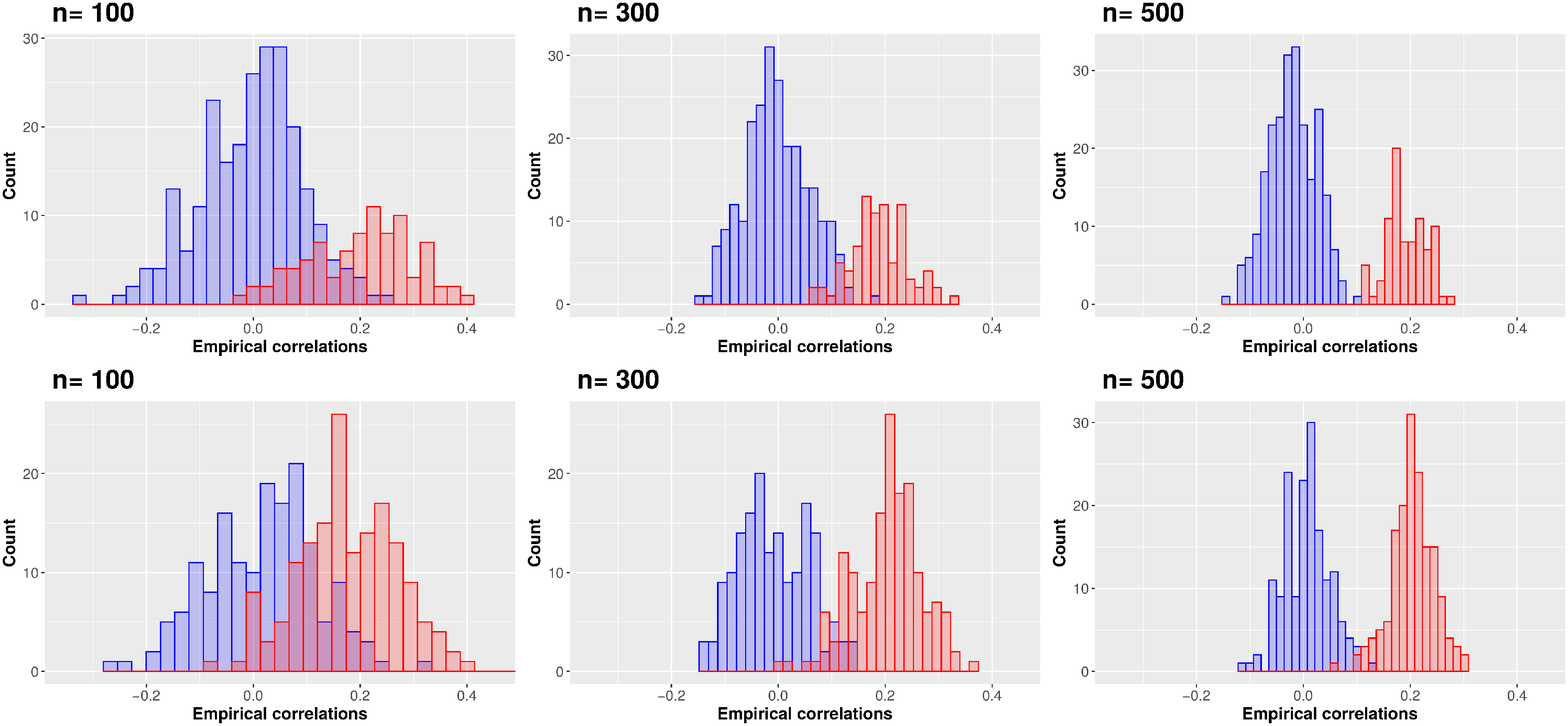}

\centering
\caption{Histograms of the empirical correlations. Blue histograms correspond to true null hypothesis and red histograms correspond to observations under the alternative hypothesis $\rho=0.1$. On the first row the sparsity parameter satisfies $p_{inter}=0.01$ and on the second row $p_{inter}=0.4$. }
\label{fig:hist1}
\includegraphics[width=14cm]{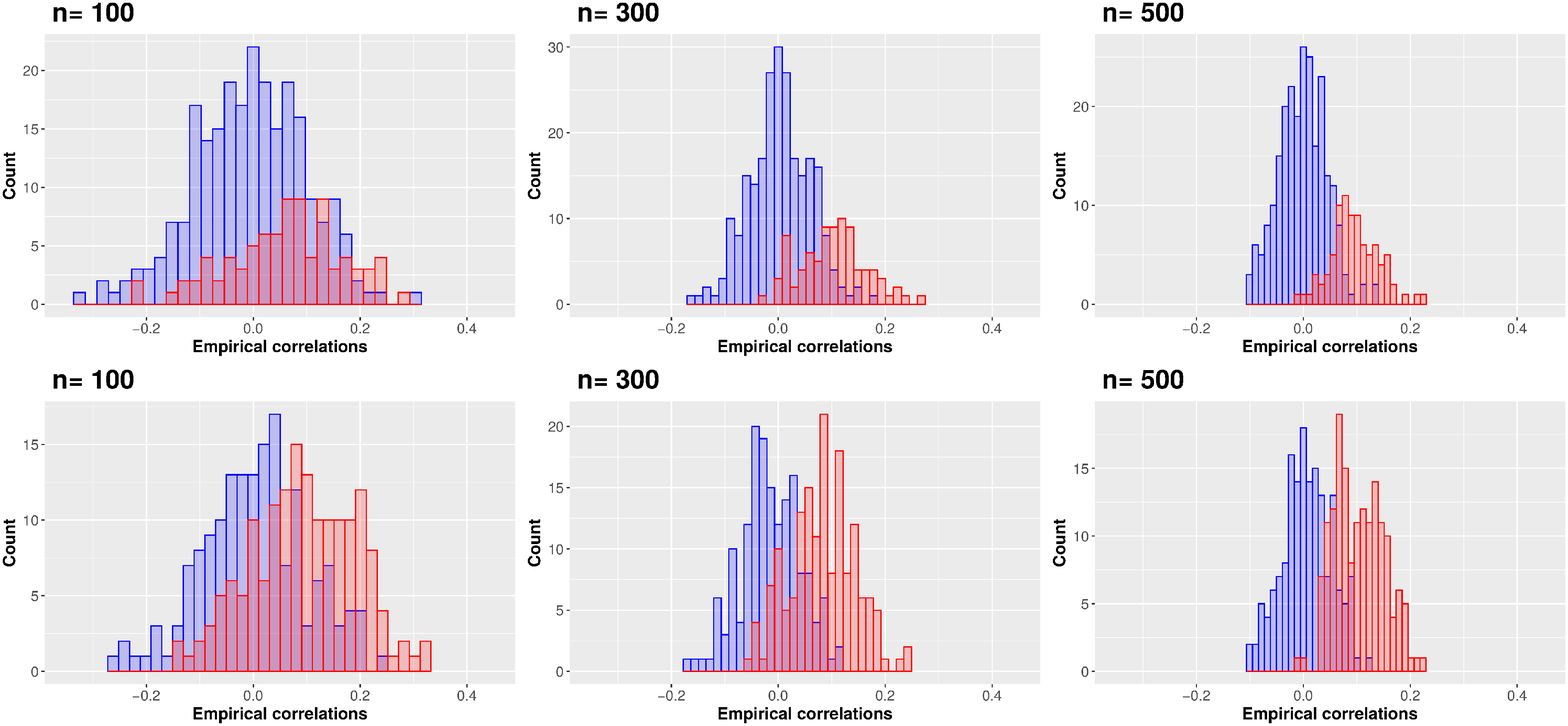}
\end{figure}

\section{Real data}

We apply our methods on real data sets from neuroscience. We use functional Magnetic Resonance images (fMRI) of both dead and alive rats. The datasets are freely available \href{https://zenodo.org/record/2452871}{10.5281/zenodo.2452871} \citep{becq_10.1088/1741-2552/ab9fec,guillaume2020functional}. The aim is to estimate the brain connectivity, that is the significant correlations between brain regions where fMRI signals are recorded. For this data set, we know that for dead rats we are under the full null hypothesis. Thus the estimated graphs should be empty. We also expect non-empty graphs for alive rats.

\subsection{Description of the dataset}

Functional Magnetic Resonance Images (fMRI) were acquired for both dead and alive rats in \cite{Pawela2008}. Rats $1$ to $4$ were scanned dead and rats $5$ to $11$ were scanned alive under anesthetic. The duration of the scanning is 30 minutes with a time repetition of 0.5 second so that $3600$ time points are available at the end of experience. After preprocessing as explained in \cite{Pawela2008}, $51$ time series for each rat were extracted. Each time series capture the functioning of a given region of the rat brain based on an anatomical atlas.

The time series resulting from fMRI experiments are nonstationary with long memory properties, which is not convenient from a mathematical point of view. To avoid such properties, the correlation coefficients are estimated in the wavelet transform domain and then the statistical tests are based on wavelet correlation coefficients, as described in \cite{achard.2014.1, moulines.2007.1}.

First we are decomposing each time series using a wavelet basis and then for each wavelet scale we are studying all the possible pairs of correlations.

Let $(\phi(\cdot),\psi(\cdot))$ be respectively a father and a mother wavelets. At a given resolution $j\geq 0$, for $k\in\mathbb Z$, we define the dilated and translated functions $\phi_{j,k}(\cdot)=2^{-j/2}\phi(2^{-j}\cdot -k)$ and $\psi_{j,k}(\cdot)=2^{-j/2}\psi(2^{-j}\cdot -k)$. Let $\tilde{\mathbf X}(t)=\sum_{k\in\mathbb Z}\mathbf  X(k)\phi(t-k)$. The wavelet coefficients of the process $\mathbf X$ are defined by 
\[
\mathbf W_{j,k}=\int_\R \tilde{\mathbf X}(t)\psi_{j,k}(t)dt\quad j\geq 0,\; k\in\mathbb Z.
\]
For given $j\geq 0$ and $k\in\mathbb Z$, $\mathbf W_{j,k}$ is a $p$-dimensional vector \[\mathbf W_{jk}=\begin{pmatrix}
W_{j,k}(1) & W_{j,k}(2) & \dots & W_{j,k}(p) \end{pmatrix},\] where $W_{j,k}(\ell)= \int_\R \tilde{X_\ell}(t)\psi_{j,k}(t)dt$.

Let $z$ be the Fisher transform, for all $x$, $-1 \leq x \leq 1$, 
$$z(x):=\frac{1}{2}\log \frac{1+x}{1-x} = \tanh^{-1}(x).$$
We fix a scale $j$ and the tests are based on the correlation of wavelets coefficients at scale $j$. Define
\begin{align*}
{\rho_{\ell m}}(j) &=  \frac{\gamma_{\ell m}(j)}{(\gamma_{\ell \ell}(j)\gamma_{m m}(j))^{1/2}}, \quad \text{~with~} \gamma_{\ell m}(j) :=\int \mathcal H_j(\lambda) \, f_{\ell m}(\lambda) \, d\,\lambda\\
\hat{\rho_{\ell m}}(j) &= \frac{\sum_{k=0}^{n_j}(W_{j,k}(\ell)W_{j,k}(m))}{(\sum_{k=0}^{n_j}W^2_{j,k}(\ell)\sum_{k=0}^{n_j}W^2_{j,k}(m))^{1/2}} ,
\end{align*}
where $\mathcal H_j$ denotes the squared gain filter of wavelet transform at scale $j$. Note that quantities $\rho_{\ell m}(j) $ depends on the memory properties of the time series, as described among others in \cite{achard.gannaz.2019}.

The convergence of wavelet correlations with Fisher transform is given by the following property.

\begin{prop}[\cite{spie2019}]
\label{prop.fisher.cor}
Under the regularity hypotheses on the wavelet transform, the estimator of empirical correlation verifies
\begin{equation}
\sqrt{(n_j-3)}(z(\widehat{\rho_{\ell m}}(j))-z(\rho_{\ell m}(j))) \mathop{\longrightarrow}^{\mathcal{L}} \mathcal{N}(0,1)\,;
\end{equation}
\end{prop}

A discussion on the choice of the relevant scale for analysis can be found in \cite{spie2019}. From there and the literature in neuroscience, it is convenient and adequate to focus on low frequencies where the best signal-to-noise ratio is obtained. Here we will focus on wavelet scale 4 corresponding to the frequency interval [0.06 ; 0.12] Hz. The number of available wavelet coefficients is then $n=122$.

For a given rat $k\in\{1,\dots,11\}$, let $\mathbf W^{(k)}_\ell$ be the vector of wavelet coefficients at scale 4 for the time series of the $\ell$-th cerebral region, $\ell=1,\dots, 51$ and $k=1,\dots,11$. The vector $\mathbf W^{(k)}_\ell$ has length 122. We introduce $(z(\hat \rho_{\ell m}^{(k)}))_{\ell,m=1,\dots,p}$ the matrix of Fisher statistics evaluated on the wavelet coefficients processes $\mathbf W^{(k)}_{\ell}$ and $\mathbf W^{(k)}_m$. With previous notations, for a given rat $k$, for all $(\ell,m)$, $z(\hat \rho_{\ell m}^{(k)})$ corresponds to the Fisher statistics $T^{(3)}_{\ell m}$ evaluated with $Y^{(\ell)}=\mathbf W^{(k)}_{\ell}$.

The aim is then to test if there are connections between each pairs of cerebral region at a given scale $j\geq 0$. 
The following tests should thus be applied, for each rat $k$, at the scale $j=4$ in this study:
\begin{equation} \label{test_wav}\tag{P-3} H_{0,\ell m}(j) : \rho_{\ell m}(j)=0  \text{~ against ~}  H_{1,\ell m}(j) :  \rho_{\ell m}(j)\neq 0,
\end{equation} all  $1\leq \ell<m\leq p$.
for
Since the whole brain of rats are aggregated into $p=51$ cerebral regions, we have to deal with $m = 1275$ tests for each of the 11 rats.

\subsection{Numerical results}

As explained above, we consider tests \eqref{test_wav} for each rat, on wavelet coefficients at scale 4. A rejection of hypothesis $(H_{0,\ell m})$ means a significant correlation between the activities of cerebral regions $\ell$ and $m$. 

Examples of correlation matrices obtained are given in Figure~\ref{fig:cor_rat}. Even if dead rats correlation matrices are much closer to identity matrices than correlation matrices on alive rats, some non zero values are observed and a test of significance is needed.

\begin{figure}[!ht]
\centering
\caption{Examples of two correlation matrices between wavelet coefficients obtained at scale 4 for an alive rat (left) and a dead rat (right).}
\label{fig:cor_rat}
\hspace*{-1cm} \includegraphics[height=8cm]{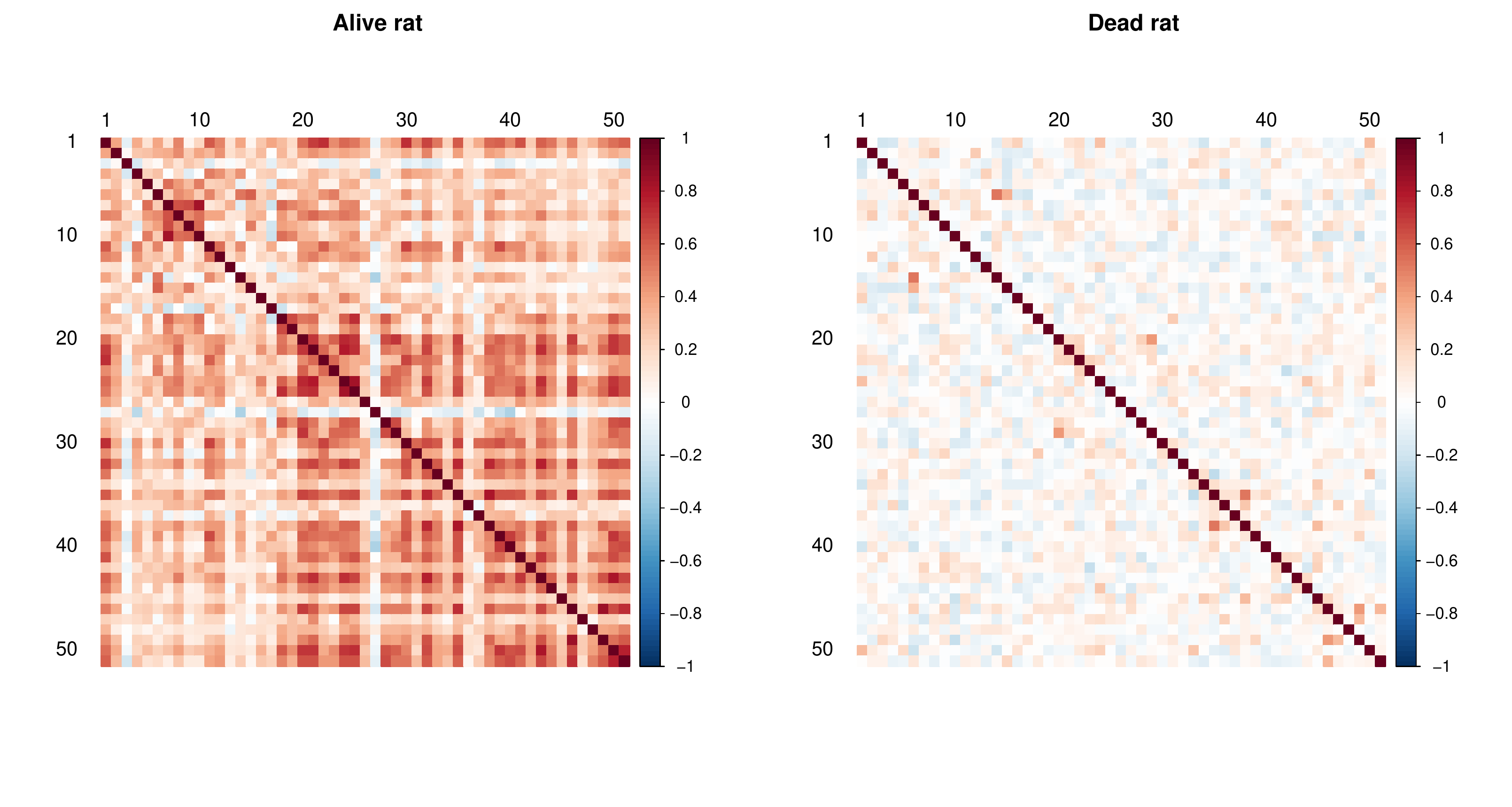}
\end{figure}

Boxplots of empirical correlations are displayed in Figure~\ref{fig:cor_rat_boxplots}. They state a clear difference between dead and alive rats correlation distributions. An important remark is also that in real application the alternative non zero value is not constant and can be much higher than the value taken in our simulation (which was at the maximum equal to 0.2). Indeed for each alive rat, we observe instances of empirical correlation higher than $0.8$. We thus expect a signal-to-noise ratio higher in the application than in our simulation, and therefore good power performances of multiple testing.

\begin{figure}[!ht]
\centering
\caption{Boxplots of empirical correlations between wavelet coefficients obtained at scale 4  (i.e., the frequency interval [0.06 ; 0.12] Hz) on fMRI recordings of rats. The first four rats are dead while rats 5 to 11 are alive.}
\label{fig:cor_rat_boxplots}
\includegraphics[height=8cm]{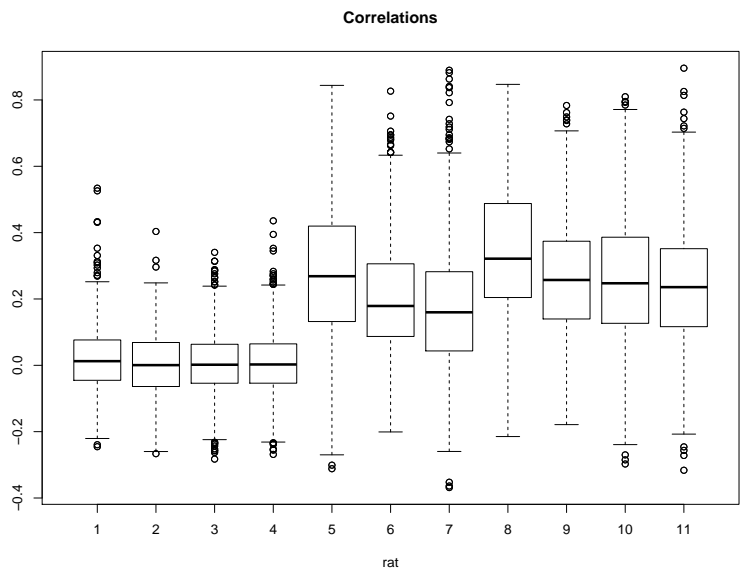}
\end{figure}

Since we consider Fisher statistics $T^{(3)}$, with sample size $n=122$, the simulation study displays that the FWER is controlled for all procedures (see Figure~\ref{fig:fwer2_fisher}). Next the most powerful approach is step-down \MaxT~parametric bootstrap in general, but for small samples \Sidak~procedure is competitive. Moreover the numerical instability on the quantile estimation for \MaxT~as well as the computation time give preference to step-down \Sidak~multiple correction (see Table~\ref{tab:power_stat}). 

We thus define the rejection set of \eqref{test_wav} applying step-down \Sidak's multiple testing procedure on Fisher statistics. As we have to deal with $1275$ tests for each rat, a multiple testing procedure is necessary. Without a multiple testing correction on $p$-values, the inferred graph is highly biased, as illustrated by the clearly non-empty graph for a dead rat displayed in Figure~\ref{fig:without}. A similar result was obtained on the detection of an activation on fMRI recordings for a dead salmon in \cite{bennett2011neural}.

\begin{figure}[!ht]
\centering
\caption{Illustration of the necessity to make multiple correction for a given dead rat:
Left: estimate of the connectivity graph between brain regions without correction for multiple testing, using the  raw values of the empirical correlations; Right: estimation of the same graph using multiple testing on correlations (Step-down \Sidak\ procedure with Fisher statistics) ; $\alpha = 5\%$.}
\label{fig:without}
\includegraphics[height=7cm]{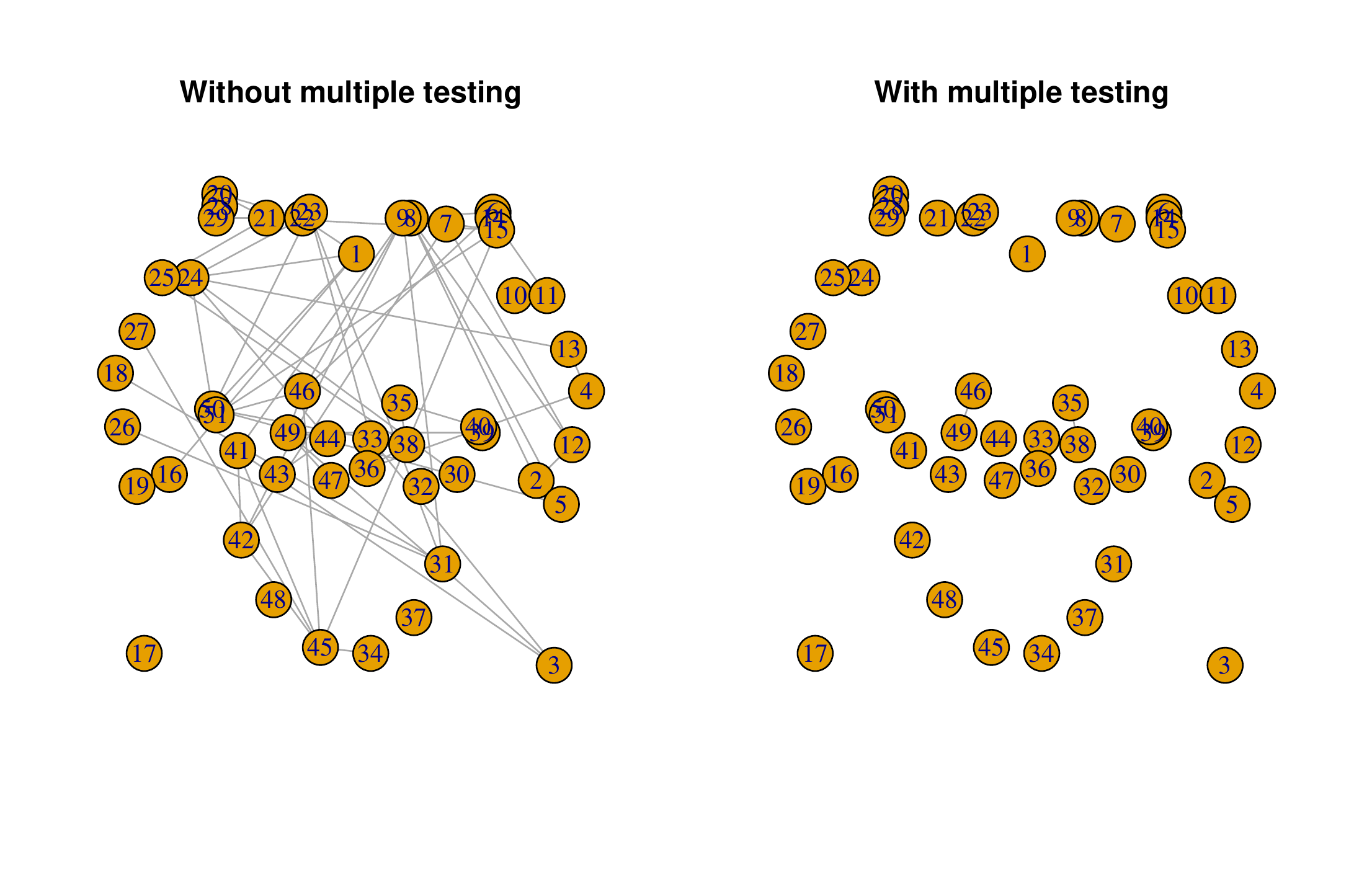}
\end{figure}

It is not possible in general to evaluate the accuracy of procedures on real dataset because of the lack of ground truth. However, since there is actually no cerebral activity in a brain of a dead rat, all null hypotheses to be tested are true nulls and the number of rejected hypotheses must be equal to zero in this case. Indeed, we obtain that the number of significant correlations is zero or near to zero for dead rats (see Table \ref{tab:rats}). We can still observe a few remaining connections in the graphs from very close brain regions. This is due to the fact that the fMRI scanner is introducing spurious correlations between neighboring parts of the brain. Some examples of estimated graphs for dead rats and alive rats are displayed in Figure~\ref{fig:graphs_rats_fisher}. Edges for dead rats are not visible because they are between nodes at very short distance.

\begin{table}[!ht]
\begin{center}

\begin{tabular}{rcccc}
\toprule
Dead rat & 1 & 2 & 3 & 4  \\ 
\midrule
$|E|$ &  4  &  1  &  0 &  2 \\  \bottomrule \\
\end{tabular}
%&
\begin{tabular}{rccccccc}
\toprule
Alive rat &  5 & 6 & 7 & 8 & 9 & 10 & 11 \\
\midrule
$|E|$&   444  &  236  &  201  &  595 &   385  &  410  &  318 \\
\bottomrule \\
\end{tabular}

\caption{\label{tab:rats} Number of rejected null hypotheses $| E |$ (that is also the number of estimate edges) obtained by  step-down \Sidak~procedure on Fisher statistics, evaluated at scale 4 with a threshold $\alpha = 5\%$.}
\end{center}
\end{table}

Note that there is a high variability of the results with respect to rats. The highest number of detections is obtained for the rats $5$ and $8$ with more than 33\% of edges. For the others, the number of detections is between $15\%$ and $33\%$.
 
\begin{figure}[!h]
\centering

\caption{Examples of graphs obtained at scale 4 for some alive rats (first row) and some dead rats (second row).  Step-down \Sidak~correction was used with Fisher statistics.}
\label{fig:graphs_rats_fisher}
\hspace*{-2cm}
\includegraphics[height=10cm]{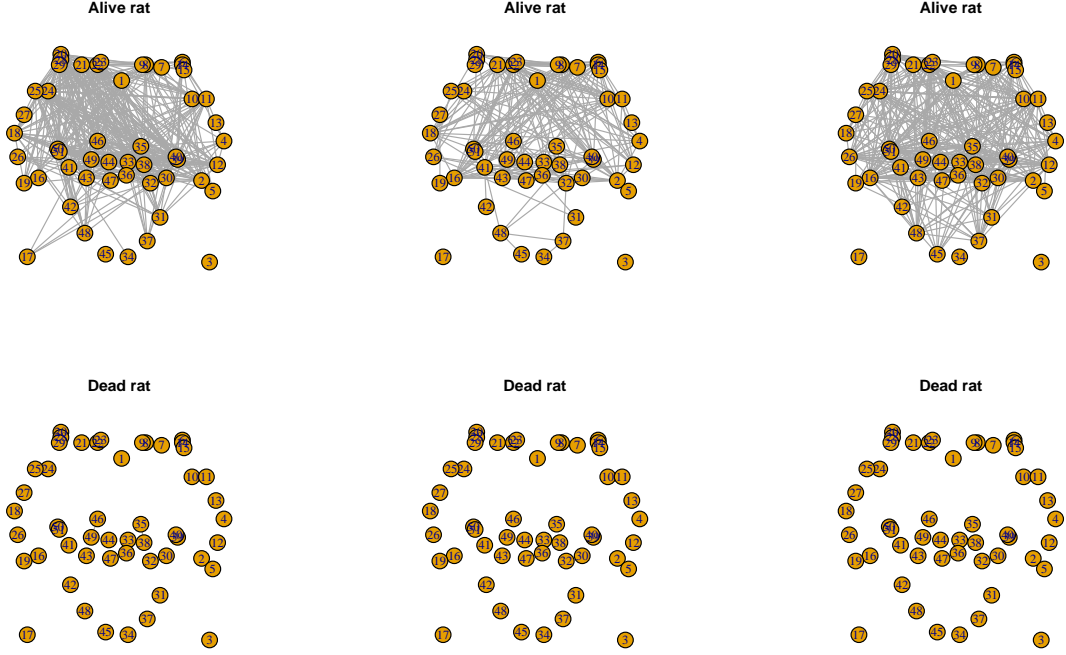}
\end{figure}

\section{Asymptotic control of the FDR for correlation tests}

Another usual criterion for multiple testing is False Discovery Rate (FDR). Denoting $\mathcal Q$ the proportion of false discoveries, for all $P \in \mathcal P$, we define the FDR as
\begin{equation}\label{eq:BH}\text{ FDR}( \mathcal R,P) = \mathbb E[\mathcal Q ], \text{ where } \mathcal Q = \frac{|\mathcal R \cap \mathcal H_0(P)|}{|\mathcal R| \vee 1}. 
\end{equation}
``$ |\mathcal R| \vee 1$" is the maximum between $|\mathcal R|$ and 1 when no hypothesis is rejected ($|\mathcal R| = 0$).

A well-known procedure for controlling FDR is Benjamini-Hochberg's (BH), \cite{benjamini1995controlling}. Denoting $(p_{(i)})_{1\le i \le m}$ the family of the ordered $p$-values, we can write the index set of rejected null hypotheses by the BH procedure as
\begin{equation} \mathcal R_{BH} = \Big \{i \in \{1,\hdots,m\}: p_{i} \le \alpha \hat k / m \Big \}, \end{equation}
where $\hat k =  \max \{ k \in \{0,\hdots,m\} : p_{(k)} \le \alpha k / m\}$ with the convention $p_{(0)}=0$.

We can prove that the limit of the FDR of the BH procedure is equal to the FDR of the limiting experience.
\begin{thm} \label{fdrlim} Suppose
\begin{equation} \label{hypoChat} \forall P \in \mathcal P, \, \sqrt n \left( \: \hat \theta_{n,\cdot} -\theta (P)\right)\overset{d}{\longrightarrow} P_{\infty}.
\end{equation}
Then for all $P \in \mathcal P$, 
\begin{equation}  \underset{n\to+\infty}{\lim} \text{ } \rm{FDR}(\textit{P}) =  \text{ } \rm{FDR}(\textit{P}_{\infty}). \nonumber
 \end{equation}
\end{thm}
In particular, if $P_{\infty}$ is Positive Regression Dependent on each one from a Subset (PRDS) on $\{1\le i \le m : c_i = 0 \}$, then \cite{benjamini1995controlling} states that the asymptotic control of FDR by BH procedure is acquired. In the one-sided Gaussian testing framework, the PRDS assumption is satisfied whenever $\Sigma_{ij} \geq 0$, for all $(i, j)$, see \cite{benjamini2001control}. However, this result is not helpful for testing  correlations in two-sided Gaussian setting. For instance, \cite{yekutieli2008false} gives some examples of non-PRDS two-sided Gaussian tests. We refer to \cite{roux}, Chapter 4, for an overview of results of the BH procedure in two-sided Gaussian setting. 

Resulting from a simulation study, we can state that indeed the matrix $\Omega(\Gamma)$ of Proposition \ref{prop:corrgauss} does not provide a positive structure dependence in general. There is no clear property to test if $p$-values are PRDS. We decided to examine a stronger but more easily verifiable assumption, that is, Multivariate Total Positivity of order 2 (MTP$_2$), see \cite{karlin1980classes}. The evaluation of the proportion of MTP$_2$ distributions is based on Theorem $3.1$ and Theorem $3.1^{\prime}$ of \cite{karlin1981total}, which give necessary and sufficient condition on the matrix $\Gamma$ (resp. $\Omega(\Gamma)$) for $|X|$ (resp. $(|\rho_h|)_{h \in \mathcal H}$) having an MTP$_2$ distribution. We generate $100,000$ covariance $d$-matrices $\Gamma$ using the R-package \textit{ClusterGeneration} of \cite{joe2006generating}, and we evaluate the occurrence of  MTP$_2$ distributions on these $100,000$ matrices. Results are displayed in Table~\ref{mtp2occ}. The proportion of the vector of the absolute value of correlation coefficients of $X$, that is, $(|\rho_h|)_{h \in \mathcal H}$, having an MTP$_2$ distribution is evaluated in two different cases: when $|X|$ does not have an MTP$_2$ distribution (third column of Table~\ref{mtp2occ}) and when $|X|$ has an MTP$_2$ distribution (fourth column of Table~\ref{mtp2occ}).

\begin{table}[!h]
\begin{tabular}{lccc}
\toprule
Distribution of: &   $|X|$ & \multicolumn{2}{c}{$(|\rho_h|)_{h \in \mathcal H}$}  \\
&&without MTP$_2$ constraint on $|X|$ &with MTP$_2$ constraint on $|X|$\\ 
\midrule
$d=3$ &   5100 &  10223 & 0 \\ 
$d=4$ &    82 &  38 &  0 \\ 
$d=5$ &   0 &  0 &  $\cdot$ \\  
\bottomrule
\end{tabular}
\caption{ \label{mtp2occ} Occurrence of MTP$_2$ distributions on 100,000 simulations. The dot means that the quantity is not computable.}
\end{table}
  As shown in Table~\ref{mtp2occ}, the MTP$_2$ constraint seems to be a very restrictive constraint (see the second column of Table~\ref{mtp2occ}). Moreover, even if the absolute value of the observed vector $X$ has an MTP2$_2$ distribution, the transformation $\Omega$ does not preserve the positive dependence structure (see the last column of Table~\ref{mtp2occ}) for the correlation coefficients.

In \cite{romano2008control}, the authors proposed both subsampling and bootstrap procedures that control the false discovery rate (FDR) under dependence, especially for the two-sided testing problem~\eqref{eqn:tests}. Unfortunately, their methods require that all the $p$-values under the alternative are equal to zero. Finally, \cite{cai2016large} developed a truncated bootstrap alternative to BH for correlation tests. The idea is to estimate the threshold index $\hat k$ in \eqref{eq:BH} by bootstrap evaluations of the FDR. They add a truncation based on the asymptotic Gaussian property of tests statistics. Theoretical control is obtained, under sparsity assumptions.

\subsection*{Conclusion}

This work was motivated by a real data application in neuroscience. Our aim is to identify the connectivity graphs of brain areas by testing if the correlation between recordings of pairs of brain regions is significant. Our statistical modeling includes asymptotically Gaussian statistics and multiple corrections because many simultaneous tests are applied (here, 1275 in the application). 
Hence in this work we have studied multiple testing when statistics are asymptotically Gaussian, and possibly correlated. Four multiple testing procedures are presented: Bonferroni, \Sidak, Romano-Wolf non parametric bootstrap (\BootRW) and a parametric bootstrap (\MaxT). For each of these methods, we verify that the FWER is asymptotically controlled. Besides, special attention is given to the case of correlation testing. A simulation study then highlights problems that may be encountered in applications. First for many statistics, a sufficient sample size is necessary to ensure the control of FWER. Next the parametric bootstrap suffers from numerical instability. Interestingly, the sparsity of the false hypothesis does not influence the quality of the procedures. Finally, we apply a multiple testing correction on a real dataset, inferring the connectivity graphs of small animals from resting state fMRI recordings. The recordings for dead animals enable us to verify that the obtained results are coherent by finding a nearly empty graph. In contrast alive animals connectivity graphs are not empty. This confirms the utility of our approach for real data applications.

\bibliographystyle{apalike}%plainnat}%agsm}
\bibliography{biblio_fwer}

\appendix

\section{Proofs}\label{app}

\subsection{Proof of~Proposition~\ref{prop:fwerbonf}}

For all $P \in \mathcal P$,
\begin{equation*}\textrm{FWER}\Big(\mathcal R^{bonf}_{\alpha},P\Big) =  \mathbb P \left (\bigcup_{ i \in \mathcal H_0(P)}  \{ p_{n,i} \le \alpha/m \} \right).
\end{equation*}

Thus,
\[
\underset{n\to+\infty}{\lim} \textrm{FWER} \Big(\mathcal R^{bonf}_{\alpha},P\Big)  \le \displaystyle{\sum_{i \in \mathcal H_0(P)}} \frac{\alpha}{m} \le  \frac{\alpha m_0(P)}{m} \le \alpha. 
\]

\subsubsection{Proof of~Proposition~\ref{prop:fwersidak}}

In order to prove Proposition~\ref{prop:fwersidak}, we first need the following result.

\begin{prop} \label{propport}
For all $x>0$, for all $P \in \mathcal P$,
\begin{equation} 
\label{port} 
\underset{n \to + \infty}{\lim} \left \{\P \left ( \sqrt{n}\,  \left | \: \widehat \theta_{n,i}\left (\mathbb X^{(n)} \right) - \theta_i(P) \right| \,\le\, x, \text{ for all } i \in \mathcal H_0(P) \right) \right \} \ge (2\Phi(x)-1)^{m_0(P)}. 
\end{equation}
\end{prop}  

\begin{proof}
Let $Z_i= \lim_{n \to + \infty} \sqrt{n}\,  \left | \: \widehat \theta_{n,i}\left (\mathbb X^{(n)} \right) - \theta_i(P) \right|$. Then $Z=(Z_i)_{1 \leq i \leq m_0(P)}$ is $m_0(P)$-dimensional random vector with a multivariate Gaussian distribution  $ \mathcal N_{m_0(P)}\Big(0,(\Sigma)_{i,i^{\prime} \in \mathcal H_0(P)}\Big) $. Let $B = \Big\{z \in \mathbb R^{m_0(P)}: \|z\|_{\infty} \le x  \Big\}$. Let $\partial B$ denote the frontier of the set $B$. For all $x>0$,
\begin{align*}
\mathbb P(Z \in \partial B ) &= \mathbb P(||Z||_\infty = x)\\
&= \mathbb P \bigg(\sup_{1\le j \le m_0(P)} |Z_j| = x \bigg) \\
&\le \mathbb P \Big(\exists j \in \{1,\hdots,m_0(P)\} : |Z_j| =x\Big)\\
&\le \sum_{j=1}^{m_0(P)} \mathbb P(|Z_j|=x)\\
&\le 0,
\end{align*}
By Portmanteau's lemma \citep{billingsley2013convergence}, it follows that for all $P \in \mathcal P$,
\begin{equation*}
\underset{n\to+\infty}{\lim}\P \left( \sup_{i\in\mathcal H_0(P)} \sqrt{n} \lvert \hat \theta_{n,\cdot}  - \theta(P) \rvert \le x \right) =\P(\|Z\|_{\infty} \le x),
\end{equation*}
%with $Z$ random vector $\mathcal N_{m_0(P)}(0,\Omega)$-distributed and $\Omega=\left(\Sigma_{i,i'}\right)_{i,i'\in\mathcal H_0(P)}$.
Applying \eqref{Sid}, for all $P \in \mathcal P$,
\begin{equation*} 
\underset{n\to+\infty}{\lim}\P \left(\sup_{i\in\mathcal H_0(P)} \sqrt{n} \lvert \hat \theta_{n,\cdot}  - \theta(P) \rvert\le x \right) \ge  \Big[\P( | Z_1 | \le x ) \Big]^{m_0(P)}, 
\end{equation*}  
which concludes the proof.  
\end{proof}

We now are in position to prove Proposition~\ref{prop:fwersidak}.

For all invertible $\Sigma$, for all $P \in \mathcal P$,
\begin{align*}
\mathrm{ FWER}(\mathcal R^s_\alpha, P) &= \P\Big( \exists \: i \in\mH_0(P) : |T_{n,i}| > c_{\alpha}^s \Big) \\
&= 1 - \P\Big( \forall \: i \in\mH_0(P), |T_{n,i}| \le c_{\alpha}^s \Big) \\
&= 1 - \P\Big(\lVert(T_{n,i})_{i \in\mH_0(P)}\rVert_{\infty} \le c_{\alpha}^s \Big)
\end{align*}
Using Proposition~\ref{propport},
\begin{equation*} 
\underset{n\to+\infty}{\lim}  \mathrm{ FWER}(\mathcal R^s_\alpha, P) \le 1 - (2\Phi(c_{\alpha}^s )-1)^{m_0(P)} \le \alpha.
\end{equation*}

\subsection{Proof of~Proposition~\ref{prop:rwcontrol}}

Under $P$, the matrix $\Sigma$ is assumed to be invertible. Therefore $\left(T_{n,i}\left ( \mathbb X^{(n)} \right) \right)_{ i \in \mathcal H_0(P)}$ has a non-degenerate asymptotic distribution.
The result then follows from Theorem~7 of~\cite{romano2005exact}.

\subsubsection{Proof of~Proposition~\ref{prop:maxTcontrol}}

Let $P \in \mathcal P$. Denote by $\restriction{ \mathcal N_m(0,\Sigma)}{\mathcal H_0(P)}$ the restriction of the Gaussian distribution $ \mathcal N_m(0,\Sigma)$ on $\mathcal H_0(P)$, namely $ \mathcal N_{m_0(P)}\Big(0,(\Sigma)_{i,i^{\prime} \in \mathcal H_0(P)}\Big) $.

First, since the functions $(T_{n,i})_{i \in\mH_0 (P)}  \xrightarrow[n\to +\infty]{\quad d \quad} \restriction{ \mathcal N_m(0,\Sigma)}{\mathcal H_0(P)}$
 and $x \mapsto \|x\|_{\infty} $ for $x\in \mathbb R^{|\mathcal H_0(P)|}$ are continuous, then by the continuous mapping theorem, we have that
\begin{equation} \label{etp1} 
\|(T_{n,i})_{i \in\mH_0(P)})\|_{\infty}  \xrightarrow[n\to +\infty]{\quad d \quad} \|\restriction{ \mathcal N_m(0,\Sigma)}{\mathcal H_0(P)}\|_{\infty} .
\end{equation}
 
Second, let us establish that  
\begin{equation} \label{cvquantile} 
t_{n,\alpha}(\hat \Sigma_n)  \overset{\mathbb P}{\longrightarrow}  t_{\alpha}( \Sigma).   
\end{equation}
Let $x\in\R$. For all $n\in\N$, we introduce the cumulative distribution functions $\varphi_x(\hat\Sigma_n)$ and $\varphi_x(\Sigma)$, defined by   
\begin{align*} 
\varphi_x(\hat \Sigma_n) &=\P(\|\hat \Sigma_n^{1/2}\xi \|_{\infty} \le x\mid \hat \Sigma_n ),\\
\varphi_x(\Sigma) &=\P(\|\Sigma^{1/2}\xi \|_{\infty} \le x),
\end{align*}
where $\xi \sim \mathcal N_m(0,I_m)$.

The Cram\'er-Wold device (\cite{gut2012probability}, Theorem 10.5) involves $\hat \Sigma_n^{1/2} \xi \xrightarrow[n\to +\infty]{d} \Sigma^{1/2}\xi$ and thus, by the continuous mapping theorem, $\|\hat \Sigma_n^{1/2} \xi \|_{\infty}\xrightarrow[n\to +\infty]{d} \|\Sigma^{1/2}\xi \|_{\infty}$. In addition, applying Portmanteau's lemma, we obtain that $\varphi_x(\hat\Sigma_n)$ converges in probability to $\varphi_x(\Sigma)$ when $n$ goes to infinity.
Therefore, for all $\epsilon > 0$,
{\begin{equation}\label{cveps} 
\varphi_{t_{\alpha}(\Sigma)-\epsilon}(\hat \Sigma_n) \overset{\P}{\longrightarrow}  \varphi_{t_{\alpha}(\Sigma)-\epsilon}(\Sigma). 
\end{equation}}

Let $\epsilon >0$. One has
\begin{align*}
\lefteqn{\mathbb P \Big(t_{n,\alpha}(\hat \Sigma_n) \le t_{\alpha}(\Sigma)-\epsilon \Big) }\\
&=\P \Big(\varphi_{t_{\alpha}(\Sigma)-\epsilon}(\hat \Sigma_n) \ge 1-\alpha \Big) \\
&=\P \Big(\varphi_{t_{\alpha}(\Sigma)-\epsilon}(\hat \Sigma_n) - \varphi_{t_{\alpha}(\Sigma)-\epsilon}( \Sigma) \ge 1-\alpha -\varphi_{t_{\alpha}(\Sigma)-\epsilon}( \Sigma) \Big).
\end{align*}
Write $ \varphi_{t_{\alpha}(\Sigma)-\epsilon}(\Sigma)=1-\alpha-\beta$, with $\beta >0$. The right-hand side is bounded by
\begin{align*}
\P \Big (\varphi_{t_{n,\alpha}(\Sigma)-\epsilon}(\hat \Sigma_n) - \varphi_{t_{n,\alpha}(\Sigma)-\epsilon}( \Sigma) \ge \beta \Big) &\le \P\Big( \big |\varphi_{t_{n,\alpha}(\Sigma)-\epsilon}(\hat \Sigma_n) - \varphi_{t_{n,\alpha}(\Sigma)-\epsilon}( \Sigma) \big | \ge \beta \Big).
\end{align*}
Then, convergence~\eqref{cvquantile} follows from~\eqref{cveps}.

We now finally establish the asymptotic FWER control for the procedure $\mathcal R^{\MaxT}$. By definition,
\begin{equation} 
\text{FWER}(\mathcal R^{\MaxT}_\alpha, P) =\P \Big(\|(T_{n,i})_{i \in\mH_0(P)} \|_{\infty} \ge t_{n,\alpha}(\hat \Sigma_n) \Big). 
\nonumber 
\end{equation} 
For all $\delta > 0$, this quantity is bounded by 
\begin{equation}
\P \Big (\|(T_{n,i})_{i \in\mH_0(P)} \|_{\infty} > t_{\alpha}(\Sigma) - \delta \Big) +\P \Big(  t_{\alpha}(\hat \Sigma_n) \le t_{n,\alpha}(\Sigma) - \delta \Big) . 
\nonumber 
\end{equation}
Then, using~\eqref{etp1} and~\eqref{cvquantile}, we have
\begin{equation} 
\underset{n \to + \infty}{\limsup} \text{ FWER}(\mathcal R^{\MaxT}_\alpha, P) \le\P \Big (\| \restriction{\mathcal N_m(0,\Sigma)}{\mathcal H_0(P)}\|_{\infty} > t_{\alpha}(\Sigma) - \delta \Big). \nonumber
\end{equation}
Letting $\delta$ tend to 0 we obtain the asymptotic control.

\subsection{Proof of Proposition~\ref{lapropo}}

Let $P \in \mathcal P$. By \eqref{controlSSas}, there exists an event $\mathcal E$, such that  $\mathcal E \cap \mathcal H_0(P) \subseteq \mathcal E\setminus {\mathcal R_{\mathcal H_0(P)}}$ and
\begin{equation} \label{conv} \P(\mathcal E) \xrightarrow[n\to +\infty]{} 1 - \alpha. 
\end{equation}
For all $j \ge 0$, $\mathcal H_0(P) \subseteq \mathcal C_j$ implies that $\mathcal R_{\mathcal H_0(P)} \supseteq \mathcal R_{\mathcal C_j}$ by \eqref{incas}. Thus we have $\mathcal E \setminus\mathcal R_{\mathcal H_0(P)} \subseteq \mathcal E\setminus \mathcal R_{\mathcal C_j}$. Consequently, $\mathcal E \cap {\mathcal H_0(P)} \subseteq \mathcal C_{j+1}$. Since $\mathcal H_0(P) \subseteq \mathcal C_0 = \{1,\hdots,m\}$, a recurrence gives $\mathcal E \cap\mathcal H_0(P) \subseteq \mathcal E \cap\mathcal C_j$ for all $j\ge0$. It results that $\mathcal E\cap\mathcal H_0(P) \subseteq \mathcal E\setminus \mathcal R_{{\mathcal C}_\infty}$.
Finally, using convergence \eqref{conv},
\begin{equation*} 
\P(\mathcal H_0(P) \subseteq \{1,\dots,m\}\setminus \mathcal R_{{\mathcal C}_\infty}) \xrightarrow[n\to +\infty]{} 1 - \alpha. 
\end{equation*}
Since the left-hand side is equal to $1-\text{FWER}\Big(R_{{\mathcal C}_\infty}, P\Big)$, Proposition \ref{lapropo} is proved.

\subsection{Proof of Theorem~\ref{fdrlim}}

Let $f : \mathbb R \rightarrow [0,1], x \mapsto 1 - \Phi(x + c)$ and $Y^{(n)} = \sqrt n \Big(\hat \theta_{n,\cdot} - \theta(P) \Big)$. For all $k \in \{1, \hdots,m\}$, for all $i \in \mathcal H_0(P)$, define the set $\mathcal B_{k,i}$ as
\[ \mathcal B_{k,i} = \Big\{ y \in \mathbb R^m:  \hat k \le k-1 \,\mathrm{ and }\,  f(y_i) \le \alpha (k-1) / m\Big \} 
\]
For all $y\in\R^m$, we can sort the m-dimensional vector $(f(y_1), f(y_2),\dots, f(y_m))$ as $f(y)_{(1)}\leq f(y)_{(2\leq\dots\leq f(y)_{(m)})}$. The set $\mathcal B_{k,i}$ can be rewritten as  
\[ \mathcal B_{k,i} =  \Big \{ y \in \mathbb R^m : f(y)_{(j)} > \alpha j /m \text{ for } j=k,\hdots, m,  \,\mathrm{ and }\, f(y_i) \le \alpha (k-1) / m \Big\}.
\]

Denote  $\eta$ the Lebesgue's measure and $\partial \mathcal B_{k,i}$ the frontier of the set $ \mathcal B_{k,i}$. We want to prove that for all $k \in \{1,\hdots,m\}$, for all $i \in \mathcal H_0(P)$, $\eta(\partial \mathcal B_{k,i}) = 0$. For all $k \in \{1, \hdots,m\}$, for all $i \in \mathcal H_0(P)$, we have 
\begin{align*} \eta(\partial \mathcal B_{k,i}) &= \eta \bigg (\partial \bigg [ \bigg ( \bigcap_{j=k}^m \Big \{ y  :  f(y)_{(j)} > \alpha j /m \Big\} \bigg ) \bigcap \Big \{ y :  f(y_i) \le \alpha (k-1) / m\Big\} \bigg] \bigg) \\
&\le \eta \bigg (\bigg ( \bigcup_{j=k}^m \partial\Big \{ y :f(y)_{(j)} > \alpha j /m \Big\}  \bigg) \bigcup \partial \Big \{ y :   f(y_i) \le \alpha (k-1) / m\Big\} \bigg ) \\
&\le \sum_{j=k}^m \eta \big (\{ y :  f(y)_{(j)} = \alpha j /m \} \big) + \eta \big ( \{ y  :  f(y_i) = \alpha (k-1) / m\} \big )  \\
& \le \sum_{l=1}^m  \sum_{j=k}^m \eta \big(\{ y  : f(y_l) = \alpha j /m \} \big) +\eta \big(\{ y : f(y_i) = \alpha (k-1) / m\} \big).
\end{align*}
Since the right-hand side is null, equality $\eta(\partial \mathcal B_{k,i})=0$ holds. Consequently for all $k \in \{1, \hdots,m\}$, for all $i \in \mathcal H_0(P)$, $\mathbb P (Y \in \partial \mathcal B_{k,i}) = 0$. Using Portmanteau's lemma, it follows that
\begin{equation*} \underset{n\to+\infty}{\lim}  \mathbb P \Big (Y^{(n)} \in \mathcal B_{k,i} \Big) = \mathbb P (Y \in  \mathcal B_{k,i} ). \end{equation*} 
 
If $\mathcal B_{k,i}^{\prime} = \Big\{ y \in \mathbb R^m : \hat k \le k-1 \,\mathrm{ and }\,  f(y_i) \le \alpha k / m\Big \}$, with a similar reasoning we can prove that  
\begin{equation*}\underset{n\to+\infty}{\lim}  \mathbb P \Big (Y^{(n)} \in \mathcal B_{k,i}^{\prime} \Big) = \mathbb P (Y \in  \mathcal B_{k,i}^{\prime} ). 
\end{equation*}
Therefore, we obtain
\begin{align*}
\lefteqn{\text{ FDR}(P)}\\
 &=  \sum_{i \in \mathcal H_0(P)} \sum_{k=1}^m \frac{1}{k} \Bigg [ \mathbb P \bigg (\hat k \le k, f \Big(Y^{(n)}_i \Big) \le \alpha k/m \bigg) - \mathbb P \bigg (\hat k \le k-1, f \Big(Y^{(n)}_i \Big) \le \alpha k/m \bigg) \Bigg] \\
&= \sum_{i \in \mathcal H_0(P)} \sum_{k=1}^m \frac{1}{k} \Bigg[ \mathbb P \Big (Y^{(n)} \in \mathcal B_{k+1,i} \Big) -\mathbb P \Big (Y^{(n)} \in \mathcal B_{k,i}^{\prime} \Big) \Bigg ]. 
\end{align*}
Taking the limit in the previous expression leads to
\begin{align}
\underset{n\to+\infty}{\lim}  \text{ FDR}(P) &= \sum_{i \in \mathcal H_0(P)} \sum_{k=1}^m \frac{1}{k} \Bigg[\mathbb P \Big (Y \in \mathcal B_{k+1,i} \Big) -\mathbb P \Big (Y \in \mathcal B_{k,i}^{\prime} \Big) \Bigg ]  \nonumber \\
&=   \text{FDR}(P_{\infty}) \nonumber .
\end{align}

\section*{Acknowledgments}
This work was partly supported by the project \textit{Graphsip} from Agence Nationale de la Recherche (ANR-14-CE27-0001).
This work was initiated in Marine Roux's PhD, \cite{roux}, and the authors are grateful to Marine Roux for the fruitful exchange. The authors would like to thank Etienne Roquain for very constructive scientific discussions. We also would like to thank Emmanuel Barbier and Guillaume Becq for providing us the data of the resting state fMRI on the rats.

\end{document}